\documentclass[10pt,reqno]{amsart}
\usepackage{amsmath}
\usepackage{amssymb, latexsym, amsfonts, amscd, amsthm, mathrsfs, enumerate, esint}
\usepackage[usenames,dvipsnames]{color}
\usepackage{amssymb}
\usepackage{comment}

\numberwithin{equation}{section}

\definecolor{purple}{rgb}{0.9,0,0.8}

\definecolor{gray}{rgb}{0.7,0.7,0.7}

\newcommand{\abbr}[1]{{\sc\lowercase{#1}}}

\newcommand{\req}[1]{(\ref{#1})}

\def\cpt{\mathsf{CP}}
\def\err{{\sf Err}}
\def\sM{{\sf M}}

\def\cro{{\rm cr }}
\def\fdt{{\rm fdt}}
\def\ep{{\epsilon}}
\def\b{{\beta}}
\def\SN{{\Bbb S}_N}
\def\SNq{q\SN}
\def\qs{q_\star}
\def\SNqs{\qs \SN}

\def\N{{\Bbb N}}
\def\bJ{{\bf J}}
\def\R{{\Bbb R}}
\def\P{{\Bbb P}}
\def\E{{\Bbb E}}

\def\BO{{\bf O}}
\def\BB{{\bf B}}
\def\BZ{{\bf Z}}

\def\BB{{\bf B}}
\def\BZ{{\bf Z}}

\def\BJ{{\bf J}}

\def\BR{{\bf R}}
\def\BD{{\bf D}}
\def\BJ{{\bf J}}

\def\bmu{{\boldsymbol \mu}}
\def\bsigma{{\boldsymbol \sigma}}
\def\bphi{{\boldsymbol \varphi}}

\def\bq{{\bf q}}

\def\Bx{{\bf x}}
\def\bx{{\bf x}}

\def\by{{\bf y}}
\def\By{{\bf y}}
\newcommand{\reals}{{\Bbb{R}}}

\def\IJ{{\bf I}}

\def\bx{\Bx}
\def\bn{\Bx_\star}

\def\Es{{E_\star}}
\def\tEs{{\widehat E_\star}}
\def\Gs{{G_\star}}
\def\mus{\mu_\star}
\def\qinf{\alpha \qs}
\def\wqinf{\widetilde \alpha \qs}

\def\CC{{\mathcal C}}
\def\CL{{\mathcal L}}

\def\CC{{\mathcal R}}
\def\CC{{\mathcal C}}
\def\CL{{\mathcal L}}

\def\Aa{{\mathcal A}}
\def\Ba{{\mathcal B}}
\def\Ca{{\mathcal C}}

\def\Ea{{\mathcal E}}
\def\Fa{{\mathcal F}}

\def\Sa{{\mathcal S}}

\def\Ua{{\mathcal U}}

\def\ep{{\epsilon}}
\def\b{{\beta}}

\def\N{{\Bbb N}}
\def\bJ{{\bf J}}
\def\R{{\Bbb R}}
\def\P{{\Bbb P}}
\def\E{{\Bbb E}}

\def\BB{{\bf B}}
\def\BZ{{\bf Z}}

\def\BB{{\bf B}}
\def\BZ{{\bf Z}}

\def\BJ{{\bf J}}

\def\BR{{\bf R}}
\def\BD{{\bf D}}
\def\BJ{{\bf J}}

\def\bmu{{\boldsymbol \mu}}

\def\bq{{\bf q}}

\def\Bx{{\bf x}}
\def\bx{{\bf x}}

\def\bv{{\mathsf v}}

\def\by{{\bf y}}
\def\By{{\bf y}}
\def\wQ{{V}}
\def\wH{\widehat{H}}

\def\bx{\Bx}

\def\CC{{\mathcal C}}
\def\CL{{\mathcal L}}

\def\CC{{\mathcal R}}
\def\CC{{\mathcal C}}
\def\CL{{\mathcal L}}

\def\Aa{{\mathcal A}}
\def\Ba{{\mathcal B}}
\def\Ca{{\mathcal C}}

\def\Ea{{\mathcal E}}
\def\Fa{{\mathcal F}}

\def\Sa{{\mathcal S}}

\def\Ua{{\mathcal U}}
\def\Uas{{\mathcal U}^{\star}}
\def\Uao{{\mathcal U}^{\dagger}}

\def\C{{\Bbb C}}
\def\R{{\Bbb R}}

\def\N{{\Bbb N}}
\def\E{{\Bbb E}}
\def\CB{{\Ba}}

\def\Var{{\rm Var}}
\def\Cov{{\rm Cov}}
\def\half{\frac{1}{2}}
\def\oneN{\frac{1}{N}}

\def\nn{\noindent}

\def\hh{\widehat}

\def\b{\beta}
\def\d{\delta}
\def\e{\epsilon}

\def\s{\sigma}

\def\NC{{\mbox{NC}}}
\def\D{\Delta}

\def\L{\Lambda}

\def\part{\partial}

\def\ts{\times}

\def\ra{\rightarrow}

\def\tilde{\widetilde}
\def\hh{\widehat}

\cleardoublepage
\newtheorem{prop}{Proposition}[section]

\newtheorem{remark}[prop]{Remark}
\newtheorem{lem}[prop]{Lemma}

\newtheorem{cor}[prop]{Corollary}
\newtheorem{theo}[prop]{Theorem}

\def\C{{\Bbb C}}
\def\R{{\Bbb R}}

\def\N{{\Bbb N}}
\def\E{{\Bbb E}}

\def\CB{{\Ba}}

\def\Var{{\rm Var}}
\def\half{\frac{1}{2}}
\def\oneN{\frac{1}{N}}

\def\nn{\noindent}

\def\hh{\widehat}

\def\b{\beta}
\def\d{\delta}
\def\e{\epsilon}

\def\s{\sigma}

\def\NC{{\mbox{NC}}}
\def\D{\Delta}

\def\L{\Lambda}

\def\part{\partial}

\def\ts{\times}

\def\ra{\rightarrow}

\def\tilde{\widetilde}
\def\hh{\widehat}

\newcommand{\wP}{\widetilde{\P}}
\newcommand{\wE}{\widetilde{\E}}
\newcommand{\Ps}{\P_\star}
\newcommand{\Est}{\E_\star}

\cleardoublepage

\begin{document}              

\title[Spherical spin glass dynamics]
{Dynamics for spherical spin glasses: \\
Disorder dependent initial conditions}
\author{Amir Dembo}
\address{Department of Statistics and Department of Mathematics\\
Stanford University\\ Stanford, CA 94305.}
\email{adembo@stanford.edu}

\author{Eliran Subag}
\address{Courant Institute, New York University\\ 
New York, NY 10012.}
\email{esubag@cims.nyu.edu}

\thanks{\noindent Research partially supported by BSF grant 2014019 (A.D. \& E.S.), 
NSF grants \#DMS-1613091,  \#DMS-1954337 (A.D), and the Simons Foundation (E.S.).
\newline
{\bf AMS (2100) Subject Classification:}
Primary: 82C44 Secondary:  82C31, 60H10, 60F15, 60K35.
}

\date{May 22, 2020}

\begin{abstract}
We derive the thermodynamic limit of the 
empirical correlation and response functions 
in the Langevin dynamics for spherical mixed $p$-spin 
disordered mean-field models, starting uniformly 
within one of the spherical bands on which the 
Gibbs measure concentrates at low temperature for the pure $p$-spin models and mixed perturbations of them. 
We further relate
the large time asymptotics of the resulting 
coupled non-linear integro-differential equations, to 
the geometric structure 
of the Gibbs measures (at low temperature),
and derive their \abbr{FDT} solution (at high temperature). 
\end{abstract}

\maketitle
\section{Introduction\label{s.introduction}}

The thermodynamic limits of a wide class of Markovian dynamics with random interactions,
exhibit complex long time behavior, which is of much interest in 
out of equilibrium statistical physics 
(c.f. the surveys 
\cite{BB,BKM,LesHouches} and the references therein). 
This work is about the thermodynamic ($N \to \infty$), 
long time ($t \to \infty$), behavior of a certain class of
systems composed of $N$ Langevin particles
$\bx_t=(x_t^i)_{1\le i\le N}\in\R^N$, 
interacting with each other through a random
potential. More precisely, one considers a diffusion of the form
\begin{equation}\label{diffusion}
 d\bx_t=-f'(\vert\vert \bx_t\vert\vert^2/N)\bx_t dt
  -\beta\nabla H_\BJ(\bx_t) dt +d\BB_t,
  \end{equation}
where $\BB_t$ is an $N$-dimensional Brownian motion, $\vert \vert \bx \vert
\vert$ denotes the Euclidean norm of $\bx \in \R^N$ and differentiable 
fast growing functions $f=f_L$ such that $e^{-f_L(r)}$ approximates
as $L \to \infty$ the indicator on $r=1$, effectively 
restricting $\bx_t$ to the sphere 
$\SN := {\Bbb S}^{N-1}({\sqrt{N}})$  
of radius  $\sqrt{N}$. 
In particular, the spherical, mixed $p$-spin model (with $p \leq m$), 
has a centered Gaussian potential $H_\BJ:\ \R^N\longrightarrow \R$ 
of non-negative definite covariance structure
\begin{equation}\label{eq:nudef}
\Cov\big(H_{\BJ}(\bx),H_{\BJ}(\by)\big) 
= N \nu \big( N^{-1} \langle \bx,\by \rangle \big)\,, \qquad
\nu(r):=\sum_{p=2}^m b_p^2 r^p 
\end{equation}
(see Remark \ref{rem:m-infty} on a possible extension to $m=\infty$).
Hereafter we shall realize this potential as
\begin{equation}\label{potential}
  H_\BJ(\bx)=\sum_{p=2}^m b_p \sum_ {1\le i_1\le \cdots\le
i_p\le N}J_{i_1\cdots i_p}x^{i_1}\cdots x^{i_p},\quad b_m\neq 0
\end{equation}
for independent centered Gaussian coupling constants 
$\BJ=\{J_{i_1\cdots i_p}\}$, such that 
\begin{equation}\label{eq:vardef}
\Var(J_{i_1\ldots i_p}) = N^{-p+1} \frac{p!}{\prod_k l_k!} \,,
\end{equation}
where $(l_1,l_2,\ldots)$ are
the multiplicities of the different elements of the set      
$\{i_1,\ldots,i_p\}$ (so having $i_1 \neq i_2 \cdots \neq i_p$ 
yields variance larger by a factor $p!$ from 
the variance in case $i_1=i_2=\cdots = i_p$).

Given a realization of the coupling constants, the dynamics of 
\eqref{diffusion} is invariant (and moreover, reversible), for 
the (random) Gibbs measure $\mu^N_{2\b,\BJ}$ on $\reals^N$, where
$\mu^N_{\b,\BJ}$ has the density 
\begin{equation}\label{Gibbs}
\frac{d \mu^N_{\b,\BJ}}{d\Bx} = Z_{\b,\BJ}^{-1}
e^{-\b H_{\BJ} (\Bx) - N f(N^{-1} \|\Bx\|^2)} 
\end{equation}
(with respect to Lebesgue measure).
The normalization factor $Z_{\b,\BJ}= \int
e^{-\b H_{\BJ} (\Bx) -N f(N^{-1} \|\Bx\|^2)} d\Bx$ 
is finite if
\begin{equation}\label{eq:fcond}
\inf_{r \geq 0} \{ f'(r) - A r^{2k-1}  \} > -\infty 
\end{equation}
for some $A>0$ and $k>m/4$. Similar random measures have been
extensively studied in mathematics and physics over
the last three decades (see e.g.
\cite{Chen,Talagrand}, for the rigorous analysis of the asymptotic of
$N^{-1} \log Z_{\b,\BJ}$ for the 
hard spherical constraint
of having $\vert\vert \bx\vert\vert^2=N$).

Large dimensional Langevin or Glauber dynamics often exhibit 
very different behavior at various time-scales (as functions 
of system size, c.f. \cite{ICM} and references therein). 
Following the physics literature (see
\cite{BKM,CHS,LesHouches,CK}), 
we study \eqref{diffusion} for the potential
$H_{\BJ}(\bx)$ of \eqref{potential} at the shortest possible 
time-scale, where $N \to \infty$ first, holding $t \in [0,T]$.
While it is too short to allow any escape from meta-stable states,
considering the hard spherical constraint, Cugliandolo-Kurchan
have nevertheless predicted a rich picture for the 
limiting dynamics when starting out of equilibrium, 
say at $\bx_0$ distributed uniformly over $\SN$.
Such limiting dynamics involve the coupled 
integro-differential equations relating the non-random 
limits $C(s,t)$ and 
\begin{equation}\label{eq:Rdef}
\chi(s,t)=\int_0^t R(s,u) du \,,
\end{equation}
of the {\it empirical covariance function}
\begin{equation}\label{empiricalcovariance}
C_N(s,t)=\frac{1}{N} \langle \bx_s, \bx_t \rangle =
\frac{1}{N} \sum_{i=1}^N x_s^i x_t^i,\qquad  s\ge t
\end{equation}
and the {\it integrated response function}
\begin{equation}\label{integrated}
 {\chi}_N(s,t)= \frac{1}{N} \langle \bx_s,\BB_t \rangle
 = \frac{1}{N}\sum_{i=1}^N x_s^i B_t^i\,,
\end{equation}
respectively. Specifically, it is predicted that for large $\b$ the 
asymptotic of $C(s,t)$ strongly depends on the way 
$t$ and $s$ tend to infinity, exhibiting 
{\it aging} behavior (where the older it gets, the longer the system takes to forget its current state,
see e.g. \cite{CK,Alice}). A detailed analysis of such aging properties is given in \cite{2001}
for the case of $m=2$ in \eqref{potential} 
(noting that $\{J_{ij}\}$ form the \abbr{GOE} random matrix, 
whose semi-circle limiting spectral measure determines the 
asymptotic of $C(s,t)$). For $m > 2$, assuming hereafter that 
$f'$ is locally Lipschitz, satisfying \eqref{eq:fcond} and such that 
for some $\kappa <\infty$, 
\begin{equation}\label{eq:fcondu}
\sup_{r \geq 0} |f'(r)| (1+r)^{-\kappa} < \infty \,,
\end{equation}
we have from \cite[proof of Proposition 2.1]{BDG2}
that for each $N$, any finite disorder $\BJ$ and initial condition $\Bx_0$, 
there exists a unique strong solution in $\Ca (\R^+,\R^N)$ of 
\eqref{diffusion} (for a.e. path $t \mapsto \BB_t$). For such $f$ the 
closed equations for $C$ and $R$ are rigorously derived in \cite{BDG2} 
when {\it $\Bx_0$ is independent of $\BJ$} and 
satisfies the concentration of measure property of \cite[Hypothesis 1.1]{BDG2}, 
provided in addition $N \mapsto \E [ C_N(0,0)^k ]$ is uniformly bounded for 
each fixed $k<\infty$, the limit
\begin{equation}\label{eq:x0cond}
\lim_{N \to \infty} \E C_N(0,0) = C(0,0) \,,
\end{equation}
exists and 
$\P(|C_N(0,0)-C(0,0)|>x)$ decay exponentially fast in $N$.
Building on it, \cite[Proposition 1.1]{DGM} proves that for 
integer $k>m/4$ and $\bphi=1$, in the limit $L \to \infty$, the resulting 
equations of \cite{BDG2} for 
\begin{align}
f_L(r) &:= L(r-1)^2 + \frac{\bphi}{4k} r^{2k} \,,
\label{eq:fdef}
\end{align}
coincide for the pure $m$-spin case $\nu(r)=\frac{1}{8} r^m$ with the \abbr{CKCHS}-equations, derived independently 
by Cugliandolo-Kurchan \cite{CK} (who consider instead $C(2\cdot,2\cdot)$ and $R(2\cdot,2\cdot)$), 
and by Crisanti-Horner-Sommers \cite{CHS}.

\smallskip
The \abbr{CKCHS}-equations are for the Langevin dynamics 
of $\bx_t$ on the sphere $\SN$,
reversible with respect to the pure spherical $m$-spin Gibbs 
measure $\tilde{\mu}^N_{2\b,\BJ}$ of
density $\tilde{Z}_{2 \b,\BJ}^{-1} e^{-2 \b H_{\BJ}(\bx)}$
with respect to the uniform measure on $\SN$.
According to the Thouless-Anderson-Palmer (\abbr{TAP}) approach \cite{TAP}, the local magnetizations of each pure state \cite{MPV,TalagrandPS} approximately minimize the mean-field \abbr{TAP} free energy. For the pure spherical $m$-spin models \cite{CS,KPV} and $\beta$ in the low temperature phase, the (stable) minimizers $\bsigma$ of the \abbr{TAP} free energy roughly have radius $\sqrt N\qs$ with $\qs^2=q_{\rm EA}$  the 
Edwards-Anderson parameter, i.e. the right-most point in the support of the Parisi measure. As 
the \abbr{TAP} free energy only depends on $\|\bsigma\|$, such $\bsigma$ also approximately 
minimize the energy 
\begin{equation}
\label{eq:minE}
H_{\BJ}(\bsigma) \approx \min_{\bsigma'\in \qs\SN
} \{ H_{\BJ}(\bsigma') \} \,.
\end{equation}
More generally, it was recently rigorously proved \cite{SubagFlandscape} that for all spherical mixed $p$-spin models and $\beta$ in the low temperature phase,  for any $\qs\in(0,1)$ such that $\qs^2$ belongs to the support of the Parisi measure, $\bsigma\in \qs \SN$ satisfies \eqref{eq:minE} if and only if the probability under the  Gibbs measure  
$\tilde{\mu}^N_{\b,\BJ}$ of sampling many (slowly diverging with $N$) i.i.d. points $\bsigma^i$ from the narrow
band
\[
\Big\{ 
\bsigma'\in \SN:\, \frac{1}{N}\langle \bsigma'-\bsigma,\bsigma\rangle\approx 0
\Big\}
\]
such that 
$\frac1N\langle \bsigma^i-\bsigma,\bsigma^j-\bsigma\rangle\approx 0$ for $i\neq j$ is not exponentially small. 
Moreover, any point $\bsigma$  in the ultrametric tree \cite{MPSTV,PanchenkoUlt}, and not only the barycenters of pure states,  satisfies \eqref{eq:minE} with $\qs = \|\bsigma\|/\sqrt N$.
In fact, even for models with Ising spins \cite{CPS,CPS2}, the above holds if one adds an appropriate deterministic correction depending on the empirical measure $N^{-1}\sum_{i\leq N}\delta_{\sigma_i}$ to the Hamiltonian in both sides of \eqref{eq:minE}. 

For the pure $m$-spin models \cite{SubagGibbs} and their \abbr{1rsb} mixed perturbations \cite{BSZ} with $\beta\gg1$ an explicit pure states decomposition was proved by an investigation of the local structure around critical points. In particular, it was shown there that the Gibbs measure $\tilde{\mu}^N_{\b,\BJ}$ of the complement of the bands of small macroscopic width around all critical points with energy within small macroscopic distance from the minimal energy is exponentially small in $N$. Hence, in steady state the path $\bx_t$ spends an exponentially small in $N$ proportion of the time outside of those bands,
hinting that they play the role of meta-stable states in the conjectured aging picture (see also \cite{BJ,GJ} for spectral gap estimates and what they reveal about 
the Langevin dynamical phase transition parameter).
If the initial distribution is independent of the disorder $\BJ$, one may
expect an exponentially in $N$ long time to reach bands around deep critical 
points and a plausible aging mechanism is having the path $\bx_t$ decompose 
to time intervals spent in bands around deeper and deeper critical points,
connected by excursions of much shorter length, having typically
$\bx_t$ within the deepest band it has yet reached by time $t \gg 1$.
With initial distribution independent of the disorder $\BJ$, the \abbr{CKCHS}-equations discussed above concern (fixed) times not long enough (exponential in $N$) to be relevant to such  meta-stability induced aging. However,
to investigate the short-time dynamics as $\bx_t$ enters meta-stable states (of different levels) it is
natural to consider initial conditions that depend on 
$\BJ$. Specifically, having a random starting point at a fixed distance 
on the sphere from a critical point, which by itself is chosen randomly. Restricting to critical points at which $H_{\BJ}$ is near 
a fixed deep energy level $-\Es$ allows us to probe the
different `layers' of wells in the landscape as we vary $\Es$.

Provided that the number of such critical points is within a fixed factor off
its mean (currently proved only for pure $m$-spin \cite{Subag2nd} and small mixed perturbation of them \cite{BSZ}), 
the Kac-Rice formula (see \cite{AT}), allows us to translate  
the study of dynamics under such disorder dependent random initial distribution
to an investigation of dynamics driven by a modified, conditional Hamiltonian 
and deterministic initial distribution. To this end, 
our first result extends \cite[Theorem 1.2]{BDG2} 
to the latter initial measures and conditional potentials.\footnote{The conditioning on \eqref{eq:cond-J} is interpreted in the usual way: the conditional law of $\BJ$ has density given, up to normalization, by the restriction of its original density to the appropriate affine subspace, and the conditional law of the independent $\BB$ is identical to the unconditional one.} Specifically, fixing $\qs>0$ and 
$\bsigma\in \SNqs$ (around which we center the law of $\bx_0$), let
\begin{equation}\label{eq:bnq} 
H_N(s) = - \frac{1}{N} H_{\BJ}(x_s) \,, \qquad \qquad 
q_N^{\bsigma}(s) =\frac{1}{N}\langle \Bx_s,\bsigma\rangle = \frac{1}{N} 
\sum_{i=1}^N x_s^i \sigma^i \,.
\end{equation}
For $|q| \le \qs$ denote by $\mu_\bsigma^{q}$  
the uniform measure on the sub-sphere
\begin{equation}
\label{eq:subsphere}
{\Bbb S}_\bsigma (q) :=
\{\Bx \in \SN : 
\quad \frac{1}{N}\langle \Bx,\bsigma \rangle=q \}\,,
\end{equation} 
with $\P^{N,q}_{\BJ,\bsigma}$ denoting the joint
law (on $\Ca(\R^+,\R^{2N})$), of the Brownian motion $\BB$ and the corresponding  
strong solution $\bx$ of \eqref{diffusion} for
$\Bx_0$ of law $\mu_{\bsigma}^{q}$ and given $\BJ$, $\bsigma$
(see Proposition \ref{existN} for the existence of such a solution).
\begin{theo}\label{thm-macro}
For $\bsigma \in \qs \SN$, $\qs>0$,  
consider $\BJ$ conditional upon the event\footnote{\label{ft:purecdn}In the pure case, i.e. having 
$\nu(r)=b_m^2 r^m$, one has that 
$\partial_{\perp} H_{\BJ} (\bsigma) = \frac{m}{\|\bsigma\|} H_{\BJ}(\bsigma)$,
hence necessarily $\Gs=m \Es/\qs^{2}$, whereas
in the mixed case the vector $(\Es,\Gs)$ can take any value.}
\begin{equation}\label{eq:cond-J}
\cpt(\Es,\Gs,\bsigma) :=\{ H_{\BJ}(\bsigma)=-N \Es \,, 
\nabla_{\mathrm{sp}} H_{\BJ}(\bsigma)={\bf 0},\,\partial_{\perp} H_{\BJ}(\bsigma) =  
-  \|\bsigma \| \Gs \} \,,
\end{equation}
where $\nabla_{\mathrm{sp}}$ and $\partial_{\perp}$ denote, respectively,
the gradient \abbr{wrt} the standard differential structure on 
$\qs \SN$, and the 
directional derivative normal to $\qs \SN$.\footnote{\label{ft:grad-cond} Alternatively 
$\nabla H_{\BJ} (\bsigma) = - \Gs \bsigma$.}
Setting $q_o \in [-\qs,\qs]$ let $\bx_0$ 
be distributed according to $\mu_{\bsigma}^{q_o}$.
Then, for fixed $T<\infty$, as $N \to\infty$ the
random functions $(C_N,\chi_N,q_N^{\bsigma},H_N)$ 
converge uniformly on $[0,T]^2$, almost surely 
and in $L_p$ with respect to $(\Bx_0,\BJ,\BB)$,
to non-random functions $C(s,t)=C(t,s)$, 
$\chi(s,t)=\int_0^t R(s,u) du$, $q(s)$ and $H(s)$,
such that $q(0)=q_o$, $C(0,0)=1$, $R(s,t)=0$ for $t>s$, $R(s,s) \equiv 1$,
and for $s > t$ the absolutely continuous functions 
$C$, $R$, $q(s)$, $H(s)$ and $K(s)=C(s,s)$
are the unique solution in the space of bounded, continuous functions,
of the integro-differential equations
\begin{align}
\partial_s R(s,t) = &
- f'(K(s)) R(s,t) + \b^2 \int_t^s
R(u,t) R(s,u) \nu''(C(s,u)) du ,\label{eqR}\\
\partial_s C(s,t)= & 
 - f'(K(s)) C(s,t) + \b^2 \int_0^s R(s,u) \Big[ \nu''(C(s,u)) C(u,t) 
 -  \frac{q(t) \nu'(q(u)) \nu''(q(s))}{\nu'(\qs^2)} \Big] 
 du
\nonumber \\
& \qquad \qquad \qquad \quad + \b^2 \int_0^t R(t,u) \Big[\nu'(C(s,u)) - \frac{\nu'(q(s)) \nu'(q(u))}{\nu'(\qs^2)} \Big] \, du + \b q(t) \bv'_\star (q(s))  \,,
\label{eqC}\\
\partial_s q(s) = & -f'(K(s)) q(s) +  \b^2 \int_0^s R(s,u) 
\Big[ q(u) \nu''(C(s,u)) 
- \frac{\qs^2 \nu'(q(u)) \nu''(q(s))}{\nu'(\qs^2)} \Big] 
du 
+ 
\b \qs^2 \bv'_\star (q(s))  \,,
\label{eqq}\\
\partial_s K(s) = &  
1 - 2 f'(K(s)) K(s) 
+ 2 \b^2 \int_0^s R(s,u) \Big[ \psi(C(s,u)) 
 -  \frac{\psi(q(s)) \nu'(q(u))}{\nu'(\qs^2)} 
\Big] \, du + 2 \b q(s) \bv'_\star (q(s))  \,, \label{eqZ}\\
H(s)= & \wH(s)+ \bv_\star(q(s)),  \qquad 
\wH(s) =  \b \int_0^s R(s,u) \Big[ \nu'(C(s,u)) 
 -  \frac{\nu'(q(s)) \nu'(q(u))}{\nu'(\qs^2)} \Big] \, du  
 \,, \label{eqH}
\end{align}
where  $\psi(r):=
r \nu''(r) + \nu'(r)$ and  
\begin{align}
	\label{def:vt}
	\bv (r) := \sum_{p=2}^m \, b_p^2 \langle \bv_p, (E,G) \rangle \, r^{p} \,, \qquad 	
	\bv_p &:=  \begin{bmatrix}
	\qs^2 \nu(\qs^2)       & \qs^2 \nu'(\qs^2) \\
	\qs^2 \nu'(\qs^2)  & \psi(\qs^2)  
	\end{bmatrix}^{-1} \begin{bmatrix} \qs^2 \\ p \end{bmatrix} \,,
	\end{align}
using $\bv_\star(\cdot)$ to denote the case 
of $(E,G)=(\Es,\Gs)$.\footnote{\label{footnote4}It is easy to verify 
that in the mixed case the matrix in \eqref{def:vt} is positive definite for
any $\qs>0$, while in the pure case taking $G=m E/\qs^2$ yields 
$b_m^2 \langle \bv_m, (E,G) \rangle = \qs^{-2m} E$.} 
\end{theo}
\begin{remark}\label{rem:beta-1}
The conditional on $\cpt(E,G,\bsigma)$ solution of 
\eqref{diffusion} at $\beta>0$, is unchanged by 
embedding $\beta$ into the coefficients $\{b_p\}$ of \eqref{potential}
while taking $(E,G) \mapsto \beta (E,G)$ and setting $\beta=1$
in \eqref{diffusion}. This modifies 
$\nu \mapsto \beta^2 \nu$, while 
$\bv \mapsto \beta \bv$, preserving the stated 
limiting dynamics of Theorem \ref{thm-macro}, apart from multiplying 
$H(s)$ (and its derivatives) by $\beta$. It thus suffices to 
establish Theorem \ref{thm-macro} for $\beta=1$. 
\end{remark}
\begin{remark}\label{rem:rot-sym} From \eqref{eq:nudef} we see that 
for any non-random orthogonal matrix $\BO$, the covariance 
and hence the law of the Gaussian field 
$\bx \mapsto (H_{\BJ}(\BO^{-1} \bx), \BO \nabla H_{\BJ}(\BO^{-1} \bx))$ matches
that of $\bx \mapsto (H_{\BJ}(\bx),\nabla H_{\BJ} (\bx))$. 
When combined with  
$\bsigma \mapsto \BO \bsigma$ the same applies for the law of this field
conditional on $\cpt(\Es,\Gs,\bsigma)$. 
By the rotational symmetry of the Brownian motion $t \mapsto \BB_t$ 
and of the law $\mu^{q_o}_{\bsigma}$ of $\bx_0$, the law of 
$\{\bsigma,\bx_t,\BB_t,  t \in [0,T]\}$ in Theorem \ref{thm-macro}, 
matches that of $\{ \BO \bsigma, \BO \bx_t,\BO \BB_t, t \in [0,T]\}$.
In particular, the joint law of $(C_N,\chi_N,q_N^{\bsigma},H_N)$ is 
invariant under the mapping $\bsigma \mapsto \BO \bsigma$, and so  
it suffices to prove Theorem \ref{thm-macro} only for 
$\bsigma=\bn=(\sqrt{N}\qs,0,\ldots,0)$.
\end{remark} 
\begin{remark}\label{rem:cond-H}
Conditional on $\cpt(E,G,\bsigma)$, an easy Gaussian computation (see \eqref{eq:Gtilde} in case $\bsigma=\bn$), yields 
\begin{equation}\label{eq:Htilde}
H_{\BJ}(\bx) = H_{\BJ_o}(\bx)- N \bv \big(N^{-1} \langle \bx, \bsigma \rangle \big) \,,
\end{equation}
for the centered Gaussian vector $\BJ_o$ the corresponds to conditioning by $\cpt(0,0,\bsigma)$. Thus, 
$(E_\star,G_\star)$ only affects \eqref{diffusion} by adding a deterministic drift, which gives rise to the terms 
involving $\bv_\star(\cdot)$, or $\bv'_\star(\cdot)$, in \eqref{eqR}-\eqref{eqH}.  
The law of $\BJ_o$ is, for $N \gg 1$, well approximated 
by the Gaussian law of $\BJ$ conditional only on 
$\nabla_{\mathrm{sp}} H_{\BJ}(\bsigma)={\bf 0}$.
It is not hard to verify that the latter law has the covariance  
\begin{equation}\label{eq:mod-nu}
 N  \nu(N^{-1} \langle \bx,\by \rangle) - 
 \big[ \langle \bx,\by \rangle - 
\|\bsigma\|^{-2}  \langle \bx, \bsigma
  \rangle \langle \by,\bsigma \rangle \big]
 \frac{\nu'(N^{-1} \langle \bx, \bsigma \rangle)
 \nu'(N^{-1} \langle \by, \bsigma \rangle)}{\nu'(N^{-1} \langle \bsigma,\bsigma \rangle)}  
\end{equation}
(c.f. \eqref{eq:kval} for essentially such computation when $\bsigma=\bn$). This change from \eqref{eq:nudef} to \eqref{eq:mod-nu}
is behind the modification \abbr{wrt} the \abbr{ckchs} equations, in the square brackets within the integral terms of \eqref{eqC}-\eqref{eqH}.
\end{remark}

For $I,\,I'\subset\mathbb{R}$, denote by 
\begin{equation}\label{eq:CrtPts}
\mathscr{C}_{N,q}(I,I')=\left\{ \bsigma \in \SNq
:\,\nabla_{\mathrm{sp}} H_{\BJ}(\bsigma)={\bf 0},\,H_{\BJ}(\bsigma)\in -NI,\,\partial_{\perp} H_{\BJ}(\bsigma)\in -\sqrt NqI'\right\} 
\end{equation}
the set of critical points of the Hamiltonian $H_{\BJ}(\bsigma)$
on the sphere of radius $\sqrt N q$ with value in $-NI$ and with directional derivative normal to the sphere $\partial_{\perp} H_{\BJ}(\bsigma)$ in $-\sqrt NqI'$.
Our next result relates the dynamics of the unconditional model with 
random initial measure centered at such a critical point 
with the limiting dynamics of Theorem \ref{thm-macro}. Specifically, denoting by $\|U_N\|_\infty$ 
the supremum of $|U_N(s,t)|$ over $s,t \in [0,T]$, 
we associate to $\bsigma\in \SNqs$ around which we center  a `band', 
the (random) error 
\begin{align}
\label{eq:err} 
\err_{N,T} (\bsigma) := 
\|C_N-C\|_\infty \wedge 1 
+\|\chi_N-\chi\|_\infty \wedge 1 
+\|q_N^\bsigma-q\|_\infty \wedge 1
+\|H_N-H\|_\infty \wedge 1 
\end{align}
for the non-random functions $(C,R,q,H)$ from 
Theorem \ref{thm-macro}, which depend only on $\Es$, $\Gs$, $\qs$, $q_o$
and the model parameters $f(\cdot)$, $\beta$ and $\nu(\cdot)$. 
\begin{theo}\label{thm-uncond}
Let $\Es,\,\Gs,\,T>0$ and suppose $I_{N}=(a_N,b_N)$ and $I'_{N}=(a'_N,b'_N)$ with $a_N\,,b_N\to \Es$ and $a'_N\,,
b'_N\to \Gs>2 \sqrt{\nu''(\qs^2)}$. Then, for any $\epsilon>0$, 
\begin{equation}\label{eq:uncond2}
\lim_{N\to\infty}\frac{1}{\E \#\mathscr{C}_{N,\qs}(I_N,\,I'_{N})}\E\bigg\{ \sum_{\bsigma\in\mathscr{C}_{N,\qs}(I_N,\,I'_{N})} \P^{N,q_o}_{\BJ,\bsigma}\left\{\err_{N,T}(\bsigma)>\ep\right\}\bigg\} = 0.
\end{equation}
Further assuming that
\begin{equation}\label{eq:conc0}
\lim_{a\to0^+}\liminf_{N\to\infty}\mathbb{P}\left\{ \#\mathscr{C}_{N,\qs}\left(I_{N},\,I'_{N}\right)>a\mathbb{E}\left\{ \#\mathscr{C}_{N,\qs}\left(I_{N},\,I'_{N}
\right)\right\} \right\} =1,
\end{equation}
we have that $\lim_{N\to\infty}\P \{ \mathscr{C}_{N,\qs}\left(I_{N},\,I'_{N}\right) \ne \emptyset \} = 1$, and, for any $\epsilon>0$, conditionally on this event,
\begin{equation}\label{eq:uncond}
\frac{1}{\#\mathscr{C}_{N,\qs}(I_N,\,I'_{N})}\sum_{\bsigma\in\mathscr{C}_{N,\qs}(I_N,\,I'_{N})} \P^{N,q_o}_{\BJ,\bsigma}\left\{\err_{N,T}(\bsigma)>\ep\right\} \stackrel{N\to\infty}{\longrightarrow}0\,,\quad \mbox{in prob.}
\end{equation}
\end{theo}

The asymptotics of the expected number of critical points $\E \#\mathscr{C}_{N,\qs}(I_N,\,I'_{N})$ were computed for the pure $m$-spin models in \cite{ABAC} and for general mixed models in \cite{ABA}. However, currently the concentration property of 
 \eqref{eq:conc0} is proved only for pure $m$-spin \cite{Subag2nd} with 
$
\Gs > 2 \sqrt{\nu''(\qs^2)}$
(i.e. $\Es > 2 b_m \qs^{m} \sqrt{1-1/m}$, see Footnote \ref{ft:purecdn}), 
or for mixed small perturbation of them \cite{BSZ} with large enough $\Es$, $\Gs$, $\qs$, 
and for $I_N,I_N'$ of length asymptotically larger than $1/N$. In both cases,
for large $\b$ the model is \abbr{1rsb} and the Gibbs measure concentrates 
on the set of spherical bands around the points in $\mathscr{C}_{N,\qs}\left(I_{N},\,I'_{N}
\right)$, provided that $\qs^2$ is set to be at the position of the non-zero atom of the Parisi measure, 
$-\Es$ is set for  the minimal normalized energy, and $\Gs$ chosen appropriately. 

	For arbitrary $\bsigma\in \qs \SN$, conditional on $\cpt(\Es,\Gs,\bsigma)$ the eigenvalues of the spherical covariant Hessian of  $H_{\BJ}$ at $\bsigma$ have the same distribution as those of a \abbr{GOE} matrix, scaled by $\sqrt{\nu''(\qs^2)(N-1)/N}$ and shifted by $\Gs$. The value $2 \sqrt{\nu''(\qs^2)}$ is the threshold beyond which the Hessian is typically positive definite, i.e., $\bsigma$ is a local minimum. 
	Consequently, as can be checked by an application of the Kac-Rice formula,  if $
	\Gs>2 \sqrt{\nu''(\qs^2)}$ then the ratio of the expected number of minima and the expected number of critical points of all indices in $\mathscr{C}_{N,\qs}\left(I_{N},\,I'_{N}
	\right)$ goes to $1$. 
	In the two situations mentioned above \cite{BSZ,Subag2nd} where \eqref{eq:conc0} holds, the latter also occurs with high probability and not just in expectation. On the other hand, if $
	\Gs<2 \sqrt{\nu''(\qs^2)}$ then the expected number of minima in $\mathscr{C}_{N,\qs}\left(I_{N},\,I'_{N}
	\right)$ decays exponentially fast in $N^2$.

Considering Theorem \ref{thm-uncond} with $q_o=\qs=1$, corresponds to starting
at a critical point $\Bx_0=\bsigma$. This is related to some of the results of \cite{BGJ},
where qualitative information about the limiting dynamics is gained from an approximate 
evolution for (only) the pair $(H_N(s),|\nabla_{\mathrm{sp}} H_{\BJ}(\Bx_s)|/\sqrt{N})$.

Extending \cite[Proposition 1.1]{DGM} to our context, we next establish  
the ``hard spherical constraint'' equations corresponding 
to the limit $L \to \infty$ and $f_L(\cdot)$ of \eqref{eq:fdef}.
\begin{prop}\label{prop-sphere}
For any $T<\infty$ the solutions $(R^{(L)}, C^{(L)},q^{(L)},H^{(L)})$ of 
(\ref{eqR})--(\ref{eqH})
for potential $f_L(\cdot)$ as in (\ref{eq:fdef}) with positive $\bphi=1+2\b q_o \bv'_\star (q_o)$,
converge
as $L \to \infty$, uniformly in $[0,T]^2$, towards 
$(R,C,q,H)$, for $H(\cdot)$ of \eqref{eqH}. Further, 
$q(0)=q_o \in [-\qs,\qs]$, $R(t,t)=C(t,t)= 1$ for all $t \geq 0$,
$R(s,t)=0$ and $C(s,t)=C(t,s)$ when $s < t$, while $(R,C,q)$ 
is for $T \ge s \ge t \ge 0$ the unique bounded solution of
\begin{align}
\partial_s R(s,t) =
& - \mu(s) R(s,t) + \b^2 \int_t^s
R(u,t) R(s,u) \nu''(C(s,u)) du ,\label{eqRs}\\
\partial_s C(s,t)= & 
 - \mu(s) C(s,t) + 
\b^2 \int_0^s R(s,u) \Big[ \nu''(C(s,u)) C(u,t) 
 -  \frac{q(t) \nu'(q(u)) \nu''(q(s))}{\nu'(\qs^2)} \Big] \, du
\nonumber \\
& \qquad \qquad \qquad + \b^2 \int_0^t R(t,u) \Big[\nu'(C(s,u)) - \frac{\nu'(q(s)) \nu'(q(u))}{\nu'(\qs^2)} \Big] \, du + 
\b q(t) \bv'_\star (q(s))  \,,
 \label{eqCs}\\
\partial_s q(s) = & -\mu(s) q(s) + \b^2 \int_0^s R(s,u) 
\Big[ q(u) \nu''(C(s,u)) 
- \frac{\qs^2 \nu'(q(u)) \nu''(q(s))}{\nu'(\qs^2)} \Big] du + 
\b \qs^2 \bv'_\star (q(s)) \,,
\label{eqqs}
\end{align}
\begin{equation}\label{eqZs}
\mu(s)= \frac{1}{2} + \b^2
\int_0^s R(s,u) \Big[ \psi(C(s,u)) -
\frac{\psi(q(s)) \nu'(q(u))}{\nu'(\qs^2)} \Big] 
\, du + 
\b q(s) \bv'_\star(q(s)) \,.
\end{equation}
In addition, $\bar C(s,t) := C(s,t)-q(s)q(t)/\qs^2$ is a 
non-negative definite kernel, and
\begin{equation}\label{eq:rbd-new}
\Big|\int_{t_1}^{t_2} R(s,u) du \Big|^2 \le t_2-t_1 \,,   
\qquad 0 \leq t_1 \leq t_2  \leq s < \infty \,.
\end{equation}
\end{prop}

\begin{remark}\label{rem:CK}
Since $\bv(0)=\bv'(0)=0$, taking $q_o=0$ yields the solution 
$q(s) \equiv 0$ in both \eqref{eqq} and \eqref{eqqs}. 
The values of $(\Es,\Gs,\qs)$ are then irrelevant, and the system
of equations \eqref{eqR}-\eqref{eqZ}, \eqref{eqRs}-\eqref{eqZs}
reduces to the \abbr{CKCHS}-equations, as in \cite[Theorem 1.2]{BDG2} and 
\cite[Proposition 1.1]{DGM}, respectively. All terms involving 
$\bv_\star(\cdot)$ disappear also when $\Es=\Gs=0$, but for
$q_o \ne 0$ the equations \eqref{eqq} and \eqref{eqqs} nevertheless
yield non-zero solutions. Unlike the special case 
of \cite[Proposition 1.1]{DGM}, here $(R,C,q)$ may take 
negative values, but with $C(s,s)=1$ and $\bar C (\cdot,\cdot)$ 
non-negative definite, necessarily $|q(\cdot)| \le \qs$ 
and $|C(\cdot,\cdot)| \le 1$. 
\end{remark}

\begin{remark}\label{rem:m-infty}
Any $\bphi \in (0,\infty)$ in \eqref{eq:fdef} result with 
equations \eqref{eqRs}--\eqref{eqZs} when $L \to \infty$,  but since 
$\mu(0)=\bphi/2$, taking $\bphi=1+2\b q_o \bv'_\star (q_o)$ (when it is positive), 
simplifies our derivation (otherwise, one merely has to use $\mu(0^+)$ when $s=0$).
The representation \eqref{potential} with $m=\infty$ applies for any real-analytic $\nu(\cdot)$ 
such that $\nu(0)=\nu'(0)=0$, $\nu^{(p)}(0) \ge 0$, $p \ge 2$, with a unique strong solution to
\eqref{diffusion} for locally Lipschitz $f'(r)$ growing fast enough as $r \to \infty$.
While not pursued here, we expect Theorem \ref{thm-macro} to hold for any 
such $f(\cdot)$ and upon considering $f_L(r) = L(r-1)^2+f(r)$, to further arrive at
the conclusions of Proposition \ref{prop-sphere}.
\end{remark}
\begin{remark} Whenever $\nu(\cdot)$ is an even polynomial, so is $\bv_\star(\cdot)$, resulting 
with \eqref{eqR}-\eqref{eqH} invariant under $(C,R,q,H) \mapsto (C,R,-q,H)$.
The same applies to \eqref{eqRs}-\eqref{eqZs} and in such cases $q_o \mapsto -q_o$
yields the same solution apart from a global sign change in $q(s)$.
Indeed, our realization is such that 
an even $\nu(\cdot)$ results with an even potential 
$H_{\BJ}(-\bx)=H_{\BJ}(\bx)$ per given $\BJ$, hence also 
with $\cpt(E,G,\bsigma) = \cpt(E,G,-\bsigma)$ and thereby a sign change 
$q_o \mapsto -q_o$ being equivalent to $\bsigma \mapsto - \bsigma$.
\end{remark}

In Section \ref{sec-asymp} we study the large time asymptotic of the solution 
$(R,C,q,\mu)$ of \eqref{eqRs}-\eqref{eqZs}, establishing the
\abbr{fdt} regime at high temperature (ie $\b$ small), and further analyzing 
the plausible \abbr{fdt} solutions for somewhat lower temperatures. While doing 
so, we observe a sharp distinction between the $m$-pure case and the mixed case, in terms of the 
emergence of aging. Such distinction was realized recently in \cite{FFT}, by a numerical solution of  the \abbr{CKCHS}-equations for initial conditions from the Gibbs measure at different temperatures, suggesting, for example, more than one dynamical phase transition  in the mixed case only.
In Section \ref{sec:strongsolutions} we prove Theorem \ref{thm-macro} 
by adapting \cite[Section 2]{BDG2} to our more challenging setting (where 
$\bx_0$ is related to $\bJ$ via \eqref{eq:subsphere}--\eqref{eq:cond-J}).
The key to our derivation are Propositions \ref{prop-macro} and \ref{lem-diff}, whose
proofs are deferred to Subsections \ref{sec:der} and \ref{sec:diff}
(adapting \cite[Section 3]{BDG2} and \cite[Section 4]{BDG2}, respectively).
From Proposition \ref{prop-macro} one further has the limit dynamics (as $N \to \infty$),
for other functions of interest (such as those given in \eqref{eq:hfdef}--\eqref{eq:wHNdef}).
Section \ref{sec:eliran} is devoted to proving our main result, Theorem \ref{thm-uncond}, whereas
Proposition \ref{prop-sphere} and Proposition \ref{prop:fdt} are  
established in Sections  \ref{sec:hard-sphere} and \ref{sec:fdt}, respectively,
by adapting \cite[Section 2]{DGM} and \cite[Section 4]{DGM},
to our more involved setting.

\section{Large time asymptotic: the \abbr{fdt} regime}\label{sec-asymp}

At high enough temperature one has that $q(s) \to 0$ for $s \to \infty$.
Our next proposition (which is comparable to \cite[Theorem 1.3]{DGM}), 
shows that the \abbr{FDT} regime of the solution of \eqref{eqRs}--\eqref{eqZs} 
then coincides with that of the 
\abbr{ckchs}-equations.
\begin{prop}\label{prop:fdt}
For $\beta$ small enough and $\alpha=0$,  the solution of \eqref{eqRs}--\eqref{eqZs} 
is such that $\varliminf \{ \mu(\tau) \} >0$,  
$(R(t+\tau,t),\bar C(t+\tau,t),q(\tau)) \to (0,0,\alpha \, \qs)$ 
exponentially fast in $\tau \to \infty$, uniformly in $t$, and for any $\tau\ge 0$,
\begin{equation}\label{dfn:cS-space}
\lim_{t\ra\infty} (R(t+\tau,t),C(\tau+t,t),q(t)) = (R_{\fdt} (\tau),C_{\fdt} (\tau), \alpha \, \qs) \,.
\end{equation}
In such case, necessarily $R_{\fdt} (\tau) = - 2 C_{\fdt}'(\tau)$. Further, setting  $\gamma=1/2$ and
\begin{equation}\label{dfn:phi-FDT}
\phi(x):=\gamma + 2 \b^2 \nu'(x) \,,
\end{equation}
we have that $\mu(t) \to \phi(1)$, 
and $C_{\fdt}(\cdot)$ is the unique $[0,1]$-valued, continuously differentiable solution of 
\begin{equation}\label{FDTDb}
 D'(s)=-\int_0^s\phi(D(v)) D'(s-v) dv - \frac{1}{2}\,, \qquad D(0)=1 \,.
\end{equation}
More generally, if the solution $(R,C,q)$ of \eqref{eqRs}--\eqref{eqZs} is uniformly 
bounded, with $\{R(t+\cdot,t), t \ge T_0\}$ uniformly integrable (\abbr{wrt} 
Lebesgue measure), $\varliminf \{\mu(\tau)\} >0$ and \eqref{dfn:cS-space} holds 
for some $\alpha \in [-1,1]$, 
then necessarily $\mu(t) \to \mu$ such that  
$(R_{\fdt},C_{\fdt},\mu)$ satisfy \cite[(4.15)-(4.17)]{DGM}, 
with
\begin{align}
\label{eq:Qa-const}
\mu \, \qinf &= \b \qs^2 \bv'_\star (\qinf)  
- \b^2 \qs^2 \frac{\nu''(\qinf) \nu'(\qinf)}{\nu'(\qs^2)} \kappa_2 
+ \b^2 \qinf \kappa_1
 \,, \quad 
\kappa_1 := \int_0^\infty R (\theta) \nu''(C(\theta)) d\theta \,, \\
\IJ &= \b  \qinf \bv'_\star(\qinf) - \b^2 \frac{\psi(\qinf) \nu'(\qinf)}{\nu'(\qs^2)} \kappa_2  + \b^2 \kappa_3 \,,
\qquad\quad  \kappa_2 := \int_0^\infty R (\theta) d\theta  \,, \quad \kappa_3 := 0 \,. \label{eqIFDT} 
\end{align}
One such solution is $
(-2D'(\cdot),D(\cdot),\phi(1))$ for $(\phi,D)$ of
\eqref{dfn:phi-FDT}-\eqref{FDTDb} and $D_\infty \in [0,1)$,  $\gamma \in \R$ such that 
\begin{align}\label{eq:I-gamma}
\IJ & =\gamma- \frac{1}{2} + 2 \b^2 D_\infty \nu'(D_\infty) \,, \\
\label{eq:gamma-cons}
D_\infty &= \sup \{ x \in [0,1] : (\gamma + 2 \b^2 \nu'(x)) (1-x) \ge 1/2 \} \,,
\end{align}
yielding in turn the values 
$\kappa_1 = 2(\nu'(1)-\nu'(D_\infty))$ and 
$\kappa_2 = 2(1-D_\infty)$.
\end{prop}

\begin{remark} Our proof of \eqref{dfn:cS-space} relies on $\Psi(\cdot)$ of 
\eqref{eqRs-Psi}-\eqref{eqqs-Psi} being a contraction on a suitable set $\Aa$
(and for uniqueness of $(R_{\fdt},C_{\fdt})$ we require that the induced
map $\Psi_{\fdt}(\cdot)$ be a contraction at the given $\alpha$). In particular, 
a global contraction 
requires that $\alpha=0$ be the unique solution of \eqref{eq:Qa-const}, which in turn 
depends not only on $\b$ and $\qs$ but also on $(\Es,\Gs)$. Nevertheless, at least when
$b_2=0$ (so  $\bv''_\star(0)=0$), we expect the \abbr{FDT} solution of Proposition \ref{prop:fdt} 
with $\alpha=0$, $\gamma=1/2$, to apply for all $\b < \b_c$ of \cite[(1.23)]{DGM},
provided $q_o=q_o(\b,\Es,\Gs)$ is small enough. 
\end{remark} 
\begin{remark} For pure $m$-spins, \cite{BBM} consider the diffusion \eqref{diffusion}
starting at $\bx_0$ of 
law $\mu^N_{2\b',\BJ}$ for various choices of $\b' \in [0,\infty)$.
Employing the mathematically non-rigorous replica method (in particular, 
its \abbr{1rsb} picture for the Gibbs measure), they predict
the resulting limit equations for $(R,C)$ and their solution in 
the \abbr{fdt} regime. Building on it (and using again the replica 
method), \cite{BF} considers in this setting 
also the limit dynamics of the overlap $q(t)$.
\end{remark} 
\begin{remark}\label{rem:alpha} 
The limit $\alpha$ of $q(t)/\qs$ provides information on the state $\bx_t$ 
in the limit $N \to \infty$, at  $t \gg 1$ which does not scale with $N$. The case $\alpha=0$ 
represents an escape from the energy well about the critical point $\bsigma$ to a point which is 
orthogonal to $\bsigma$. In contrast, $\alpha=1$ implies convergence to the 
projection $\qs^{-1} \bsigma \in \SN$ of the critical point around which the state was initialized. 
Note also 
that for $\alpha=\qs$ the eventual support ${\Bbb S}_\bsigma (\qs^2)$ of the state, is
precisely the sphere of co-dimension $1$ and radius $\sqrt{N(1-\qs^2)}$, \emph{centered 
at} the critical point $\bsigma$.
\end{remark}

While Proposition \ref{prop:fdt} is limited to small $\b$,  we do 
expect \eqref{dfn:cS-space} to hold at all $\b$, albeit having 
$\alpha \ne 0$ for some $(\Es,\Gs)$ and $q_o$ close enough to $\qs$, 
as soon as $\b > \b_+(\Gs)$, where as we detail in the sequel,
$\b_+$ is in general \emph{lower than $\b_c$} of \cite[(1.23)]{DGM}.
To this end, we first briefly review the physics prediction for the
(large time) asymptotic for the \abbr{CKCHS}-equations, namely when 
$q_o=0$, or alternatively, when all terms involving $q(\cdot)$ are
omitted from \eqref{eqRs}-\eqref{eqZs} (see Remark \ref{rem:CK}).
Recall that for this limiting \abbr{CKCHS} dynamics, aging amounts to having  
a non-identically constant $C_{\rm aging}(\cdot)$ such that $C(\tau+t,t) \to C_{\rm aging}(0)$ as
$t \to \infty$ followed by $\tau \to \infty$, whereas
$C(s,\lambda s) \to C_{\rm aging}(\lambda)$ as $s \to \infty$. Now, in
the absence of aging, such prediction 
is given by the \abbr{fdt} solution from Proposition \ref{prop:fdt}, for
$\alpha=0$ and parameters which solve 
\eqref{eqIFDT}-\eqref{eq:gamma-cons} assuming the limit 
$D_\infty$ of $C_{\fdt}(\tau)$ as $\tau \to \infty$ is zero. As explained 
before, doing so amounts to setting $\IJ=0$ and $\gamma=1/2$, whereas 
\eqref{eq:gamma-cons}  holds for such values iff $\b < \b_c$ of 
\cite[(1.23)]{DGM}. 

In contrast, when $\b > \b_c$ the limit $D_\infty$ of $C_{\fdt}(\tau)$ must be strictly positive,
which for $\alpha=0$ indicates the onset of aging and in particular 
having $R_{\fdt}(\tau) \to 0$ at a sub-exponential rate.  Such slow decay is expected
in turn to require the additional relation 
\begin{equation}\label{eq:gamma-aging}
\gamma = 2 \b^2 [ \nu''(D_\infty) (1-D_\infty) - \nu'(D_\infty)] 
\end{equation}
(see \cite[(1.22)]{DGM}), which together with \eqref{eq:gamma-cons} dictate the values
of $\gamma>1/2$ and of $D_\infty = D_\star(\b)>0$, with 
\begin{equation}\label{dfn:D-star}
D_\star (\b) := \sup\{ x \in [0,1] : 4 \b^2 g(x) \ge 1 \} , \qquad {\rm for } \qquad g(x) := \nu''(x)  (1-x)^2 
\end{equation}
(as in \cite[(1.24)]{DGM}). While  \eqref{eq:I-gamma} thereby determines 
$\IJ$, our expressions for $\kappa_i$ in \eqref{eqIFDT} (and in \eqref{eq:Qa-const}),
relied on the uniform in $t$, integrability of $\tau \mapsto R(t+\tau,t)$, which is no longer 
expected. To rectify this, at $\b \ge \b_c$ one adds to these formulas the 
contribution from the aging regime, namely having $\lambda=u/s$ 
bounded away from zero and one, to the integrals on 
the \abbr{rhs} of \eqref{eqCs}-\eqref{eqZs}. As
explained after \cite[(1.24)]{DGM}, the physics ansatz of a single 
aging regime with $R_{\rm aging}(\lambda)=A C'_{\rm aging}(\lambda)$ 
starting at $C_{\rm aging}(1)=D_\infty$ and ending at $C_{\rm aging}(0)=\alpha^2$ 
(ie, having $\bar C_{\rm aging}(0)=0$), implies the increase
\begin{equation}\label{eq:mod-coef}
\begin{aligned}
\kappa_1  &\leftarrow \kappa_1 + A (\;\;\; \;\; \nu'(D_\infty) - \;\;\;\; \nu'(\alpha^2)) \,,\\
\kappa_2  &\leftarrow \kappa_2 + A (D_\infty \qquad \quad - \; \alpha^2 \qquad \;\; ) \,,\\
\kappa_3 &\leftarrow \kappa_3 + A(D_\infty \nu'(D_\infty) - \alpha^2 \nu'(\alpha^2)) \,,
\end{aligned}
\end{equation}
of the coefficients in the identity \eqref{eqIFDT}, which in turn determines the value of $A$. Finally,
should the self-consistency requirement of $A>0$ and $\bar C_{\rm aging}(0)=0$ fail, one 
moves from the latter ansatz into the richer hierarchy of multiple aging regimes.

Recall Remark \ref{rem:alpha}, that for $\alpha=0$ and $\beta>\beta_c$ aging occurs for 
a state which is already orthogonal to the critical point $\bsigma$ around which we initialized the 
system, i.e. after the escape from the energy well about it.
Here we consider another alternative, of having a still localized state, namely a solution with $\alpha \ne 0$
that in addition satisfies  \eqref{eq:Qa-const}. Indeed, recall \cite[Proposition 6.1]{DGM} that the \abbr{FDT}
regime of the \abbr{CKCHS}-equations must be given by \eqref{FDTDb} as soon as a key integral
$I(t+\cdot,t)$ converges for $t \to \infty$ (uniformly on compacts), to some constant
(which in terms of our notations, turns out to be $\widehat \IJ := \gamma - \frac{1}{2} - \IJ + \b^2 \kappa_3$).
Assuming in addition that such convergence to constants $(\IJ^{(q)}_1,\IJ^{(q)}_2)$ applies also for the integrals 
\[
I^{(q)}_1(s) := \int_0^s R(s,u) q(u) \nu''(C(s,u)) du \,, \qquad \qquad I^{(q)}_2(s) := \int_0^s R(s,u)\nu'( q(u)) du \,, 
\]
we have in \eqref{eqqs}, we can approximate the latter dynamics (at $s \gg 1$),
by the much simpler \abbr{ode} 
\begin{equation}\label{dfn:ode}
\begin{aligned}
q'(s) & = - \mu(s) q(s) + {\sf Q}( q(s) )\,, \qquad \qquad {\rm for } \quad \qquad \mu(s)={\sf P} ( q(s) )\,, \\
{\sf Q} (x) &= \b \qs^2 \bv'_\star(x) - \b^2 \qs^2 \frac{\nu''(x)}{\nu'(\qs^2)} \IJ^{(q)}_2 + \b^2 \IJ^{(q)}_1 \,, 
\qquad 
{\sf P} (x) = \b x \bv'_\star(x) - \b^2 \frac{\psi(x)}{\nu'(\qs^2)} \IJ^{(q)}_2 + \rho + \widehat \IJ \,.
\end{aligned}
\end{equation}
Such an \abbr{ode} has no limit sets beyond its finitely many limit points, which are at the 
isolated solutions of 
\begin{equation}\label{eq:fp}
{\bf P}(x) x = {\bf Q}(x)\,,  \qquad \qquad x \in [-\qs,\qs] \,.
\end{equation}
Hence our earlier prediction that \eqref{dfn:cS-space} 
remains valid at all $\b$. Further, a convergence of $q(u)$ to some limit point 
$x=\alpha \qs$ 
implies by  self-consistency the values $\IJ_1^{(q)} = \alpha \qs \kappa_1$ and 
$\IJ_2^{(q)} = \nu'(\alpha \qs) \kappa_2$, which upon substitution in \eqref{dfn:ode}-\eqref{eq:fp} 
yield the requirements \eqref{eq:Qa-const}-\eqref{eqIFDT} on $\alpha$ and $\IJ$.

The analysis of the \abbr{fdt} regime in the presence of aging starts precisely as for \abbr{ckchs}-equations  
with $\b > \b_c$, $D_\infty=D_\star(\b)>0$ of \eqref{dfn:D-star}
and the corresponding values of $(\gamma,\IJ)$ 
(as determined by \eqref{eq:I-gamma}-\eqref{eq:gamma-aging}). The only 
difference is that now we can try beyond the \abbr{ckchs}-solution $\alpha=0$ and
$\IJ = \b^2 \kappa_3$, also any $A>0$ and 
$\alpha^2 = C_{\rm aging}(0) < D_\infty$ which satisfy 
\eqref{eq:Qa-const}-\eqref{eqIFDT} for $\kappa_i$ of \eqref{eq:mod-coef}. Since
$D_\star(\b) \uparrow 1$, 
taking $\b$ large provides 
access to all solutions of \eqref{eq:fp} (but we do not expect a
simple, explicit way to determine which interval of $q_o$ values is attracted
to each  stable solution).

The most interesting case is that of a localized state with \emph{no-aging at $\alpha \ne 0$.} 
Specifically, seeking $(R_{\fdt} (\tau),C_{\fdt} (\tau),\mu)$ as in Proposition \ref{prop:fdt} for 
$\gamma \ne 1/2$ such that $\bar C_{\fdt} (\tau) \to 0$, i.e. 
with $D_\infty=\alpha^2$. Plugging such a solution in \eqref{eq:Qa-const} gives 
\begin{equation}\label{eq:FDT-ident1} 
\gamma  \alpha  = \b \qs \bv'_\star (\qinf) 
- 2 \b^2 \frac{\qs \nu''(\qinf) \nu'(\qinf)}{\nu'(\qs^2)} (1-\alpha^2)  - 2 \b^2  \alpha \nu'(\alpha^2)
\,.
\end{equation}
Similarly, plugging it in  \eqref{eqIFDT} and comparing with \eqref{eq:I-gamma} results with  
\begin{equation}\label{eq:FDT-ident2}
\gamma - \frac{1}{2} = \b  \qinf \bv'_\star(\qinf)
 - 2 \b^2 \frac{\psi(\qinf) \nu'(\qinf)}{\nu'(\qs^2)} (1-\alpha^2) 
- 2 \b^2 \alpha^2 \nu'(\alpha^2)  
  \,.
\end{equation}
Recall \eqref{eq:gamma-cons}, that having $D_\infty = \alpha^2$ requires in addition to the preceding that 
\begin{equation}\label{eq:FDT-ident3}
( \gamma + 2 \b^2 \nu'(\alpha^2) ) (1-\alpha^2) - \frac{1}{2} =  
\frac{2 \b^2}{\nu'(\qs^2)} [ \nu'(\alpha^2) \nu'(\qs^2) - \nu'(\alpha \qs)^2 ] (1-\alpha^2) = 0  \,.
\end{equation}
In the pure case the \abbr{rhs} of \eqref{eq:FDT-ident3} always holds, while otherwise 
it holds only\footnote{except for $\alpha = - \qs$ equivalently holding whenever $\nu(\cdot)$ 
is an even polynomial 
} 
for $\alpha=\qs$. 
Proceeding first with the $m$-pure case, utilizing 
Footnotes \ref{ft:purecdn} and \ref{footnote4},
we get that both \eqref{eq:FDT-ident1} and \eqref{eq:FDT-ident2} hold for $\alpha \ne 0$ iff 
\begin{equation}\label{eq:pure-no-aging}
4 \b^2 g(\alpha^2) = y^2 \qquad {\rm and } \qquad 
\Gs = \sqrt{\nu''(\qs^2)} ( y + y^{-1} )  
\,. 
\end{equation}
In view of \eqref{eq:gamma-cons}, only the  smaller positive 
root $y \in (0,1]$ for the \abbr{rhs} of \eqref{eq:pure-no-aging} is relevant, with the  
condition $\Gs > 2 \sqrt{\nu''(\qs^2)}$ for existence of such $y  \in (0,1)$
matching our assumption in Theorem \ref{thm-uncond}
(alternatively, the latter inequality amounts to $\tEs > 2 \sqrt{1-\frac{1}{m}}$ 
where $\tEs := \Es / (b_m \qs^m)$ denotes the given energy level,  
measured in standard deviations of $H_{\BJ}(\bsigma)$).
Moreover, the \abbr{lhs} of \eqref{eq:pure-no-aging} can not hold for some $y<1$, unless
\begin{equation}\label{eq:tap-cond}
\frac{1}{\b} > 2 \sqrt{\nu''(\alpha^2)} (1-\alpha^2) \,,
\end{equation}
which is precisely the stability condition for \abbr{TAP} solutions on $\alpha \, \SN$
(see \cite[Eq. (25)]{KPV}). Fixing $\tEs$ as above, namely $y \in (0,1)$ via the \abbr{rhs} of 
\eqref{eq:pure-no-aging}, here
$g(\cdot)$ attains its maximum over $[0,1]$ at $\alpha_m^2 := 1-\frac{2}{m}$, and 
by the same reasoning as for the \abbr{ckchs}-equations, one should choose 
the larger solution $\alpha^2$ in \eqref{eq:pure-no-aging}, namely take
\begin{equation}\label{eq:beta-crit-pure}
D_\infty = D_\star(\b/y)  \qquad {\rm provided } \qquad  
\b > \b_+ := y/(2 \sqrt{g(\alpha_m^2)})\,, 
\end{equation}
where $\b_+ < \b_c$ of \cite[(1.23)]{DGM}, for any $m \ge 2$ and all $\tEs$ as above. 

Turning to the mixed case, first note that $\bv_\star'(\qs^2) = \Gs$ (see 
\eqref{eq:Gtilde} at
$\bx_t = \bn$).
Upon plugging the generic solution $\alpha= \qs$ of \eqref{eq:FDT-ident3} 
into \eqref{eq:FDT-ident1}, it follows that no-aging with $\alpha \ne 0$ requires 
the \abbr{rhs} of \eqref{eq:pure-no-aging} 
to hold for $y \le 1$ and 
$\qs^2 = D_\star(\b/y)$ of \eqref{dfn:D-star}.
Taken together, we see that \eqref{eq:pure-no-aging} must hold at $\alpha=\qs$, yielding the
relation 
\begin{equation}\label{eq:G-alpha}
\Gs= \Gs(\alpha,\b) := 2 \b \nu''(\alpha^2) (1-\alpha^2) + \frac{1}{2 \b (1-\alpha^2)} \,,
\end{equation}
where the restriction to  $y < 1$ amounts to the inequality \eqref{eq:tap-cond}.
\newline
It is easy to check that having such $(R_{\fdt} (\tau),C_{\fdt} (\tau),\mu)$ as
in Proposition \ref{prop:fdt}, except for possibly $\gamma \ne 1/2$, and with 
the no-aging condition $D_\infty=\alpha^2$ in place, implies the convergence of 
$H(s)$ of \eqref{eqH} as $s \to \infty$, to the limiting (macroscopic) energy
\begin{equation}\label{eq:pre-H-inf}
H(\infty) := \bv_\star(\alpha \qs) + 2 \b \theta(\alpha^2) \,, \qquad \mathrm{where} \qquad 
\theta(q) :=  \nu(1) - \nu(q) - \nu'(q) (1-q) 
\end{equation}
(and to arrive at \eqref{eq:pre-H-inf} we also use the \abbr{rhs} of \eqref{eq:FDT-ident3}).  

For $\bsigma\in \alpha \SN$, similarly to the proof of Lemma \ref{kval}, one can check that
conditionally on $\cpt(\Es,\Gs,\bsigma)$
the Gaussian field $H_\BJ(\bx)$ has expectation $-N\Es$ and variance $N\theta(\alpha^2)$
at any $\bx$ in the sub-sphere ${\Bbb S}_\bsigma(\alpha^2)$ of \eqref{eq:subsphere}.
Using this conditional field, one has the spherical model \abbr{wrt} the uniform measure 
$\mu_\bsigma^{\alpha^2}(\bx)$ on ${\Bbb S}_\bsigma(\alpha^2)$, its Gibbs measure
$\mu_{\beta_0,\BJ}^{\bsigma}$  of density $(Z_{\b_0,\BJ}^{\bsigma})^{-1} e^{-\beta_0 H_\BJ(\bx)}$ 
and the corresponding free energy $F_{\beta_0}(\bsigma)$ to which $N^{-1}\log Z_{\b_0,\BJ}^{\bsigma}$ 
converges.
If for any  $\beta_0$ near $2\beta$ this model is replica symmetric, then 
$F_{\beta_0}(\bsigma)= \beta_0 \Es +\frac{\beta_0^2}{2} \theta(\alpha^2)$ 
and
most of the mass of $\mu_{2\beta,\BJ}^{\bsigma}$ is indeed typically carried at 
the energy $\Es+2\beta \theta(\alpha^2)$.  In the mixed case we know that 
$\alpha=\qs$ hence the state $\bx_t$ is supported for $t \gg 1$ on that same sub-sphere 
${\Bbb S}_\bsigma(\alpha \qs)=
{\Bbb S}_\bsigma(\alpha^2)$ (see Remark \ref{rem:alpha}). Further, in the $m$-pure case 
$\cpt(\Es,\Gs,\bsigma)=\cpt(r^m \Es,r^{m-2} \Gs, r \bsigma)$ for any $r>0$, with  
$r=\alpha/\qs$ eliminating the effect of $\qs$ and allowing us to 
take again \abbr{wlog} $\|\bsigma\|=\alpha \sqrt{N}=\qs \sqrt N$. Recall that
$\bv_\star(\alpha \qs)=\bv_\star(\qs^2) = \Es$
(see \eqref{eq:Gtilde} at $\bx_t=\bn$),  so the energy $\alpha^m \tEs + 2 \b \theta(\alpha^2)$
carrying most 
of the mass of the spherical model $\mu_{2\beta,\BJ}^{\bsigma}$ is for such $\bsigma$ 
precisely the limit $H(\infty)$ of \eqref{eq:pre-H-inf}. Further, re-writing the conditional Gaussian
field of $\mu^\bsigma_{2 \b,\BJ}$ as a polynomial in the re-centered coordinates 
$\bx - \bsigma$ gives a new spherical mixed model, see \cite[Lemma 7.1]{BSZ}, 
whose $2$-spin interaction part is in the replica symmetric regime precisely when
\eqref{eq:tap-cond} holds (c.f. \cite[(7.6) and (8.8)]{BSZ}). Finally,
in the $m$-pure case, the relation
\eqref{eq:G-alpha} determines from the energy $\tEs$ a limiting sub-sphere height $\alpha$ 
which is a local maximum of the free energy $F_{2 \b}(\bsigma)$ plus the entropy
$\frac{1}{2} \log (1-\alpha^2)$.

\subsection{Limiting dynamics for spherical SK-model}
While of less interest from the physics  point of view, for the spherical {\bf SK}-model, 
namely $m=2$, one can solve \eqref{eqRs}-\eqref{eqZs} and thereby confirm our 
predictions. Specifically, for $\nu(x)=\frac{x^2}{8}$ (hence $\psi(x) = 2 \nu'(x)=\frac{x}{2}$,
$\nu''(x)=\frac{1}{4}$, $\bv'_\star(x) = \frac{\Gs}{\qs^2} x$), starting at $R(s,s)=1$, 
$\bar C(s,s)= 1 - q(s)^2/\qs^2$ and $q(0)=q_o$ these equations are for 
$s>t$,
\begin{align}\label{eq:SSSK-Rq}
&\partial_s R(s,t) = -  \mu(s) R(s,t) + \frac{\b^2}{4} \int_t^s R(s,u) R(u,t) du \,,  \qquad  
q'(s) =  - (\mu(s)  - \b \Gs) q(s) \,,\\
&\label{eq:SSSK-CZ}
\begin{aligned} 
\partial_s \bar C(s,t) &= -  \mu(s) \bar C(s,t) + \frac{\b^2}{4} \Big[ \int_0^s R(s,u) \bar C(u,t) du 
+  \int_0^t R(t,u) \bar C(u,s) du \Big] \,, 
 \\
\mu(s)  &= \frac{1}{2} + \frac{\b^2}{2} \int_0^s R(s,u) \bar C(s,u) du + \b \Gs \frac{q^2(s)}{\qs^2}  \,.
\end{aligned}
\end{align}
Further, in this case we get from \eqref{eqH} and \eqref{eq:SSSK-CZ} that 
\begin{equation}\label{eq:SSK-H}
H(s)= \frac{1}{2\b} \Big[ \frac{\b^2}{2} \int_0^s R(s,u) \bar C(s,u) du +  \b \Gs \frac{q^2(s)}{\qs^2} \Big] 
= \frac{\mu(s)}{2\b} - \frac{1}{4\b} \,.
\end{equation}
Setting $\Lambda(s) := \qs e^{\int_0^s (\mu(u) - \b \Gs) d u}$
the solution of \eqref{eq:SSSK-Rq} must be 
\[
q(s) = \frac{\qs q_o}{\Lambda(s)}\,, 
\qquad \qquad R(s,t)=\frac{\Lambda(t)}{\Lambda(s)} \CL_{\Gs} (s-t) \,,
\]
where $\CL_{G} (\theta) =  e^{-\b G \theta} \CL(\theta)$ for  
$\CL(\theta) := \frac{2}{\pi} \int_{-1}^1 e^{\b \theta x} \sqrt{1-x^2} \, dx$
(see  \cite[(4.9)]{DGM}).
Substituting this in \eqref{eq:SSSK-CZ}, the symmetric  
$\sM(s,t) :=  \bar C(s,t) \Lambda(s) \Lambda(t)$, is the positive, 
unique solution of 
\begin{equation}\label{sol:SSSK}
\begin{aligned}
\partial_s \sM(s,t) &= - \b \Gs \sM(s,t) + \frac{\b^2}{4} \Big[ \int_0^s \CL_{\Gs} (s-u) \sM (u,t) du 
+  \int_0^t \CL_{\Gs} (t-u) \sM(u,s) du \Big] \,, \quad \forall s>t \,, \\\
\sM'(t) &= q_o^2  + (1-2\b \Gs) \sM(t) + \b^2 \int_0^t \CL_{\Gs} (t-u) \sM(t,u) du \,, \quad \qquad \qquad 
\sM(t,t) = \sM(t) \,,
\end{aligned}
\end{equation}
starting at $\sM(0) = \qs^2 - q_o^2$, and with $\Lambda(t) = \sqrt{q_o^2 + \sM(t)}$. 
By the super-position principle for this linear system 
\begin{equation}\label{eq:super-pos}
\sM(s,t) = (\qs^2 - q_o^2) e^{-\b \Gs (s+t)} \, \sM_{\rm ck}(s,t) + q_o^2 \, \sM_{\Gs} (s,t) \,,
\end{equation}
where $\sM_{\rm ck}$ denotes the \abbr{ckchs}-type solution of 
\eqref{sol:SSSK} with $q_o=\Gs=0$, starting at $\sM_{\rm ck}(0)=1$, while
$\sM_{\Gs}$ is the solution of \eqref{sol:SSSK} for $q_o^2=1$ and  
$\sM_{\Gs}(0)=0$. The spherical {\bf SK}-model is somewhat degenerate, as in
view of \eqref{eq:super-pos}, having
$q(t) \to \alpha \ne 0$, or equivalently a finite limit for $\sM(t)$ as $t \to \infty$,
does not depend on the value of $0<|q_o|<\qs$ and when such non-zero limit
exists, the same invariance to $q_o$ applies to the issue of no-aging
(i.e. having $\sM(t+\tau,t) \to 0$ as $t \to \infty$ followed by $\tau \to \infty$). 
The analog of $\sM_{\rm ck}$
for \eqref{eqR}-\eqref{eqZ} at $q(\cdot) \equiv 0$ and linear $f'(x)=c x$,
is studied in \cite[Section 3]{2001}. A similar 
but finer analysis
shows that $\sM_{\rm ck}(s,t)$ grows as $s,t \to \infty$,
up to some polynomial pre-factors, at the exponential rate $\mus (s+t)$, 
where $\mus=\b$ for $\b>1$ and otherwise $\mus=(1+\b^2)/2$.
Focusing on the case of a stable energy well around the critical point $\bsigma$, namely 
$\Gs>1$ as in Theorem \ref{thm-uncond}, we have that 
$\b \Gs > \mus$ 
iff $\b>y$, with $y \in (0,1)$ as in the \abbr{rhs} 
of \eqref{eq:pure-no-aging}. We thus have the dichotomy predicted earlier, 
that $q_o=0$ requires $\alpha=0$, with the onset of aging at $\b_c$ determined
by the asymptotic of $\sM_{\rm ck}(s,t)/ \sqrt{\sM_{\rm ck} (s) \sM_{\rm ck} (t)}$,
whereas for any $q_o \ne 0$, $\Gs > 1$ and $\b > y$ we have a localized state, 
with $\alpha^{-2} - 1$ given by the finite
limit of $\sM_{\Gs}(t)$, and $C_{\fdt}(\tau)$ being the limit as $t \to \infty$ of 
$(1+\sM_{\Gs}(t+\tau,t))/(1+\sM_{\Gs}(t))$. We get these limits by replacing $\sM_{\Gs}(s,t)$
with the stationary solution $\sM^{(\rm st)}_{\Gs}(s,t)$ of \eqref{sol:SSSK}
when all the integrals start at $-\infty$ (instead of at zero). By translation invariance,
$\sM^{(\rm st)}_{\Gs}(s,t)$
must be of the form $\Gamma(s-t)$ for symmetric  $\Gamma(\cdot)$ such that
\begin{equation}\label{sol-st:SSSK}
\begin{aligned}
\Gamma'(\tau) &=  - \b \Gs \Gamma (\tau) + \frac{\b^2}{4} \big[ \int_0^\infty 
\CL_{\Gs} (u) \Gamma (u+\tau) du +  \int_0^\infty \CL_{\Gs} (u) \Gamma (u-\tau) du \big]  \,, \\
0 &= 1  + (1-2\b \Gs) \Gamma(0) + \b^2 \int_0^\infty \CL_{\Gs} (u) \Gamma(u) du \,.
\end{aligned}
\end{equation}
Next, recall that $y \in (0,1)$ on the \abbr{rhs} of \eqref{eq:pure-no-aging} satisfies
\begin{equation}\label{dfn:y-alt}
1 - 2 \Gs y + y^2 = 0 \,, \qquad \mbox{that is} \qquad y =  \Gs - \sqrt{G^2_\star -1} 
\end{equation}
and hence (see  \cite[Page 16]{2001}), also   
\begin{equation}\label{dfn:int-CL}
y = \frac{1}{2\pi} \int_{-2}^2 \frac{\sqrt{4-x^2}}{(2\Gs) -x} \, dx = 
\frac{\b}{2} \int_0^\infty \CL_{\Gs}(\theta) d\theta \,.
\end{equation}
Further, utilizing \eqref{dfn:y-alt}, \eqref{dfn:int-CL}, with $\CL_{\Gs}(0)=1$ and having 
\begin{equation}\label{eq:ODE-CL}
\CL_{G}' (\tau) = - \b G \CL_{G}(\tau) + \frac{\b^2}{4} \int_0^\infty \CL_G(u) \CL_{G} (\tau-u) du 
\end{equation}
(compare with the \abbr{lhs} of \eqref{eq:SSSK-Rq}), one can verify that 
\[
\Gamma(\tau) = \frac{1}{c} \int_\tau^\infty \CL_{\Gs}(u) du \,, 
\qquad c := 2 - \int_0^\infty \CL_{\Gs}(u) du = 2\Big(1-\frac{y}{\b}\Big) \,,
\]
satisfies \eqref{sol-st:SSSK}. Consequently, in this case
\begin{equation}
\alpha^{-2} - 1 = \Gamma(0) = \frac{2}{c} - 1 \,, \qquad \mbox{that is} \qquad \alpha^2 = \frac{c}{2} = 1 - \frac{y}{\b} 
\end{equation}
in agreement with our prediction 
on the \abbr{lhs} of \eqref{eq:pure-no-aging}, whereas
\begin{equation}
C_{\fdt}(\tau) = \frac{1+\Gamma(\tau)}{1+\Gamma(0)} = 1 - \frac{1}{2} \int_0^\tau \CL_{\Gs}(u) du \,,
\end{equation}
is precisely $D(\tau)$ of \eqref{FDTDb} for 
$\phi(x)=\b \Gs  + \frac{\b^2}{2} (x-1)$, and 
converges to  $D_\infty = \alpha^2$ (i.e. with no-aging). In addition, having here $\mu(s) \to \Gs \b$ we get
from \eqref{eq:SSK-H} that $H(s) \to H(\infty) = \frac{\Gs}{2} - \frac{1}{4\b}$ (matching  
the expression $H(\infty)=\alpha^2 \frac{\Gs}{2} + \frac{\b}{4} (1-\alpha^2)^2$ of \eqref{eq:pre-H-inf}).


\section{\label{sec:strongsolutions} Proof of Theorem \ref{thm-macro} at $\bsigma=\bn$, $\beta=1$}

In view of Remarks \ref{rem:beta-1}--\ref{rem:rot-sym}, \abbr{wlog} we fix 
throughout this section $\beta=1$ and $\bsigma=\bn=(\sqrt{N} \qs,0,\ldots,0)$. 
Fixing also $T$ and letting $d(N,m)$ be the length of the coupling 
vector $\bJ$, following \cite{BDG2} 
we equip the product space
$\Ea_N = \R^N  \times \R^{d(N,m)}\times \C([0,T],\R^N) $ with the norm 
\begin{equation}\label{eq:norm}
\|(\Bx_0, \BJ, \BB)\|^2= \sum_{i=1}^N(x_0^i )^2 
+ \sum_{p=2}^m \;\; \sum_{1 \leq i_1 \leq \ldots \leq i_{p} \leq N} 
(N^{\frac{p-1}{2}}J_{i_1\cdots i_{p}})^2
+ \sup_{0 \leq t \leq T}\sum_{i=1}^N(B_t^i)^2 
\end{equation}
and denote by $\wP = \mu_{\bn}^{q_o} \otimes \gamma_N^{(E,G,\qs)} \otimes P_N$
the product probability measure of $(\Bx_0,\BJ,\BB)$ 
on $\Ea_N$, where $\Bx_0$ follows the law
$\mu_{\bn}^{q_o}$ (defined above \eqref{eq:subsphere}), 
$\gamma_N^{(E,G,\qs)}$ denotes the (Gaussian) distribution of $\BJ$ 
conditional upon  $\cpt(E,G,\bn)$\footnote{which in the pure case is  
restricted to $G=m E/ \qs^2$; see Footnote \ref{ft:purecdn}} 
and $P_N$ stands for the distribution of $N$-dimensional Brownian motion.
Next, for
$C_N(s,t)$ of \eqref{empiricalcovariance} and 
$q_N(s) = q_N^{\bn}(s)$ of \eqref{eq:bnq}, we let
\begin{equation}
\bar{C}_N(s,t) := C_N(s,t) - \qs^{-2} q_N(s) q_N(t) = 
\frac{1}{N} \sum_{i=2}^N x^i_s x^i_t \,, 
\qquad \qquad 
q_N(s) = \frac{\qs x_s^1}{\sqrt{N}}  \,.
\label{eq:barC}
\end{equation}
Setting $G^i(\Bx):= - \partial_{x^i} H_{\BJ} (\Bx)$, the derivation of Theorem
 \ref{thm-macro} builds on the proof of \cite[Thm. 1.2]{BDG2}, which utilizes
beyond $C_N$ and $\chi_N$ of \eqref{empiricalcovariance}-\eqref{integrated}, two auxiliary functions 
$A_N$ and $F_N$ (see \cite[(1.15)]{BDG2}). Having here a
distinguished first coordinate, those four functions of \cite{BDG2}  are replaced by
$\Uao_N := \{C_N,\chi_N ,\bar C_N,\bar \chi_N, \bar A_N, \bar F_N \}$, for $\bar C_N$ of \eqref{eq:barC} and
\begin{align}\label{eq:hfdef}
\bar{\chi}_N(s,t) := 
\frac{1}{N} \sum_{i=2}^N x^i_s B^i_t \,, \qquad 
\bar A_N(s,t):=\oneN \sum_{i=2}^N G^i(\Bx_s) x^i_t\,,\qquad
\bar F_N(s,t):=\oneN \sum_{i=2}^N G^i(\Bx_s) B^i_t \,.
\end{align}
Beyond $\Uao_N$, our derivation clearly has to involve $q_N$ of \eqref{eq:barC} and in addition,
the pre-limit of $\wH$ 
from \eqref{eqH}, and the (centered) contribution of the first coordinate to $A_N$, given respectively by
\begin{equation}\label{eq:wHNdef}
\wH_N(s) := - \oneN [H_{\bJ}(\bx_s)-\bar H(\bx_s) ] \,,
\qquad \qquad \qquad 
\wQ_N(s)  :=\frac{\qs}{\sqrt{N}} (G^1(\Bx_s)-\bar G^1(\Bx_s)) \,,
\end{equation}
where $\bar{G} (\Bx) := - \nabla \bar H (\Bx)$ and
$\bar{H}(\Bx) := \E [ H_{\BJ} (\Bx)\,|\,\cpt(E,G,\bn) ]$. 
Analogously to $D_N$ and $E_N$ \cite[(1.16)]{BDG2},  it is convenient to define
in addition to $\wQ_N$, 
$\bar A_N$ and $\bar F_N$, also their contribution to the incremental changes in
$q_N$, $\bar C_N$ and $\bar \chi_N$, which for $K_N(t):=C_N(t,t)$ are given respectively by
\begin{align}\label{eq:qvdef}
Q_N(s) & :=-f'(K_N(s))) q_N(s) 
+ \qs^2 \bv'(q_N(s)) 
+  \wQ_N(s) 
\,,  \\
\bar D_N(s,t)& :=-
f'(K_N(t)) \bar{C}_N(t,s) + \bar A_N(t,s) \,, \qquad
\bar E_N(s,t):=-f'(K_N(s))) \bar \chi_N(s,t) + \bar F_N(s,t) \,.
\label{eq:dedef}
\end{align}
We shall establish limit equations for
$\Ua_N = \Uao_N \cup \{ q_N,\wH_N,\wQ_N,Q_N,\bar D_N,\bar E_N,\Upsilon_N,\Phi_N,\Phi^1_N,\Psi_N,\Psi^1_N \}$,
where
\begin{equation}\label{eq:PhiN-val}
\begin{aligned}
\Upsilon_N(s,u) &:= \nu(C_N(s,u)) 
  - \bar C_N(s,u) \frac{\nu'(q_N(s))\nu'(q_N(u))}{\nu'(\qs^2)} \,, 
\\
\Phi_N(s,u) &:= \nu'(C_N(s,u)) - \frac{\nu'(q_N(u))\nu'(q_N(s))}{\nu'(\qs^2)} \,,
\\
\Phi_N^1(s,u) &:= q_N(u) \nu'(C_N(s,u)) - 
\bar C_N(s,u) \frac{\qs^2 \nu'(q_N(u))\nu''(q_N(s))}{\nu'(\qs^2)} \,,
\end{aligned}
\end{equation}
\begin{equation}\label{eq:PsiN-val}
\begin{aligned}
\Psi_N(s,u) := & 
\nu''(C_N(s,u)) \big( \bar D_N(s,u)+
\frac{q_N(s)}{\qs^2} Q_N(u) \big)  
- \frac{\nu'(q_N(s))\nu''(q_N(u))}{\nu'(\qs^2)} 
Q_N(u)  \,, \\
\Psi^1_N(s,u) :=&
\bar D_N(s,u)
\big[ \nu''(C_N(s,u)) q_N(u) -  \frac{\qs^2 \nu'(q_N(u))\nu''(q_N(s))}{\nu'(\qs^2)} \big]  \\
& + Q_N(u) 
[\nu'(C_N(s,u)) + \frac{q_N(s) q_N(u)}{\qs^2} \nu''(C_N(s,u)) 
- \qs^2  
\bar C_N(s,u) \frac{\nu''(q_N(s))\nu''(q_N(u))}{\nu'(\qs^2)}  \big] \,.
\end{aligned}
\end{equation}
The functions $\Upsilon_N$, $\Phi_N$, $\Phi_N^1$, $\Psi_N$  and $\Psi^1_N$, which
arise out of conditional covariances (see \eqref{eq:kval}, \eqref{eq:varhpi-val} and \eqref{eq:psii-val}), 
are used in approximating certain conditional expectations 
of $\wH_N$, $V_N$ and $\bar A_N$. 

For convenience we refer hereafter to all elements of $\Ua_N$ 
as functions on $[0,T]^2$, with the obvious modification in force  for $q_N,\wH_N,\wQ_N$ and $Q_N$.
Adopting this convention, our  proof of Theorem \ref{thm-macro} relies on  
pre-compactness and self-averaging of
functions from $\Ua_N$. Specifically, in Section \ref{sec:self-average} we establish 
the following analog of \cite[Prop. 2.3 and 2.4]{BDG2}.
\begin{prop}\label{tight-self}
For any $U_N \in \Ua_N$, fixed finite $T$ and $k$,
\begin{equation}\label{eq:u-mombd}
\sup_{|E|,|G| \le \alpha} \sup_N 
\wE \big[ \sup_{s,t \leq T} |U_N(s,t)|^k 
\big] < \infty \,,
\end{equation}
with the sequence of 
continuous functions $U_N(s,t)$ being pre-compact 
almost surely and in expectation, \abbr{wrt} the uniform topology on $[0,T]^2$. 
Moreover, for  any
$U_N \in \Ua_N$, $T<\infty$ and $\rho > 0$,
\begin{equation}\label{eq:asself}
\sum_N \sup_{|E|,|G| \le \alpha} 
\wP [ \sup_{s,t \le T} |U_N(s,t)- \wE \, U_N (s,t))| \geq \rho] < \infty 
\end{equation}
and hence by \eqref{eq:u-mombd}, also 
\begin{equation}\label{eq:l2self}
 \lim_{N \rightarrow \infty} \sup_{|E|,|G| \le \alpha}
\sup_{s,t \leq T} \wE \Big[ |U_N(s,t)-\wE \, U_N (s,t)|^{2} \Big]=0\,.
\end{equation}
\end{prop}

In view of \eqref{eq:u-mombd} and \eqref{eq:l2self}
we thus deduce the following, 
exactly as in \cite[proof of Corollary 2.8]{BDG2}.
\begin{cor}\label{cor-self}
Suppose $\Psi: \R^\ell \to \R$ is locally Lipschitz with
$|\Psi(z)| \leq M' \|z \|_k^k$ for some $M',\ell,k<\infty$,
and $\BZ_N \in \R^\ell$
is a random vector, where 
for $j=1,\ldots,\ell$, the $j$-th 
coordinate of $\BZ_N$ is of the form $U_N(s_j,t_j)$,
for some $U_N \in \Ua_N$ and some $(s_j,t_j) \in [0,T]^2$. Then,
\begin{equation}\label{eq:self-av-Psi}
\lim_{N \to \infty} 
\sup_{|E|,|G| \le \alpha} \sup_{s_j,t_j} |\wE \Psi(\BZ_N) - \Psi(\wE \BZ_N)| = 0 \;.
\end{equation}
\end{cor}

As explained in Remark \ref{rem:cond-H}, the expectation $\wE$ 
amounts to taking $\BJ=\BJ_o$ of the Gaussian law $\gamma_N^{(0,0,\qs)}$, while 
adding to \eqref{diffusion} the drift corresponding to \eqref{eq:Htilde},
provided that we add back to $(G^1,H_{\BJ})$ the relevant constant
shift $(\bar G^1,\bar H)$. For $\beta=1$, $\bsigma =\bn$, this provides an alternative representation via the diffusion
\begin{equation}\label{eq:integ}
x^i_s=x^i_0+B^i_s -\int_0^s f'(K_N(u))x^i_u du +\int_0^s
G^i(\bx_u) du 
+ {\bf 1}_{\{i=1\}} \sqrt{N} \qs \int_0^s \bv' (q_N(u)) du  \,, \quad i=1,\ldots,N,
\end{equation}
starting at $\Bx_0$ of law $\mu_{\bn}^{q_o}$ independently of $\BB$ and $\BJ$, 
while in studying $\Ua_N$ we re-adjust to have $(\bar G^1, \bar H) \equiv (0,0)$ in \eqref{eq:wHNdef}. 
Adopting hereafter the latter setting, it is  more convenient to consider the solution of 
\eqref{eq:integ} under the joint law $\Ps$ of $\bx_0$, $\BB$ and the disorder $\BJ$ conditional only upon
$\cpt_\star :=\big\{ \forall i\geq 2: 
\partial_{x^i} H_{\BJ} (\bn)=0  \big\}$ (whose covariance is given by \eqref{eq:mod-nu} at $\bsigma=\bn$).
Indeed, our next proposition, whose proof is deferred to Section \ref{sec:use-CPs}, 
relates $\wP$ to $\Ps$ and further extends the conclusions of Proposition \ref{tight-self} to $\Ps$.

\begin{prop}\label{use-CPs}
Proposition \ref{tight-self} applies
for $\Ps$ instead of 
$\wP$. Further, for $\Psi$ and $\BZ_N$ of Corollary \ref{cor-self},
\begin{equation}\label{eq:to-CPs}
\lim_{N \to \infty} \sup_{|E|,|G| \le \alpha} \sup_{s_j,t_j \le T} 
|\wE \Psi(\BZ_N) - \Est \Psi(\BZ_N)| = 0 \,.
\end{equation}
Setting hereafter for the filtration $\Fa_u=\sigma(\Bx_v : v \in [0,u])$, $U_N \in \Ua_N$ 
and $\tau \in [0,T]$,
\begin{equation}\label{def:Uhat}
U_N(s,t|\tau) := \E_\star[U_N(s,t)|\Fa_\tau] \,, 
\end{equation}
Corollary \ref{cor-self} applies for $\Est$, with coordinates of $\BZ_N$ taken from 
$\Uas_N := \Ua_N \cup \{  U_N(\cdot|\tau), U_N \in \Ua_N, \tau \in [0,T] \}$.
\end{prop}
Our next result, whose proof is deferred to Section \ref{sec:unique}, shows that
the limiting dynamics of \eqref{eqR}--\eqref{eqH} admits at most one solution. 
\begin{prop}\label{uniqueness}
Let $T<\infty$ and $\D_T=\{s,t\in (\R^+)^2 : 0\le t\le s\le T\}$.
There exists at most one solution $(R,C,q,K,H)
\in \Ca_b^1( \D_T)^2 \ts \Ca_b^1([0,T])^3$ to 
\eqref{eqR}--\eqref{eqH} 
at $\beta=1$ with $C(s,t)=C(t,s)$ and boundary conditions
\begin{align}
R(s,s)&\equiv 1\quad \;\; \qquad\forall s\ge 0\label{bcR}\\ 
C(s,s)&=K(s) \qquad\forall s\ge 0\label{bcC}\\
K(0)&=1, \qquad q(0)=q_o \quad \mbox{ known.}\label{bcQ}
\end{align}
\end{prop}

Our next proposition, whose proof is deferred to Section \ref{sec:der}, plays here
the role of \cite[Prop. 1.3]{BDG2}.
\begin{prop}\label{prop-macro}
Let $U_N^a :=  \E_\star U_N$. Fixing $T<\infty$, any limit point of the
sequence $\Ua_N^a := \{ U_N^a, U_N \in \Ua_N \}$ with respect to
uniform convergence on $[0,T]^2$, satisfies the integral
equations in $\Ca_b([0,T]^2)$,
\begin{align}
C(s,t) =& \bar C(s,t) + \frac{q(s) q(t)}{\qs^2} \,, \qquad \qquad \qquad \qquad \qquad \qquad
\chi(s,t) = \bar\chi (s,t)  \,, 
\label{eq:ext} \\
Q(s) =& - f'(K(s)) q(s) + \qs^2 \bv'(q(s)) + V(s)\,, 
\qquad \qquad \qquad 
q(s) = q(0) + \int_0^s Q(u) du\,, 
\label{eq:qQdef} \\
\bar D(s,t)=& -f'(K(t)) \bar C(t,s) + \bar A(t,s) \,, \qquad \qquad \qquad\qquad
\bar E(s,t)= -f'(K(s)) \bar \chi(s,t) + \bar F(s,t) \,, \label{eq:dedef1} \\
\label{eq:Upsil}
\Upsilon(s,t) =& \nu(C(s,t)) 
  - \bar C (s,t) \frac{\nu'(q(s))\nu'(q(t))}{\nu'(\qs^2)} \,, 
\\
\Phi(s,t) =&  \nu'(C(s,t)) - \frac{\nu'(q(s))\nu'(q(t))}{\nu'(\qs^2)}\, 
\label{eq:Phi}\\
\Phi^1(s,u) =& q(u) \nu'(C(s,u)) - 
\bar C(s,u) \frac{\qs^2 \nu'(q(u))\nu''(q(s))}{\nu'(\qs^2)} \,,
\label{eq:Phi1}\\
\Psi(s,u) =&  \nu''(C(s,u)) \big(\bar D(s,u)+
\frac{q(s)}{\qs^2} Q(u) \big)  
- \frac{\nu'(q(s))\nu''(q(u))}{\nu'(\qs^2)} Q(u)  \,,
\label{eq:Psi}\\
\Psi^1(s,u) =&
\bar D(s,u)
\big[ \nu''(C(s,u)) q(u) -  \frac{\qs^2 \nu'(q(u))\nu''(q(s))}{\nu'(\qs^2)} \big]  
\nonumber \\
& + Q(u) 
[\nu'(C(s,u)) + \frac{q(s) q(u)}{\qs^2} \nu''(C(s,u)) 
-  
\bar C(s,u) \frac{\qs^2 \nu''(q(s))\nu''(q(u))}{\nu'(\qs^2)}  \big] \,,
\label{eq:Psi1}
\end{align}
\begin{align}
\bar C(s,t)=& \bar C(s,0)+ \bar \chi(s,t)+\int_0^t \bar D(s,u)du,
\qquad \qquad \qquad \quad \bar \chi(s,t)= s\wedge t+ \int_0^s \bar E(u,t)du,
\label{eq:C1-chi1}\\
V(s) =&  \Phi^1(s,s)-\Phi^1(s,0) - \int_0^s \Psi^1(s,u) du \,,
\label{eqQ1}\\
\bar A(t,s) =& \; \bar C(s,\tau) \Phi(t,\tau) - \bar C(s,0) \Phi(t,0)
- \int_0^\tau \Big\{ \bar D(s,u) \Phi (t,u) + \bar C (s,u) \Psi(t,u) \Big\} du \,,
\label{eqD}
\end{align}
\begin{align}
\bar F(s,t) =& \; \bar \chi (s,t) \Phi (s,s)  
- \int_0^{t \wedge s} \Phi(s,u) du  
- \int_0^s \bar E(u,t) \Phi(s,u)  du 
- \int_0^s \bar \chi(u,t) \Psi(s,u) du 
\,, \label{eqE}\\
\wH (s) =& \; \Upsilon (s,s) - \Upsilon (s,0)   
 - \int_0^s \Big\{\bar D (s,u) \Phi (u,s) + \frac{Q(u)}{\qs^2} \Phi^1 (u,s) \Big\} du \,,
\label{eqwH}
\end{align}
where $\tau=t \vee s$, 
subject to the symmetry $C(s,t)=C(t,s)$ and 
boundary conditions $q(0)=q_o$, $K(0)=1$,
$K(s)=C(s,s)$, $\bar E(s,0)=0$ for all $s$, and
$\bar E(s,t)=\bar E(s,s)$ for all $t \geq s$.
\end{prop}

Our final ingredient for Theorem \ref{thm-macro} is the 
following link between \req{eq:ext}--\req{eqwH} and 
\req{eqR}--\req{eqH}, whose proof we defer to Section \ref{sec:diff}. 
\begin{prop}\label{lem-diff}
Fixing $T<\infty$, if $(C,\chi,q,\wH) \in \Ca_b([0,T]^2;\reals^4)$
satisfies \req{eq:ext}--\req{eqwH}, with $\bv_\star(\cdot)$ instead of 
$\bv(\cdot)$,
subject to the symmetry and boundary conditions of Proposition \ref{prop-macro},
then  
$\chi(s,t)=\int_0^t R(s,u)du$ where
$R(s,t)=0$ for $t>s$, $R(s,s)=1$ and on $\Delta_T$ the bounded and
absolutely continuous functions $(C,R,q,\wH)$ satisfy 
the integro-differential equations \req{eqR}--\req{eqH} (at $\beta=1$).
\end{prop}

\medskip
\begin{proof}[Proof of Theorem \ref{thm-macro}:] 
Setting \abbr{wlog} $\beta=1$  and $\bsigma=\bn$,
recall from Proposition \ref{use-CPs} that all conclusions of Proposition \ref{tight-self} 
apply for $\Ps$. In particular, we thus have
pre-compactness of $(U_N^a, U_N \in \Ua_N) : [0,T]^2 \to \reals^{17}$  
in the topology of uniform convergence on $[0,T]^2$, 
implying the existence of limit points of this sequence as $N \to \infty$.
By Proposition \ref{prop-macro} any such limit point must be a solution of 
the integral equations \req{eq:ext}--\req{eqwH} with the stated symmetry and 
boundary conditions. Further, by Proposition \ref{lem-diff}, for $(E,G)=(\Es,\Gs)$
any such solution results with $(C,R,q,\wH)$ 
that satisfy the integro-differential
equations \req{eqR}--\req{eqH} (at $\beta=1$).
In view of Proposition \ref{uniqueness}
the latter system admits at most one solution per
given boundary conditions. Hence, we conclude that the sequence 
$(\chi_N^a,C_N^a,q_N^a,\wH_N^a)$ converges as $N \to \infty$, 
uniformly in $[0,T]^2$ to the unique 
solution determined by \req{eqR}--\req{eqH} subject to 
the appropriate boundary conditions. Thanks to 
Proposition \ref{use-CPs}, the same applies to $\wE [(\chi_N,C_N,q_N,\wH_N)]$.
Further, by \req{eq:asself} of Proposition \ref{tight-self}, almost surely
$|(\chi_N,C_N,q_N,\wH_N) - \wE (\chi_N,C_N,q_N, \wH_N)| \to {\bf 0}$ as $N \to \infty$, uniformly on $[0,T]^2$. 
In addition, $H_N(s)=\wH_N(s)+\bv_\star(q_N(s))$
(see \eqref{eq:Htilde} and the \abbr{lhs} of \eqref{eq:bnq}, \eqref{eq:wHNdef}).
Thus, the function $(\chi,C,q,H)$ determined from \req{eqR}--\req{eqH}
is also the unique almost sure uniform (in $s,t$)
limit of $(\chi_N,C_N,q_N,H_N)$, as stated in Theorem \ref{thm-macro}. 
The $L_p$ convergence follows by the uniform moments 
bounds of Proposition \ref{tight-self}, 
thereby completing the proof of the theorem.
\end{proof}

\subsection{Proof of Proposition \ref{tight-self}}\label{sec:self-average} 
We start by computing the covariances conditional on the event $\cpt_\star$, 
which replace here
the unconditional covariances of \cite[Lemma 3.2]{BDG2}.
\begin{lem}\label{kval}
For $\bv_p$, $p \ge 2$, of \eqref{def:vt} one has the 
following conditional expectations 
\begin{align}
\label{eq:bp1}
\E[J^{(p)}_{1 \cdots 1} | \cpt(E,G,\bn)] =  - b_p N^{1-\frac{p}{2}} 
\qs^{p} \, \langle \bv_p, (E,G) \rangle \,.
\end{align}	
Letting $\E_{\BJ}$ denote the expectation with respect to the Gaussian
law $\P_{\BJ}$ of the disorder $\BJ$, it follows that
for $\bv(\cdot)$ of \eqref{def:vt}, any 
$\Bx\in\Ca(\R_{+},\R^{N})$ which is independent of 
$\BJ$ and all $s,t\in[0,T]$, $i,j\in\{1,\ldots,N\}$,
	\begin{equation}
	\begin{aligned}
	\bar{H}(\Bx_t) & := 
	\E_{\BJ}\left\{ H_{\BJ} (\Bx_{t})\,|\,\cpt(E,G,\bn) \right\} =
	 - N \bv (q_N(t)) \,,\\
	\label{eq:Gtilde}
	\bar{G}^{i}(\Bx_t) & :=\E_{\BJ}\left\{ G^{i}(\Bx_{t})\,|\,\cpt(E,G,\bn) \right\} ={\bf 1}_{\{i=1\}} \sqrt N \qs \bv'(q_N(t)) \,.
	\end{aligned}
	\end{equation}	
Further, for 
$\cpt_\star =\big\{ \forall i\geq 2: 
\partial_{x^i} H_{\BJ} (\bn)=0  \big\}$, we have that  
$\E_{\BJ}\left\{G^{i}(\Bx_{t})\,|\,\cpt_\star \right\}=0$ for any $(t,i)$,
while 
\begin{equation}
	\label{eq:kval} 
\begin{aligned}
k_{ts}^{ij}(\Bx): & =\E_{\BJ}\left\{ G^{i}(\Bx_t)G^{j}(\Bx_s)\,|\,\cpt_\star \right\} 
= \partial_{x^i_t} \partial_{x^j_s} \{ \widetilde{k} (\Bx_t,\Bx_s) \}  \,, \\
	\widetilde{k} (\Bx_s,\Bx_t) & :=  
\E_{\BJ}\left\{ H_{\BJ} (\Bx_s) H_{\BJ} (\Bx_t) \,|\,\cpt_\star \right\} 
= N  \Upsilon_N(s,t)
 \,,
\end{aligned}
\end{equation}
for $\Upsilon_N(s,t)$ of \eqref{eq:PhiN-val}.
\end{lem}  
\begin{proof}
Fix two points $\bar\Bx,\bar\By\in\R^{N}$. Recall that \cite[Eq. (5.5.4)]{AT}
\begin{equation}
\label{eq:cov}
\begin{aligned}
\E\left\{ \partial_{\bar x^{i}}H_{\BJ}^{N}(\bar \Bx)H_{\BJ}^{N}(\bar \By)\right\}  & =\partial_{\bar x^{i}}\Cov\left(H_{\BJ}^{N}(\bar\Bx)H_{\BJ}^{N}(\bar\By)\right)\\
& =\bar y^{i}\nu'(N^{-1}\langle\bar\Bx,\bar\By\rangle)\,,\\
\E\left\{ \partial_{\bar x^{i}}H_{\BJ}^{N}(\bar\Bx)\partial_{\bar y^{j}}H_{\BJ}^{N}(\bar\By)\right\}  & =\partial_{\bar x^{i}}\partial_{\bar y^{j}}\Cov\left(H_{\BJ}^{N}(\bar\Bx)H_{\BJ}^{N}(\bar\By)\right)\\
& =\frac{\bar x^{j}\bar y^{i}}{N}\nu''(N^{-1}\langle\bar\Bx,\bar\By\rangle)+{\bf 1}_{\{i=j\}}\nu'(N^{-1}\langle\bar\Bx,\bar\By\rangle)\,.
\end{aligned}
\end{equation}
In particular, \ $\mathbf w = (\qs H_{\BJ}(\bn)\,,  
\sqrt N\partial_{x^1} H_{\BJ} (\bn))$ and 
$\mathbf z = \sqrt N(\partial_{x^i} H_{\BJ}(\bn))_{i>1}$ are independent. Therefore, from the well-known formula for conditional Gaussian distributions 
\cite[pages 10-11]{AT}, 
	\begin{align*}
	\bar{H}(\Bx_t) & =\Big\langle 
	\E_{\BJ}\{ H_{\BJ}(\Bx_{t}) \mathbf w \}  (\E_{\BJ}\{ \mathbf w^T \mathbf w \})^{-1}, (-N\qs E, -N\qs G) \Big\rangle 
	\end{align*}
	which by substitution yields the top line of \eqref{eq:Gtilde}. 
    Recall that $\bar G = - \nabla_{\bx} \bar H$ to complete
    the derivation of \eqref{eq:Gtilde}.
	The formula \eqref{eq:bp1}
	for the conditional expectations of $J^{(p)}_{1\cdots 1}$ is similarly verified 
	from 
	\begin{align*}
	\E\left\{ J_{1 \cdots 1}^{(p)}H_{\BJ}^{N}(\bn)\right\} &=  b_p \qs^p N^{\frac{p}{2}} \E\left\{ (J_{1 \cdots 1}^{(p)})^2 \right\}= b_p \qs^p N^{1-\frac{p}{2}}  \\
\E\left\{ J_{1 \cdots 1}^{(p)} (\partial_{x^1}H_{\BJ}^{N}(\bn)) \right\} &= 
b_p p \qs^{p-1} N^{\frac{p-1}{2}} \E\left\{ (J_{1 \cdots 1}^{(p)})^2 \right\}= b_p p \qs^{p-1} N^{\frac{1-p}{2}}.
	\end{align*}
	Next, recall that any centered Gaussian field, conditional on a
	linear map being zero, remains centered. In particular,
$\E_{\BJ}\left\{G^{i}(\Bx_{t})\,|\,\cpt_\star \right\}=0$ 	for any
choice of $\bn$ and $(t,i)$. Further, with  
$z_k =\sqrt N \partial_{x^k} H_{\BJ}(\bn)$ independent for different $k$,
the formula for the conditional covariance of $H_{\BJ}(\cdot)$, simplifies to 
		\begin{align*}
\widetilde{k}(\Bx_t,\Bx_s)  & =\E_{\BJ}\left\{ H_{\BJ}(\Bx_t) H_{\BJ} (\Bx_s) \right\} - \sum_{k=2}^{N} \E_{\BJ}\left\{ H_{\BJ} (\Bx_t) z_k \right\}
\left\{\E_{\BJ} z_k^2 \right\}^{-1}
\E_{\BJ}\left\{ H_{\BJ} (\Bx_s) z_k\right\},
		\end{align*}
		from which \eqref{eq:kval} follows by substitution (and comparison with the definition 
		 of $\Upsilon_N$  in \eqref{eq:PhiN-val}).
		\end{proof}
		
Preparing to adapt \cite[Section 2]{BDG2}, recall $K_N(t)=C_N(t,t)$ and set hereafter
$B_N(t):=\oneN \sum_{i=1}^N |B_t^i|^2$ and $G_N(t):=\oneN \sum_{i=1}^N |G^i(\Bx_t)|^2$. 
Using throughout 
the corresponding sup-norms $\|K_N\|_\infty:=\sup \{ K_N(t) : 0 \leq t \leq T \}$, 
$\|B_N\|_\infty$ and $\|G_N\|_\infty$ as well as 
the $N$-dependent disorder-norms 
\begin{equation}\label{dfn:Jnorm}
	\|\BJ\|^{N}_\infty:= \max_{1\leq p\leq m}\sup_{\|\mathbf u^i\|\leq 1, 1\leq i\leq p}\Big| \sqrt N^{-1} \sum_{1\leq i_k\leq N,1\leq k\leq p} N^{\frac{p-1}{2}}J_{i_1\cdots i_p}u_{i_1}^1\cdots  u_{i_p}^p\Big| 
\end{equation} 
of \cite[(2.1)]{BDG2}, we first mimic \cite[Proposition 2.1]{BDG2} about the 
strong solution $\Bx_t$ of \eqref{diffusion}.
\begin{prop}\label{existN} 
Assume that $f'$ is locally Lipschitz, satisfying \eqref{eq:fcond}.
Then, for any $N\in\N$, $\Bx_0$, $\BJ$ there exists a unique 
strong solution of \eqref{diffusion} for a.e. Brownian path $\BB$.
Denoting by $\P^N_{\bJ,\Bx_0}$ the (unique) law of 
$\{\BB_t,\Bx_t\}$ as $\Ca(\R^+,\R^{2N})$-valued variable, we have that 
for some $c$, $\kappa$ finite, all $N$, $z>0$, $\bJ$ and $\Bx_0$, 
\begin{equation}\label{eq:conc}
\P^N_{\bJ,\Bx_0}\Big(\sup_{t\in\R^+}
K_N(t) \ge K_N(0) + \kappa (1 + \|\bJ\|_\infty^N )^c + z \Big)
\le  e^{-z N} \,.
\end{equation}
In particular, for some $D_o(k,M)$ finite, any $k,M$ and all $N$,
\begin{equation}\label{eq:K-unif-bd}
\sup_{\{\BJ,\Bx_0 : K_N(0) + \|\BJ\|^N_\infty \le M\}} \; \Big\{ \E^N_{\BJ,\Bx_0}  \big[
\sup_{t \in \R^+} K_N(t)^k \big]  \Big\} \le D_o(k,M) \,.
\end{equation}
Further, for any finite positive $\qs$, $k$ and $\alpha$ 
\begin{equation}\label{eq:ger2}
\sup_{|E|,\,|G|\leq \alpha}
\sup_N \wE \big[ \, (||\bJ||_\infty^{N})^k \, \big] < \infty 
\end{equation}
and there exist finite $\tilde \kappa \ge 1$ such that for any $t \ge 0$,
\begin{equation}\label{eq:2712-02}
\sup_{|E|,\,|G|\leq \alpha}\sup_N \wP \big[
\|\BJ\|_\infty^N >\tilde\kappa+t \big]
	   \leq \tilde \kappa e^{-Nt^2/\tilde\kappa}\,.
\end{equation}
Consequently, for 
any $|q_o| \le \qs$ positive, finite $k$ and $\alpha$,
\begin{equation}\label{eq:conc2}
\sup_{|E|,\,|G|\leq \alpha}
\sup_{N} \wE \big[ \sup_{t\in\R^+} K_N(t)^{k} 
 \big] <\infty
\end{equation}
and for any finite $L$ there exist $z=z(L)$ finite such that
\begin{equation}\label{eq:ktail}
\sup_{|E|,\,|G|\leq \alpha} \sup_N
\wP \big[ \sup_{t\in\R^+} K_N(t) \ge z \, 
\big] \le  2 \tilde \kappa e^{-L N} \,.
\end{equation}
\end{prop}
\begin{proof}
From \cite[Proposition 2.1]{BDG2} we have the existence of 
a unique strong solution as well as the bound \eqref{eq:conc} 
(while stated in \cite{BDG2} for a.e. $\bJ$, $\Bx_0$, 
examining their proof we see that it holds for all $\bJ$ and $\Bx_0$). 
Clearly, \eqref{eq:K-unif-bd} and
\eqref{eq:ger2} are immediate consequences of \eqref{eq:conc} and
 \eqref{eq:2712-02}, respectively. Further, taking $\Bx_0 \in \SN$ amounts to $K_N(0)=1$, yielding  
\eqref{eq:conc2} and \eqref{eq:ktail} upon combining 
\eqref{eq:conc} with \eqref{eq:ger2} and \eqref{eq:2712-02}, respectively. 
Turning to the only remaining task, of proving \eqref{eq:2712-02}, 
recall \cite[(B.7)]{BDG2} that for some $\tilde \kappa$ and all $t \ge 0$,
\begin{equation}\label{eq:B7-BDG2}
\sup_N\P \big[
\|\BJ\|_\infty^N >\tilde\kappa+t
\big]
\leq \tilde \kappa e^{-Nt^2/\tilde\kappa} \,.
\end{equation}
Since $\|\BJ\|_\infty^N$ is a symmetric, convex function of $\BJ$, by Anderson's inequality \cite[Corollary 3]{Anderson}, the bound \eqref{eq:B7-BDG2} holds when $\BJ$ is 
replaced by the centered Gaussian vector $\BJ_o$ having the  
law $\gamma_N^{(0,0,\qs)}$. Further, conditionally 
on $\cpt(E,G,\bn)$, we have that $\BJ = \BJ_o+\bar{\BJ}_{E,G}$ for
the non-random vector $\bar{\BJ}_{E,G}:=\E[\BJ\,|\,\cpt(E,G,\bn)]$. The only non-zero
entries of $\bar{\BJ}_{E,G}$ correspond to $\{J^{(p)}_{1 \cdots 1}\}$ and 
are given by \eqref{eq:bp1}. Consequently,  
\begin{equation}\label{eq:Lip-barJ}
\|\bar{\BJ}_{E,G} \|_\infty^N 
=\max_{2\leq p\leq m} \{ | b_p  \qs^{p} 
\, \langle \bv_p, (E,G) \rangle | \} \,,
\end{equation}
is bounded, uniformly over $|E|,|G| \le \alpha$ by some 
$\widehat{\kappa}(\alpha,\qs)$ finite. In conjunction with 
the triangle inequality for $\|\cdot\|_\infty^N $, this yields  
\eqref{eq:2712-02} (upon adding $\widehat{\kappa}$ to $\tilde\kappa$).
\end{proof}

The same reasoning as in proving \cite[Proposition 2.3]{BDG2}, but with
\eqref{eq:ger2}--\eqref{eq:ktail} of Proposition \ref{existN} replacing \cite[Eqn.
(2.12), (B.7), (2.13), (2.3)]{BDG2}, respectively, yields for $U_N \in \Uao_N$ both
\eqref{eq:u-mombd} and the stated pre-compactness. Along the way we also 
find that for some
$M=M(L,T,\alpha)<\infty$ the subsets
\begin{equation}\label{dfn:LMN}
\CL_{N,M} := \{
(\Bx_0, \BJ, \BB)\in \Ea_N \,:\, 
\|\BJ\|^{N}_\infty + \|B_N\|_\infty + \| K_N\|_\infty + \| G_N \|_\infty
\le M \,\} 
\end{equation} 
of $\Ea_N$ are such that for any finite $L,T,\alpha$ and all $N$,
\begin{equation}\label{eq:BDG2-217}
\sup_{|E|,|G| \le \alpha} \wP (\CL_{N,M}^c) \le M e^{-L N} \,.
\end{equation}
Next, similarly to \cite[(2.10)]{BDG2}, 
\begin{equation}\label{eq:lipH}
\frac{1}{\sqrt{N}} |H_{\BJ}(\bx)-H_{\BJ}(\widetilde{\bx})|
\le c \, \|\BJ\|_\infty^N \big(1 + (N^{-1}\|\bx\|^2)^r \big)
\big(1 + (N^{-1} \|\widetilde{\bx}\|^2)^r \big) \|\bx - \widetilde{\bx}\| \,,
\end{equation}
for $r=(m-1)/2$, $c = m \sqrt{\nu'(1)}$ and any $\bx, \widetilde{\bx} \in \reals^N$. In particular,
\begin{equation}\label{eq:lipHN}
|H_N(t)-H_N(t')| \le c \, \|\BJ\|_\infty^N (1 + K_N(t)^r)
(1 + K_N(t')^r) \frac{\|\bx_t - \bx_{t'}\|}{\sqrt{N}} \,.
\end{equation}
The uniform moment bound \eqref{eq:u-mombd} then   
extends to all $\Ua_N$ since $q^2_N(s) \le \qs^2 C_N(s,s)$ and 
$\wQ^2_N(s) \le \qs^2 G_N(s)+\qs^4 (\bv'(q_N(s))^2$, with the locally 
Lipschitz $f'(\cdot)$, $\nu''(\cdot)$ 
and $\bv'(\cdot)$ having at most a polynomial growth.
In addition, from \cite[(2.18)]{BDG2}  
adapted to our setting of $\wP$, we have 
for any $\e>0$, some  
$L'(\d,\e) \to \infty$ as $\d \to 0$, and all $N$, 
\begin{align*}
\sup_{|E|,|G| \le \alpha} \wP \big[ \sup_{|t-t'|<\delta} \{|q_N(t)-q_N(t')|\}
> \qs \sqrt{\e} \,
 \big] & \le e^{-L'(\d,\e) N} \\
\sup_{|E|,|G| \le \alpha} \sup_N \wE \big[ \sup_{|t-t'|<\delta} \,
|q_N(t)-q_N(t')|^4 \,
 \big] & \le L'(\d,\e)^{-1} \,.
\end{align*}
The same holds also for $\wH_N(\cdot)$ (see \eqref{eq:lipHN}), and for
$\wQ_N(\cdot)$ (c.f. \cite[display preceding (2.18)]{BDG2}).
Such bounds yield the equi-continuity of $q_N(\cdot)$,
$\wQ_N(\cdot)$ and $\wH_N(\cdot)$ (a.s. and in expectation), from which 
we deduce the pre-compactness, first of $q_N,\wQ_N,\wH_N$,
then of $Q_N,\bar D_N,\bar E_N$
and finally of $\Upsilon_N$,$\Phi_N$,$\Phi_N^1$,$\Psi_N$, $\Psi^1_N$ 
(by the uniform moments control \eqref{eq:u-mombd} 
and the Arzela-Ascoli theorem). In particular, this way we have further 
established that for some $\widetilde{L}(\d,\e) \to \infty$ as $\d \to 0$, any $\e>0$ and 
$U_N \in \Ua_N$
\begin{equation}\label{eq:quant-equi}
\begin{aligned}
\sup_{|E|,|G| \le \alpha} 
\wP (\sup_{|s-s'|+|t-t'|<\d} |U_N(s,t)-U_N(s',t')|>\e) &\leq e^{-
\widetilde{L}(\d,\e) N}\,, \\
\sup_{|E|,|G| \le \alpha} \sup_N 
\sup_{|s-s'|+|t-t'|<\d} |\wE \, U_N(s,t)- \wE \, U_N(s',t')| 
&\leq \widetilde{L}(\d,\e)^{-1} \,.
\end{aligned}
\end{equation}

Turning to the self-averaging property \eqref{eq:asself},  similarly to
\cite[Proposition 2.4]{BDG2} our proof relies on the
following pointwise Lipschitz estimate on $\CL_{N,M}$ of \eqref{dfn:LMN}.
\begin{lem}\label{U are Lipschitz} 
Let $\bx,\widetilde{\bx}$ be the two strong solutions
of \eqref{diffusion} constructed from $(\Bx_0,\BJ,\BB)$
and $(\widetilde{\Bx}_0,\widetilde{\BJ}, \widetilde{\BB})$, respectively.
If $(\Bx_0,\BJ,\BB)$
and $(\widetilde{\Bx}_0,\widetilde{\BJ},\widetilde{\BB})$ are both in 
$\CL_{N,M}$, then
we have the Lipschitz estimate for each $U_N \in \Ua_N$,
\begin{equation}\label{eq:lippr}
\sup_{s,t \le T}|U_N(s,t)- \widetilde{U}_N(s,t)| 
\le \frac{D(M,T)}{\sqrt{N}} 
\|(\Bx_0, \BJ, \BB) - (\widetilde{\Bx}_0,\widetilde{\BJ},\widetilde{\BB})\|\,,
\end{equation}
where the constant $D(M,T)$ depends only on $M$ and $T$ and not on $N$. 

Further, for $e_N(s) := N^{-1} \|\bx_s- \widetilde \bx_s\|^2_2$ any $N$ and $T$, 
if $\widetilde \BB = \BB$ and $(\widetilde \Bx_0,\widetilde \BJ) \to (\Bx_0,\BJ)$, then 
\begin{equation}\label{eq:diffeo}
\E\big[ \,  1 \wedge \| e_N \|_\infty \, |\,  \widetilde \BJ, \BJ, \widetilde \Bx_0, \Bx_0 \big] \to 0 \,.
\end{equation}
\end{lem}
\begin{proof} For $U_N \in \Uao_N$ the bound \eqref{eq:lippr}
is precisely the statement of \cite[Lemma 2.7]{BDG2}, while 
for $U_N=q_N$ it follows upon taking the square-root 
of the bound 
\begin{equation}\label{eq:lipbound}
\| e_N  \|_\infty 
\le \frac{D_1 (M,T)}{N} \|(\Bx_0, \BJ, \BB) 
- (\widetilde{\Bx}_0,\widetilde{\BJ},\widetilde{\BB})\|^2
\end{equation}
from \cite[Lemma 2.6]{BDG2}. Further, while proving \cite[Lemma 2.7]{BDG2}
it is shown that on $\CL_{N,M}$ 
\[ 
\|G(\Bx_s)-\widetilde G(\widetilde{\Bx}_s)\|_2  \le
D_2(M,T)
\|(\Bx_0, \BJ, \BB) - (\widetilde{\Bx}_0,\widetilde{\BJ},\widetilde{\BB})\|
\]
(where $\widetilde{G}(\cdot):=
- \nabla H_{\widetilde \BJ}(\cdot)$, 
see \cite[Page 636]{BDG2}). Utilizing \eqref{eq:lipH} instead of 
\cite[(2.10)]{BDG2} yields the same bound for 
$\frac{1}{\sqrt{N}} |H_{\BJ}(\bx_s)-H_{\widetilde{\BJ}}(\widetilde{\bx}_s)|$.
Recall \eqref{dfn:LMN} that 
$\|q_N\|_\infty \le \qs \|K_N\|_\infty^{1/2} \le \qs \sqrt{M}$ 
on $\CL_{N,M}$, which thus
in view of \eqref{eq:Gtilde} for the locally Lipschitz $\bv'(\cdot)$, 
thus results with \eqref{eq:lippr} holding for $U_N=\wQ_N$ and
$U_N=\wH_N$. Similarly, 
having $f'(\cdot)$, $\nu''(\cdot)$ locally Lipschitz 
and $\|K_N\|_\infty \le M$ on $\CL_{N,M}$,
extends the validity of \eqref{eq:lippr} first to $U_N \in \{Q_N,\bar D_N,\bar E_N\}$,
then also to $U_N \in \{\Upsilon_N,\Phi_N,\Phi_N^1,\Psi_N,\Psi^1_N\}$. 

In case $\widetilde \BB = \BB$  we see from \cite[Proof of Lemma 2.6]{BDG2}
that \eqref{eq:lipbound} holds
when $\|\BJ\|^N_\infty + \| K_N \|_\infty + \| \widetilde K_N \|_\infty \le M$.
With $(\widetilde \Bx_0,\widetilde \BJ) \to (\Bx_0,\BJ)$, the \abbr{rhs} of \eqref{eq:lipbound} decays to zero
and $\widetilde K_N(0) + \|\widetilde \BJ\|^N_\infty$ is uniformly bounded. Such uniform boundedness implies 
in view of \eqref{eq:K-unif-bd} that 
as $M \to \infty$, 
\[
\P ( \|\BJ\|_\infty^N + \| K_N \|_\infty + \|\widetilde K_N\|_\infty  > M \,  | \, \widetilde \BJ, \BJ, \widetilde \Bx_0, \Bx_0 ) \to 0 \,,
\]
uniformly in $(\widetilde \Bx_0,\widetilde \BJ )$, from which we deduce by bounded convergence 
that  \eqref{eq:diffeo} holds. 
\end{proof}

We next verify that $\wP$ satisfies the Lipschitz concentration of 
measure, as in \cite[Hypothesis 1.1]{BDG2}, uniformly 
over $|E|,|G| \le \alpha$.
\begin{prop}\label{Hyp1.1}
For some $C>0$, any $(E,G,\qs)$, 
function 
$V: \Ea_{N} \mapsto \R$ of Lipschitz constant $K$ and all $\rho>0$, 
\begin{equation}\label{eq:Hyp1.1}
\wP \left\{ |V-\wE V| \geq\rho\right\} 
\leq C^{-1}\exp\left(-C\rho^2/K^2\right)\,.
\end{equation}
\end{prop}
\begin{proof} 
Assume first that $\wE V=0$. Recall that 
$\wP = \mu_{\bn}^{q_o} \otimes \gamma_N^{(E,G,\qs)} \otimes P_N$. 
Denoting a generic point in $\Ea_{N}$ by $(\Bx_{0},\BJ,\BB)$,
let $\E_{\Bx_0}$ denote the expectation \abbr{wrt} $\mu_{\bn}^{q_o}$
and the variable $\Bx_{0}$ only, and for fixed $\Bx_{0}$,
let $\wP_{\BJ,\BB}=\gamma_N^{(E,G,\qs)} \otimes P_N$. By 
conditioning on $\Bx_{0}$,
\begin{equation}
\begin{aligned}
\wP (V>\rho) & 
\leq\E_{\Bx_0}\wP_{\BJ,\BB}(V-\wE_{\BJ,\BB} V>\rho/2)+
\P_{\Bx_0}(\wE_{\BJ,\BB}V>\rho/2).
\end{aligned}
\label{eq:0302-01}
\end{equation}
For any fixed $\Bx_{0}$, $(\BJ,\BB)\mapsto V(\Bx_{0},\BJ,\BB)$
has Lipschitz constant $K$ \abbr{wrt} the norm 
\begin{equation*}
\|(\BJ, \BB)\|^2=  \sum_{p=2}^m \;\; \sum_{1 \leq i_1 \leq \ldots \leq i_{p} \leq N} 
(N^{\frac{p-1}{2}}J_{i_1\cdots i_{p}})^2
+ \sup_{0 \leq t \leq T}\sum_{i=1}^N(B_t^i)^2.
\end{equation*}
Next, set $\P_{\BJ,\BB}:=\gamma_{N}\otimes P_{N}$
for the unconditional 
Gaussian law $\gamma_N$ of $\BJ$,
and $W(\Bx_{0},\BJ,\BB):=(\Bx_{0},\bar{W}(\BJ),\BB)$,
for the orthogonal projection $\bar{W}$ to the affine subspace of
$\mathbb{R}^{d(N,m)}$ defined by
$\cpt(E,G,\bn)$.
The composition $V\circ W$ necessarily has at most the Lipschitz
constant $K$. Hence, for some $C>0$, any $N$, $V(\cdot)$, $\rho>0$ and 
all $\Bx_{0}$, by the concentration of measure of the Gaussian
measure (see, e.g. \cite{lesToulousains}),
\begin{align*}
\wP_{\BJ,\BB}(V-\wE_{\BJ,\BB}V>\rho/2)  =\P_{\BJ,\BB}(V\circ W-\E_{\BJ,\BB}V\circ W>\rho/2)
\le C^{-1}\exp(-C\rho^2/K^{2})\,.
\end{align*}
Further, by Jensen's inequality,
$\Bx_{0}\mapsto\wE_{\BJ,\BB} V
$
has Lipschitz constant $K$ \abbr{wrt} the Euclidean norm on $\mathbb{R}^{N}$.
Moreover, $\E_{\Bx_{0}}\wE_{\BJ,\BB}V=\wE V=0$, so by the 
concentration of measure of the uniform measure on the sphere
\cite[Theorem 1.7.9]{Levyinequality}, for some $C>0$ and any $N$, $V(\cdot)$, 
$\rho>0$,
\[
\P_{\Bx_{0}}(\wE_{\BJ,\BB}V>\rho/2)<C^{-1}\exp(-C\rho^{2}/K^{2})\,.
\]
Combining the above we deduce from (\ref{eq:0302-01}) that for some $C>0$
any $K$-Lipschitz $V$ and $\rho>0$, 
\[
\wP (V>\rho) \le C^{-1}\exp(-C\rho^{2}/K^{2}) \,.
\]
Considering this bound for $\pm (V-\wE V)$ yields \eqref{eq:Hyp1.1}.
\end{proof}

Equipped with Lemma \ref{U are Lipschitz} and Proposition \ref{Hyp1.1} 
we establish  \eqref{eq:asself} via the same reasoning as in  
\cite[proof of Proposition 2.4]{BDG2}. Specifically, fixing $(s,t) \in [0,T]^2$, 
we use \cite[Lemma 2.5]{BDG2} to extend (thanks to \eqref{eq:BDG2-217}),
the tail control of Proposition \ref{Hyp1.1} to $V=U_N(s,t)$ for $U_N$ satisfying 
only \eqref{eq:u-mombd} and \eqref{eq:lippr}. With 
constants $C$, $K$, $M(L)$, $D=D(M(L),T)$ in \cite[(2.21)]{BDG2} which are
independent of $s,t$, $\rho$, $N$ (and uniform over $|E|,|G| \le \alpha$),
we get by the union bound that \eqref{eq:asself} holds whenever
the supremum is restricted to $s,t$ in some (arbitrary) finite 
subset $\Aa$ of $[0,T]^2$. The preceding
quantitative equi-continuity control of \eqref{eq:quant-equi},
further allow for strengthening to the full summability result 
\eqref{eq:asself} by considering a finite $\d$-net
$\Aa$ of $[0,T]^2$ (say with $\d>0$ small, so  
$\widetilde{L}(2\d,\rho/3) > 3/\rho$).

\subsection{Proof of Proposition \ref{use-CPs}.}\label{sec:use-CPs}
Under both $\wP$ and $\Ps$ the vector $\bJ$ has the Gaussian law $\P_\BJ$ 
of independent coordinates, conditioned on $\cpt_\star$. Indeed, 
the only difference between $\wP$ and $\Ps$ is that $\wP$ imposes on $\BJ$ 
an \emph{additional} conditioning via $\cpt_1:=\{H_{\BJ}(\bn)=\partial_{x^1} H_{\BJ}(\bn)=0\}$.
Having a conditional law for $\BJ$ enters twice throughout the 
whole derivation of Proposition \ref{tight-self} (via Propositions \ref{existN} and 
\ref{Hyp1.1}):
first in upgrading \eqref{eq:B7-BDG2} from $\P$ to $\wP$ 
via Andreson's inequality, then in proving Proposition \ref{Hyp1.1}
by representing the conditional disorder as $\bar{W}(\BJ)$ 
(for some orthogonal projection $\bar{W}$). Both arguments are
applicable also for $\Ps$ (namely, without 
conditioning on $\cpt_1$), hence so are all the conclusions 
of Proposition \ref{tight-self} (and of Proposition \ref{existN}).

Turning to \eqref{eq:to-CPs}, we set 
$\widetilde{J}_p := N^{\frac{p-1}{2}} J^{(p)}_{\{1 \cdots 1\}}$, noting 
that $\cpt_\star$ is independent of the standard Gaussian vector 
$\underline{\widetilde{J}} := (\widetilde{J}_p, 2 \le p \le m)$, whereas
\begin{equation}\label{eq:cpt1}
\cpt_1 = \bigg\{\underline{\widetilde{J}} : \sum_{p=2}^m b_p \widetilde{J}_p \qs^p =
\sum_{p=2}^m b_p p \widetilde{J}_p \qs^{p-1} = 0 \bigg\} \,.
\end{equation}
Denoting by $\widetilde{W}$ the orthogonal projection 
sending $\underline{\widetilde{J}}$ to the linear subspace 
determined by \eqref{eq:cpt1}, leaving the remainder of 
$(\bx_0,\BJ,\BB)$ unchanged, we thus have that 
$\wE V = \Est V \circ \widetilde{W}$
for any $V: \Ea_N \mapsto \R$. Further, with 
\[
\| \widetilde{W}(\bx_0,\BJ,\BB)- (\bx_0,\BJ,\BB) \| \le  
\|\underline{\widetilde{J}}\| \le \frac{\sqrt m}{\sqrt N} \|\BJ\|^N_\infty \,,
\]
we deduce from \eqref{eq:lippr} that when $(\Bx_0,\BJ,\BB)$
and $\widetilde{W}(\Bx_0,\BJ,\BB)$ are both in $\CL_{N,M}$
\begin{equation}\label{eq:lipp-psi}
\sup_{E,G} \sup_{s_j,t_j \le T} \|\BZ_N \circ \widetilde{W} - \BZ_N\|_2
\le \frac{D'}{N} \,,
\end{equation}
where $D':= \sqrt{\ell m} D(M,T) M$. 
With $|\Psi(z)| \leq M' \|z \|_k^k$ and
$c_r$ denoting the finite Lipschitz constant of $\Psi(\cdot)$
(with respect to $\|\cdot\|_2$), on the compact set 
$\Gamma_r := \{ z: \|z\|_k \le r\}$, we thus have that for any 
$E,G$, $M,r<\infty$ and $s_j,t_j \le T$,
\begin{align*}
& |\wE \Psi(\BZ_N) - \Est \Psi(\BZ_N)| \le 
\Est |\Psi(\BZ_N \circ \widetilde{W}) - \Psi(\BZ_N)| \\
&\le M' \wE [\|\BZ_N\|_k^k ({\bf 1}_{\CL_{N,M}^c} + {\bf 1}_{\|\BZ_N\|_k > r})]
+ M' \Est [\|\BZ_N\|_k^k ({\bf 1}_{\CL_{N,M}^c} + {\bf 1}_{\|\BZ_N\|_k >r})]
+ c_r \frac{D'}{N} \,,
\end{align*}
The last term on the \abbr{rhs} vanishes when $N \to \infty$. 
Recall \eqref{eq:u-mombd}, that both
$\wE \|\BZ_N\|_k^{2k}$ and $\Est \|\BZ_N\|_k^{2k}$ are bounded,
uniformly over $|E|,|G| \le \alpha$ and $s_j,t_j \le T$.  
Thus, by Cauchy-Schwartz, 
considering \eqref{eq:BDG2-217}
for $\wP$ and $\Ps$, the contribution to the \abbr{rhs} 
from the pair of terms with $\CL_{N,M}^c$ also vanishes as $N \to \infty$.
Now, to arrive at \eqref{eq:to-CPs}, simply combine \eqref{eq:u-mombd} 
with Markov's inequality, to deduce that $\wP(\|\BZ_N\|_k > r)+\Ps(\|\BZ_N\|_k > r) \to 0$
as $r \to \infty$, uniformly in $N$, $|E|,|G| \le \alpha$ and $s_j,t_j \le T$.
Finally, combining \eqref{eq:self-av-Psi} and \eqref{eq:to-CPs} we deduce that
\begin{equation}\label{eq:self-av-star}
\lim_{N \to \infty} 
\sup_{|E|,|G| \le \alpha} \sup_{s_j,t_j} |\Est \Psi(\BZ_N) - \Psi(\Est \BZ_N)| = 0 
\end{equation}
whenever the coordinates of $\BZ_N$ are from $\Ua_N$. Clearly,   
$\Est | U_N(\cdot|\tau) - \Est U_N| \le \Est |U_N - \Est U_N|$ and 
$\Est U_N(\cdot|\tau) = \Est U_N$ for any $U_N \in \Ua_N$, $\tau \in [0,T]$,
thereby extending the 
validity of \eqref{eq:self-av-star} to coordinates of $\BZ_N$ from 
$\Uas_N$.

\subsection{Proof of Proposition \ref{uniqueness}.}\label{sec:unique}
Fixing $T<\infty$ note that $H(\cdot)$ does not affect $(R,C,q,K)$.
With $H(\cdot)$ uniquely determined by $(R,C,q)$
via \eqref{eqH}, it suffices to prove the uniqueness of the solution 
$(R,C,q,K)$ of the reduced system 
(S):=(\ref{eqR},\ref{eqC},\ref{eqq},\ref{eqZ}). To this end, fixing 
two solutions $(R,C,q,K)$, $(\tilde R,\tilde C,\tilde q,\tilde K)$,
of (S) at $\beta=1$ 
of the same boundary condition (BC):=(\ref{bcR},\ref{bcC},\ref{bcQ}), 
let
\[
\underline{\Delta U}
:= (\Delta R,\Delta C,\Delta q,\Delta K) = |(R,C,q,K)-(\tilde R,\tilde C,\tilde q,
\tilde K)|\,.
\]
From (BC) we have that $\Delta C(s,s) = \Delta K(s)$ and 
$\Delta R(s,s) \equiv 0$,
$\Delta K(0) = \Delta q(0)=0$. 
Denoting all constants by $M$ (which may depend on $T$ and the 
uniform bound on both solutions), even though
they may change from line to line, we arrive at 
$\underline{\Delta U} \equiv 0$ by adapting the
Gronwall's type argument leading to \cite[Proposition 4.2]{BDG2}.
To this end, \eqref{eqR} yields, exactly as in 
\cite[(4.9)]{BDG2} that for all $(s,t) \in \Delta_T$,
\begin{align}\label{ineqR}
\Delta R (s,t) & \le M \int_t^s \Delta K(u) du +
M \int_{t\le t_2\le t_1\le s} \Delta C(t_1,t_2)
dt_1 dt_2 := I_2(s,t) + I_8(s,t) \,.
\end{align}
Next, integrating \eqref{eqq} yields that
\begin{align*}
q(t) = q(0) -& \int_0^t f'(K(u)) q(u) du  + \int_0^t \qs^2 \bv'_\star (q(u)) du \\
+& \int_0^t du \int_0^u R(u,v) \Big[ q(v) \nu''(C(u,v)) - 
 \frac{\qs^2 \nu'(q(v)) \nu''(q(u))}{\nu'(\qs^2)}
\Big] dv 
 \,.
\end{align*}
The same identity holds for $(\tilde R,\tilde C,\tilde q,\tilde K)$. With
$f'(\cdot)$, $\bv'_\star(\cdot)$ locally Lipschitz, considering the
difference between that identity for our two uniformly bounded 
on $\D_T$ solutions of (S), yields that
\begin{equation*}
\Delta q(t) \le M \Big[ \int_0^t \Delta q(u) du 
+ \int_0^t du \int_0^u \Delta R(u,v) dv
+ \int_0^t du \int_0^u \Delta C(u,v) dv 
+ \int_0^t \Delta K(u) du \Big] \,.
\end{equation*}
By Gronwall's lemma, upon suitably increasing the value of $M$ 
we can eliminate the first term on the \abbr{rhs}, whereas  
by \eqref{ineqR} the second term on the \abbr{rhs}
is controlled by the remaining two terms. Hence,
\begin{equation}\label{ineqq2}
\Delta q(t) \le I_2(t,0)+I_8(t,0)  
\,, \qquad \forall t \in [0,T] \,.
\end{equation} 
Likewise, integrating \eqref{eqC} yields that each solution 
of (S) satisfies for $s \ge t$,
\begin{align}
C(s,t)=K(t)&-\int_t^s f'(K(u)) C(u,t) du
+\int_t^s du \int_0^t dv \nu'(C(u,v)) R(t,v)
\nonumber \\
&+
\int_t^s du\int_0^t dv R(u,v) \nu''(C(u,v)) C(t,v)
+\int_t^s du \int_t^u dv R(u,v) \nu''(C(u,v)) C(v,t) \nonumber \\
&- q(t) \int_t^s du \frac{\qs^2 \nu''(q(u))}{\nu'(\qs^2)}
 \int_0^t dv R(u,v) \nu'(q(v)) \nonumber \\
& - \int_t^s du \frac{\nu'(q(u))}{\nu'(\qs^2)} \int_0^t dv \nu'(q(v))  R(t,v)   
+  q(t) \int_t^s \bv'_\star (q(u)) du  \,.
\label{eq:extra-q}
\end{align}
By \eqref{ineqq2}, the terms on the \abbr{rhs} which involve $q(\cdot)$, 
contribute to $\Delta C(s,t)$ at most
\begin{align*}
& M \Big[ \Delta q(t) + \int_t^s \Delta q(u) du + 
\int_t^s du \int_0^t \Delta q(v) dv + \int_t^s du\int_0^t dv \Delta R (u,v) 
+ \int_t^s du \int_0^t dv \Delta R(t,v) \Big] 
\\ & 
\le I_2(s,0) + I_8(s,0) + I_7(s,t)+I_6(s,t) 
\end{align*}
(see \eqref{b1} for $I_6$ and $I_7$). Utilizing 
\cite[(4.10)]{BDG2} to bound the effect on $\Delta C(s,t)$ from the
rest of \eqref{eq:extra-q}, yields
\begin{align}
\Delta C (s,t) \le 
& M \Big[ \Delta K(t)+\int_0^s \Delta K(u)du
+\int_t^s \Delta C (u,t) du
+\int_t^s du\int_0^t dv \Delta C (u,v)\nonumber\\
&+\int_t^s du\int_0^t dv \Delta C (t,v)
+\int_t^s du\int_0^t dv \Delta R (t,v)
+\int_t^s du\int_0^t dv \Delta R(u,v)\nonumber\\
&+\int_0^s du\int_0^u dv \Delta C (u,v)+
\int_t^sdu\int_t^u dv \Delta C (v,t)+
\int_t^sdu\int_t^u dv \Delta R (u,v)\Big]\nonumber\\
:= & I_1(s,t)+I_2(s,0)+I_3(s,t) + \cdots + I_7(s,t) + I_8(s,0) + I_9(s,t) + I_{10}(s,t) \,.\label{b1}
\end{align}
Similarly, by \req{eqZ} we have for each solution of (S) and any $t \in [0,T]$, 
\begin{align}\label{eq:extra-K}
K(t) - K(0) - t &= - 2\int_0^t f'(K(u)) K(u)du +
2\int_0^t du \int_0^u dv \psi(C(u,v)) R(u,v) \nonumber \\
& - \frac{2}{\nu'(\qs^2)} \int_0^t du \, \psi(q(u)) \int_0^u dv \, \nu'(q(v)) R(u,v)  
+  2 \int_0^t q(u) \bv'_\star (q(u))  du \,. 
\end{align}
Clearly, the terms involving $q(\cdot)$ on the \abbr{rhs} 
contribute to $\Delta K(t)$ at most
$
M \int_0^t \Delta q(u) du 
 + I_{10}(t,0)$.  Further, with $\Delta K(0)=0$, utilizing \eqref{ineqq2} 
and bounding the effect of the rest of \eqref{eq:extra-K} as in 
 \cite[(4.11)]{BDG2}, yields here 
\begin{eqnarray}
\Delta K (t) \le I_2(t,0) + I_8(t,0) + I_{10}(t,0) \,.
\label{ineqK}
\end{eqnarray}
We follow the derivation of \cite[(4.13)]{BDG2}, by first plugging \eqref{ineqR} 
into \eqref{ineqK} to eliminate $I_{10}(t,0)$, then by Gronwall's lemma
eliminating $I_2(t,0)$. 
Setting $\mathsf{D} (s):=\int_0^s \Delta C (s,v) dv$, we thereby
get, as in \cite[(4.13)]{BDG2}, that 
\begin{equation}
\Delta K (t) \le  I_8(t,0) = M \int_0^t \mathsf{D}(u) du  \,.
\label{ineqK4}
\end{equation}
Plugging \eqref{ineqK4}
into \eqref{ineqR} and \eqref{ineqq2}, yields in turn that 
\begin{equation}\label{ineqR4}
\Delta R(s,t) \le M \int_0^s \mathsf{D}(u) du \,, \qquad
\Delta q(s) \le M \int_0^s \mathsf{D}(u) du \,, \qquad \forall (s,t) \in \D_T \,. 
\end{equation}
With \eqref{b1} differing from \cite[(4.10)]{BDG2} only in having
$I_2(s,0)+I_8(s,0)$ instead of $I_2(s,t)+I_8(s,t)$, upon integrating
both sides of \eqref{b1} with respect to $t \in [0,s]$, 
we deduce from \eqref{ineqK4}-\eqref{ineqR4}, exactly as in \cite[Page 652]{BDG2}, that 
\[
\mathsf{D} (s) \le M \int_0^s \mathsf{D}(u) du \,, \qquad \forall s \in [0,T] \,.
\]
Recall that $s \mapsto \mathsf{D}(s)$ is non-negative and non-decreasing. 
Hence, by yet another Gronwall argument we conclude that $\mathsf{D} \equiv 0$. 
In particular, $\Delta C(s,t) = 0$ for almost every $(s,t) \in \D_T$, while 
from \eqref{ineqK4}--\eqref{ineqR4} 
$$
\Delta K \equiv 0, \qquad \Delta R \equiv 0, \qquad 
\Delta q \equiv 0, \qquad \mbox{on} \quad \D_T \,.
$$
Going back to 
\eqref{b1}, this suffices for its \abbr{rhs} to be zero at any 
$t \le s \le T$, thereby having $\Delta C \equiv 0$ on $\D_T$.

\section{Proof of Propositions \ref{prop-macro} and \ref{lem-diff}}\label{sec:der+diff}

\subsection{Proof of Proposition \ref{prop-macro}}\label{sec:der}
Consider the limit $N \to \infty$ of the $\Ps$-expectation of the identities
\[
C_N(s,t) = \bar C_N(s,t) + \frac{q_N(s) q_N(t)}{\qs^2} \,, \qquad
\chi_N(s,t) = \bar \chi_N(s,t) + q_N(s) \frac{B_t^1}{\qs \sqrt{N}} \,.
\]
From \eqref{eq:self-av-star} we see 
that any limit point $(C,\chi,q,\bar C,\bar \chi)$
must satisfy \eqref{eq:ext} (with $\chi=\bar \chi$ as both $\E_\star [q^2_N(s)]$ 
and $\E [|B_t^1|^2]$ are bounded uniformly 
in $N$ and 
on $[0,T]$). The $\Ps$-expectation of 
\eqref{eq:integ} at $i=1$, amounts in view of \eqref{eq:qvdef}, to  
$q_N^a(s) = q_N^a(0) + \int^s_0 Q_N^a(u)  du$,
from which, by utilizing again \eqref{eq:self-av-star} as $N \to \infty$, 
we deduce the validity of the \abbr{rhs} of 
\eqref{eq:qQdef}. By the same reasoning, each  
limit point of the $\Ps$-expectation of \eqref{eq:qvdef}--\eqref{eq:PsiN-val} 
must satisfy \eqref{eq:qQdef}--\eqref{eq:Psi1}, 
respectively. Observing that $\bar \chi_N^a(0,t)=0$, and having 
as in \cite[Eqn. (3.2)-(3.3)]{BDG2}, 
\begin{align}\label{eq:BDG32-33}
\bar C_N(s,t) &= \bar C_N(s,0) + \bar \chi_N(s,t)  
+ \int_0^t \bar D_N(u,s) d u \,, \nonumber \\
\bar \chi_N(s,t) &= \bar \chi_N(0,t) + \frac{1}{N} \sum_{i=2}^N B_s^i B_t^i  + 
\int_0^s \bar E_N(u,t) d u 
\end{align}
(recall the definition \eqref{eq:dedef} of $\bar D_N$ and $\bar E_N$), we likewise 
deduce that  \eqref{eq:C1-chi1} holds. 
Recall that by the $\Ps$-independence of the standard Brownian increments
\begin{equation}\label{eq:FEchi-triang}
U_N^a(s,0)=0, \qquad U_N^a(s,t)=U_N^a(s,t \wedge s),
\qquad \qquad U_N \in \{ \bar F_N,\bar \chi_N,E_N \} 
\end{equation} 
(c.f. \cite[Page 638]{BDG2}),
hence our stated boundary conditions on the limit point.
The key to the proof is Proposition \ref{comp1}, which 
approximates 
$(\wQ_N^a,\bar A_N^a,\bar F_N^a, \wH_N^a)$
for $N \to \infty$, 
by a combination of functions from $\Ua_N^a$ 
(where expressions involving $\nu$, $\nu'$ and $\nu''$ 
emerge via the covariance 
kernels of Lemma \ref{kval}).
Indeed, with Proposition \ref{comp1} replacing \cite[Prop. 3.1]{BDG2}, 
we get \eqref{eqQ1}--\eqref{eqwH} (and thereby establish
Proposition \ref{prop-macro}), by following the derivation of 
\cite[Prop. 1.3]{BDG2}, while utilizing \eqref{eq:self-av-star} 
and the pre-compactness results of Proposition \ref{tight-self} (for $\Ps$), 
instead of \cite[Cor. 2.8]{BDG2} and \cite[Prop. 2.3]{BDG2}, respectively.
\begin{prop}\label{comp1}
Set $a_N \simeq b_N$ when $|a_N(\cdot)-b_N(\cdot)| \to 0$ 
as $N \to \infty$, uniformly on $[0,T]^2$. Then, for $\tau=t \vee s$, 
\begin{align}\label{eq:apxv}
V_N^a(s) \simeq &\;
\Phi_N^{1,a}(s,s) - \Phi_N^{1,a}(s,0) - \int_0^s \Psi_N^{1,a}(s,u) du 
\,,\\
\label{eq:apxa}
\bar A_N^a(t,s)\simeq & \;  
 \bar C_N^a(s,\tau) \Phi^a_N(t,\tau) - \bar C_N^a(s,0) \Phi^a_N(t,0)
- \int_0^\tau \Big\{ \bar D_N^a(s,u) \Phi^a_N (t,u) + \bar C_N^a (s,u) 
\Psi^a_N(t,u) \Big\} du 
\,,\\
\label{eq:apxf}
 \bar F_N^a (s,t)
 \simeq & \; \bar \chi_N^a (s,t) \Phi^a_N (s,s)  
- \int_0^{t \wedge s} \Phi^a_N(s,v) dv  
- \int_0^s \Big\{
 \Phi^a_N(s,u) \bar E_N^a(u,t) +  \bar \chi_N^a(u,t) \Psi^a_N(s,u) \Big\} du  \;,
 \\
 \label{eq:apxh}
\wH_N^a(s) \simeq & \; \Upsilon^a_N (s,s) - \Upsilon^a_N(s,0)   
 - \int_0^s \Big\{\bar D^a_N(s,u) \Phi^a_N(u,s) + \frac{Q^a_N(u)}{\qs^2} 
 \Phi_N^{1,a}(u,s) \Big\} du  \,.
\end{align}  
\end{prop}

\noindent
Towards proving Proposition \ref{comp1} we fix
a continuous path $\Bx$ satisfying \eqref{eq:integ}. Then, 
for any operator $k_t$ of kernel $k_{uv}^{ij}(\Bx)$ 
on $L_2(\{1,\cdots N\}\ts[0,t])$ and 
$f\in L_2(\{1,\cdots N\}\ts[0,t])$, let
\begin{equation}\label{eq:kfdef}
[k_t f]_u^i :=\sum_{j=1}^N \int_0^t k_{uv}^{ij} f^j_v dv\,, \qquad 
(i,u) \in\{1,\cdots,N'\} \times [0,t]\,,
\end{equation}
which is clearly in $L_2(\{1,\cdots N'\}\ts[0,t])$. Assuming that 
each $k_{uv}^{ij}(\Bx)$ is the finite sum of terms such as
$x^{i_1}_{u}\cdots x^{i_a}_{u}  x^{j_1}_v\cdots x^{j_b}_v$
(for some non-random 
$a$, $b$ and $i_1,\ldots,i_a,j_1,\ldots,j_b$), we further extend 
\req{eq:kfdef} to stochastic integrals of the form 
\begin{equation}\label{eq:stoch-integ}
[k_t \circ dZ]_u^i=\sum_{j=1}^N \int_0^t k_{uv}^{ij} dZ^j_v\,,
\end{equation}
where $Z_v$ is a continuous $\Fa_v$-semi-martingale 
(composed for each $j$, of a squared-integrable continuous martingale 
and a continuous, adapted, squared-integrable finite variation
part). Adopting the conventions of \cite[Page 640]{BDG2} for interpreting $\int_0^t k_{uv}^{ij} dZ^j_v$ 
in terms of It\^o integrals,
note that $[k_t \circ dZ]_u^i \in L_2(\{1,\cdots, N'\}\ts[0,t])$ 
(recall \req{eq:conc2} that $\Bx_t$ has uniformly over time, bounded 
moments of all orders under $\Ps$,
hence so does the kernel $k_{ts}^{ij}(\Bx)$), with the following
extension of \cite[Lemma 3.3]{BDG2}.
\begin{lem}\label{condexp}
Fixing $\tau \in \R_+$ there exist a version of 
$V_{s;\tau}^i :=\Est [G^i(\Bx_s) |\Fa_\tau]$ 
and $Z_{s;\tau}^i :=\Est [B^i_s|\Fa_\tau]$ with
\begin{align}
Z^i_{s;\tau} &= x^i_s - x^i_0 - \int_0^s Q^i_{u;\tau} du \;,
\label{eq:UBdef}
\qquad
Q^i_{s;\tau}  & := V^i_{s;\tau} - f'(K_N(s)) x^i_s + 
{\bf 1}_{\{i=1\}} \sqrt{N} \qs \bv' (q_N(s))  \,,
\end{align}
such that $s \mapsto Z^i_{s;\tau}$ are
continuous semi-martingales with respect to 
the filtration $(\Fa_s, s \le \tau)$,
composed of squared-integrable continuous 
martingales and finite variation parts. If $\{S^i(\Bx)$, $i \le N'\}$ 
are linear forms in $\BJ$ with covariance kernels 
\begin{equation}\label{def:HG-k}
k_{st}^{ij}(\Bx):= \E_{\BJ}
\left\{ S^{i}(\Bx_s)G^{j}(\Bx_t)\,|\,\cpt_\star \right\}, \;\;
\widetilde{k}_{st}^{i l}(\Bx):= \E_{\BJ}
\left\{ S^{i}(\Bx_s)S^{l}(\Bx_t)\,|\,\cpt_\star \right\},
\;\; 1 \le i,l \le N', \, 1 \le j \le N \,,
\end{equation}
consisting of polynomials in $\Bx$, then 
\begin{equation}\label{eq:Vid}
Y_{s;\tau}^i := \Est[S^i(\Bx_s)|\Fa_\tau] = 
[k_\tau \circ dZ]_s^i = [k_\tau \circ dx]^i_s - [k_\tau Q]^i_s \,,
\quad \forall (i,s) \in \{1,\cdots,N'\} \times [0,\tau] \,.
\end{equation}
Further, there exist then a version of 
\begin{equation}\label{eq:gammadef}
\begin{aligned}
\widetilde{\Gamma}_{st;\tau}^{i l} :=& 
\Est \Big[(S^i(\Bx_s) - Y_{s;\tau}^i)(S^{l} (\Bx_t)- 
Y_{t;\tau}^{l})|\Fa_\tau\Big]\,, \qquad 
  i,l \in \{1,\cdots,N'\},\\
\Gamma_{s t;\tau}^{j l} :=& 
\Est \Big[(G^j (\Bx_s)- V_{s;\tau}^{j})(S^l(\Bx_t) - Y_{t;\tau}^l)
|\Fa_\tau\Big]\,, 
\qquad s,t \in [0,\tau], j \in \{1,\cdots,N\},  
\end{aligned}
\end{equation}
such that  
\begin{equation}\label{eq:gaid} 
\widetilde{\Gamma}_{st;\tau}^{i l} = \widetilde{k}_{st}^{i l} -
\sum_{j=1}^N \int_0^\tau k_{su}^{ij} \Gamma_{ut;\tau}^{j l} du \,,
\qquad 
\forall s,t \in [0,\tau], \;\; i,l \in \{1,\cdots,N'\}\,.
\end{equation}
\end{lem}
\begin{proof} The right equality in \eqref{eq:Vid} follows from 
the relation \eqref{eq:UBdef} between $\Est [B^i_s|\Fa_\tau]$ and 
$\Est [G^i(\Bx_u)|\Fa_\tau]$, which in turn is a consequence of 
having in \req{eq:integ},
\begin{eqnarray}
U^i_s := x^i_s-x^i_0+\int_0^s f'(K_N(u)) x^i_u du - 
{\bf 1}_{\{i=1\}} \sqrt{N} \qs \int_0^s \bv' (q_N(u)) du 
=\int_0^s G^i(\Bx_u) du + B_s^i  \,.
\label{eq:Udef} 
\end{eqnarray}
The latter relation implies the stated
continuity and integrability properties of the semi-martingales 
$U^i_s$ and $Z^i_{s;\tau} = U^i_s - \int_0^s V^i_{u;\tau} du$.
By Girsanov formula (see \cite[Eqn. (3.16)]{BDG2}), 
the restriction to $\Fa_\tau$ satisfies
\begin{equation}\label{eq:lndf}
\P^N_{\BJ,\Bx_0}|_{\Fa_\tau}= \L^N_\tau \,
\P^N_{{\bf 0},\Bx_0}|_{\Fa_\tau} \,, \qquad
\L^N_\tau  := \exp\Big\{ \sum_{i=1}^N \int_0^\tau G^i(\Bx_s) dU^i_s
-\half\sum_{i=1}^N \int_0^\tau (G^i(\Bx_s))^2 ds\Big\} 
\,,
\end{equation}
with $U^i_s$ a standard Brownian motion under $\P^N_{{\bf 0},\Bx_0}$.
Setting $\P^\star_{\BJ}$ for the law of $\BJ$ conditional on $\cpt_\star$, 
we thus have (as in the proof of \cite[Lemma 3.3]{BDG2}), that
\begin{equation}\label{eq:pre-C2-C3} 
Y_{s;\tau}^i = \frac{\E^{\star}_{\BJ} 
[S^i(\Bx_s) \L^N_\tau]}{\E^\star_{\BJ}[\L^N_\tau]} \,, 
\qquad 
\widetilde{\Gamma}_{st;\tau}^{i l} =
\frac{\E^\star_{\BJ} \Big[(S^{i} (\Bx_s)- Y_{s;\tau}^{i}) 
(S^l(\Bx_t) - Y_{t;\tau}^l)
\L^N_\tau \Big]}{\E^\star_{\BJ}[\L^N_\tau]} \,.
\end{equation}
The centered Gaussian law $\P^\star_\BJ$ is not a product measure, but the
arguments used in proving \cite[Proposition C.1]{BDG2} still apply. Specifically,
here $G^j(\Bx_s)=\sum_\alpha J^o_\alpha L_s^j(\alpha)$ and
$S^i(\Bx_t)=\sum_\alpha J^o_\alpha M_t^i(\alpha)$ for some independent
centered Gaussian $\{J^o_\alpha\}$ of positive variances $v_\alpha$, with 
$k^{ij}_{su} = \sum_\alpha M_s^i(\alpha) v_\alpha L_u^j(\alpha)$.
Our Radon-Nikodym derivative $\L^N_\tau$ is given 
in terms of $\BR=\{R_{\alpha \gamma}\}$ of \cite[(C.4)]{BDG2} 
and $\BJ^o:=\{J^o_\alpha\}$, by the display following 
\cite[(C.4)]{BDG2}. Under such a change of measure 
the Gaussian law of $\BJ^o$ has the covariance matrix 
$(\BD^{-1}+\BR)^{-1}$ for $\BD=$diag($v_\alpha$)
and the mean vector $\bq=(\BD^{-1} + \BR)^{-1} \bmu$
of \cite[(C.5)]{BDG2}. From the \abbr{lhs} of \eqref{eq:pre-C2-C3} we have
that $Y^i_{s;\tau} = \sum_\alpha M_s^i(\alpha) q_\alpha$ and
$V^j_{u;\tau} = \sum_\alpha L_u^j(\alpha) q_\alpha$. Further, by
definition 
$[k_\tau \circ dU]^i_s = \sum_\alpha M^i_s(\alpha) v_\alpha \mu_\alpha$ 
and 
$[k_\tau V]^i_s = \sum_{\alpha,\gamma} M^i_s(\alpha) v_\alpha R_{\alpha \gamma} 
q_\gamma$ (thanks to \cite[(C.4)]{BDG2}), with the identity 
$Y^i_{s;\tau}=[k_\tau \circ Z]^i_s$ of \eqref{eq:Vid} thus a direct consequence
of \cite[(C.5)]{BDG2}. Next, note that 
$\widetilde{k}^{il}_{st} = \sum_\alpha M_s^i(\alpha) v_\alpha M_t^l(\alpha)$, 
whereas from the \abbr{rhs} of \eqref{eq:pre-C2-C3} we have that 
\[
\Gamma_{ut;\tau}^{j l} = \sum_{\alpha,\gamma} L^j_u(\alpha) 
[(\BD^{-1}+\BR)^{-1}]_{\alpha \gamma} M^l_t(\gamma) \,, \qquad
\widetilde{\Gamma}_{st;\tau}^{i l} = \sum_{\alpha,\gamma} M^i_s(\alpha) 
[(\BD^{-1}+\BR)^{-1}]_{\alpha \gamma} M^l_t(\gamma) \,.
\]
By \cite[(C.4)]{BDG2} we thus get \eqref{eq:gaid} out of
$(\BD^{-1}+\BR)^{-1} = \BD - \BD \BR (\BD^{-1}+\BR)^{-1}$
(as in the proof of \cite[(C.3)]{BDG2}).
\end{proof}

\noindent
\emph{Proof of Proposition \ref{comp1}.} In view of 
\req{eq:hfdef}-\eqref{eq:dedef} and \eqref{def:Uhat}, one has 
as in \cite[Pg. 642]{BDG2}, for any $\tau \in [t \vee s,T]$,   
\begin{equation}\label{def:Ahat}
\bar A_N(t,s|\tau) 
=\oneN \sum_{i=2}^N V^i_{t;\tau} x^i_s \,, \;\;
\bar D_N(s,t|\tau) = \oneN \sum_{i=2}^N Q^i_{t;\tau}  x^i_s \,,
\;
V_N(s|\tau) = \frac{\qs}{\sqrt{N}} V^1_{s;\tau}  \,,
\;
Q_N(s|\tau) = \frac{\qs}{\sqrt{N}} Q^1_{s;\tau}  \,.
\end{equation}
Recall It\^o's formula for $u \mapsto \widetilde{k}(\Bx_s,\Bx_u)$, 
\begin{equation}\label{eq:HN-stoch}
\sum_{j=1}^N \int_0^\tau 
\partial_{x^j_u} \{ \widetilde{k} (\Bx_s,\Bx_u) \} dx_u^j
= \widetilde{k}(\Bx_s,\Bx_\tau) - \widetilde{k} (\Bx_s,\Bx_0) - 
\frac{1}{2} \int_0^\tau  
\{ \Delta_{\Bx_u} \widetilde{k} (\Bx_s,\Bx_u) \} du \,.
\end{equation}
Thus, for the operator $k_\tau$ corresponding to 
$S^i \equiv G^i$ in Lemma \ref{condexp}, we get from 
the first identity of \eqref{eq:kval} that  
\begin{equation}\label{eq:k-dx}
[k_\tau \circ dx]_s^i =
\sum_{j=1}^N \int_0^\tau 
\partial_{x^i_s}\partial_{x^j_u} \{ \widetilde{k} (\Bx_s,\Bx_u) \} dx_u^j
= \varphi^i_N(s,\tau) - \varphi_N^i (s,0) - \int_0^\tau  \delta^i_N(s,u) du \,,
\end{equation}
for any $(i,s) \in [0,\tau] \times \{1,\ldots,N\}$, where
\[
\varphi^i_N(s,u) := \partial_{x^i_s} \widetilde{k}(\Bx_s,\Bx_u) \,, 
\qquad \qquad 
\delta^i_N(s,u) := \frac{1}{2} \partial_{x^i_s}
\{ \Delta_{\Bx_u} \widetilde{k} (\Bx_s,\Bx_u) \} \,.
\]
By the second identity of \eqref{eq:kval} we arrive at
\begin{align}\label{eq:varhpi-val}
\varphi_N^i(s,u) = \partial_{x^i_s} \widetilde{k}(\Bx_s,\Bx_u) 
&= x^i_u {\bf 1}_{\{i \ne 1\}} \Phi_N(s,u) + \frac{\sqrt{N}}{\qs} {\bf 1}_{\{i =1 \}} 
\Phi_N^1(s,u) \,,
\end{align}
in terms of $\Phi_N(\cdot,\cdot)$ and $\Phi_N^1(\cdot,\cdot)$ of 
\eqref{eq:PhiN-val}. Consequently,
\begin{equation}\label{eq:kval-exp}
\begin{aligned}
k_{su}^{ij} =& \partial_{x_u^j} \{ \varphi^i_N (s,u) \}  
= \frac{x_{s}^{j}x_{u}^{i}}{N}\nu''(C_{N}(s,u)) 
+{\bf 1}_{\{i=j \ne 1\}}\Phi_N(s,u) \\
&	+
	{\bf 1}_{\{i=j=1\}} \Big[
	\nu'(C_{N}(s,u)) 
	 - \bar C_N(s,u) \frac{\qs^2 \nu''(q_N(s))\nu''(q_N(u))}{\nu'(\qs^2)} \Big]\\ 
	& -{\bf 1}_{\{i=1\}}{\bf 1}_{\{j\neq1\}}\frac{\qs x_s^{j}}{\sqrt{N}}\frac{\nu''(q_N(s))\nu'(q_N(u))}{\nu'(\qs^2)}-{\bf 1}_{\{j=1\}}{\bf 1}_{\{i\neq1\}}\frac{\qs x_u^{i}}{\sqrt{N}}\frac{\nu''(q_N(u))\nu'(q_N(s))}{\nu'(\qs^2)}\,.
 \end{aligned}	
	\end{equation}	 
Similarly, 	
\begin{equation}\label{eq:HN-corr}
\Delta_{\Bx_u} \widetilde{k} (\Bx_s,\Bx_u) = K_N(s) \nu''(C_N(s,u)) 
- \frac{\qs^2 \nu'''(q_N(u))}{\nu'(\qs^2)} \nu'(q_N(s)) \bar C_N(s,u) \,,
\end{equation}
resulting after some algebra with
\begin{align}\label{eq:delta-x}
\delta^i_N(s,u)
= & \frac{x_s^i}{N}  \nu''(C_N(s,u))  
+ \frac{x^i_u}{2N} 
K_N(s) \nu'''(C_N(s,u))  
\nonumber \\ &
 - \frac{\qs^2 \nu'''(q_N(u))}{2 \nu'(\qs^2)} 
 \big[{\bf 1}_{\{i \ne 1\}} \frac{x^i_u \nu'(q_N(s))}{N} + 
 {\bf 1}_{\{i=1\}} \frac{\qs \nu''(q_N(s))}{\sqrt{N}} \bar C_N(s,u) \big] \,.
\end{align}
Next, with 
$\varphi^j_N(u,s) = \partial_{x_u^j} \{\widetilde{k} (\Bx_s,\Bx_u)\}$ 
it follows from \eqref{eq:kval} and \eqref{eq:kfdef}, that 
\begin{align}\label{eq:kQ-val}
[k_\tau Q]^i_s = \int_0^\tau \psi^i_N(s,u|\tau) du \,, \qquad \qquad 
\psi^i_N(s,u|\tau) := \partial_{x_s^i} 
\Big[\sum_{j=1}^N \varphi^j_N(u,s) Q^j_{u;\tau} \Big] \,.
\end{align}
Combining \eqref{def:Ahat} and \eqref{eq:varhpi-val}, we have 
\begin{align}\label{eq:HN-Q}
\sum_{j=1}^N \varphi^j_N(u,s) Q^j_{u;\tau} = & \; \; N 
\bar D_N(s,u|\tau) \Phi_N(u,s) 
+ \frac{N}{\qs^2} Q_N(u|\tau) \Phi_N^1(u,s) 
\,,
\end{align}
which in view of \eqref{eq:PhiN-val}, \eqref{eq:PsiN-val} 
and the symmetry of $\Phi_N(\cdot,\cdot)$
yields that 
\begin{equation}\label{eq:psii-val}
\psi^i_N(s,u|\tau) = \Big\{\!\!\!
\begin{array}{cc} 
& \qquad \quad \frac{\sqrt{N}}{\qs} \Psi^1_N(s,u |\tau) \,, \qquad \qquad i= 1 \,, \\
& Q^i_{u;\tau} \Phi_N(s,u) + x^i_u \Psi_N(s,u|\tau) \,, \quad i \ne 1 \,.
\end{array}
\end{equation}
In this case $Y^i_{s;\tau}=V^i_{s;\tau}$, so by 
\eqref{eq:Vid}, \eqref{eq:k-dx} and \eqref{eq:kQ-val} we get 
\begin{equation}\label{eq:Vi-val}
V^i_{s;\tau} =  [k_\tau \circ dx]^i_s - [k_\tau Q]^i_s 
= \varphi^i_N(s,\tau) - \varphi_N^i (s,0) - \int_0^\tau  
[\psi^i_N(s,u|\tau) + \delta^i_N(s,u)] du \,
\qquad \forall s \in [0,\tau] \,.
\end{equation}
In particular, for 
$\epsilon_N(s) := \frac{\qs}{\sqrt{N}} \int_0^s \delta^1_N(s,u) du$,
by \eqref{def:Ahat}, \eqref{eq:varhpi-val} and \eqref{eq:psii-val},
\begin{align*}
V_N(s|s) + \epsilon_N(s)  
& = \Phi_N^1(s,s) - \Phi_N^1(s,0) 
- \int_0^s \Psi_N^1(s,u|s) du 
\,.
\end{align*}
We now consider the $\E_\star$-expected value of the preceding identity.
From \eqref{eq:delta-x} we have that $\epsilon_N^a \simeq 0$, 
so with $U_N^a(s,t) = \Est U_N(s,t|\tau)$ we arrive at
\eqref{eq:apxv}.
Turning to the derivation of \eqref{eq:apxa}, for $\tau=t \vee s$ and 
\[
\tilde{\epsilon}_N(t,s):= 
\int_0^\tau \Big\{ \oneN \sum_{i=2}^N x^i_s \delta^i_N(t,u) \Big\} du \,,
\]
we have in view of \eqref{def:Ahat}, \eqref{eq:varhpi-val}, 
\eqref{eq:psii-val} and \eqref{eq:Vi-val}, that 
\begin{align*}
\bar A_N (t,s|\tau) &+ \tilde{\epsilon}_N(t,s) = 
\oneN \sum_{i=2}^N  \varphi^i_N(t,\tau) x_s^i 
- \oneN \sum_{i=2}^N \varphi_N^i (t,0) x_s^i 
- \int_0^\tau  \Big\{ \oneN \sum_{i=2}^N \psi^i_N(t,u|\tau) x_s^i \Big\} du 
\nonumber \\
&= \bar C_N(s,\tau) \Phi_N(t,\tau) - \bar C_N(s,0) \Phi_N(t,0)
- \int_0^\tau 
\Big\{\bar D_N(s,u|\tau) \Phi_N(t,u) + \bar C_N(s,u) \Psi_N(t,u|\tau) \Big\} du 
\,.
\end{align*}
Since $\tilde{\epsilon}^{\,a}_N \simeq 0$,
we get \eqref{eq:apxa} from the 
preceding identity (upon applying \eqref{eq:self-av-star} for the function 
$z_1 z_2$).

Moving to \eqref{eq:apxf}, by \eqref{eq:FEchi-triang} it suffices to consider
hereafter $t \in [0,s]$. Further,  
$B^i_t = U^i_t - \int_0^t G^i(\Bx_v) dv$ with $U^i_t$ measurable on $\Fa_\tau$
(c.f. \eqref{eq:Udef}). Hence, in view of \eqref{eq:gammadef},
\begin{equation}\label{eq2}
\E_\star[ (V^i_{s;\tau} - G^i(\Bx_s)) B^i_t | \Fa_\tau ] = 
\E_\star[ (V^i_{s;\tau} - G^i(\Bx_s)) \int_0^t (V_{v;\tau}^i-G^i(\Bx_v))
 dv | \Fa_\tau ] = 
 \int_0^t \Gamma_{sv;\tau}^{ii} dv \,.
\end{equation}
In particular, setting 
\[
\Gamma_N(s,t|\tau) := \oneN \sum_{i=2}^N \int_0^t \Gamma_{sv;\tau}^{ii} dv 
\]
we deduce that 
\begin{equation}\label{def:Fhat}
\Gamma_N(s,t|\tau)  
= \oneN \sum_{i=2}^N V_{s;\tau}^i Z_{t;\tau}^i 
 - \bar F_N(s,t|\tau) =
\oneN \sum_{i=2}^N Q_{s;\tau}^i Z_{t;\tau}^i 
 - E_N(s,t|\tau)  \,, \;\; 
\bar \chi_N(s,t|\tau) = \oneN \sum_{i=2}^N x_s^i Z_{t;\tau}^i  \,. 
\end{equation}  
From \eqref{eq:varhpi-val}, \eqref{eq:psii-val} and 
\eqref{def:Fhat} (at $\tau=s$), we also have that 
\begin{equation}\label{def:Theta-Pi}
\begin{aligned}
\Theta_N(s,u;t) &:= \oneN \sum_{i=2}^N \varphi^i_N(s,u) Z^i_{t;s}
= \bar \chi_N(u,t|s) \Phi_N(s,u) 
 \,, \\
\Pi_N(s,u;t) &:= \oneN \sum_{i=2}^N \psi^i_N(s,u|s) Z^i_{t;s} =
\big[\Gamma_N(u,t|s) + \bar E_N(u,t|s)\big] \Phi_N(s,u) 
+ \bar \chi_N(u,t|s) \Psi_N(s,u|s)  \,.
\end{aligned}
\end{equation}
Further, from \eqref{eq:Vi-val} we get 
\begin{align}\label{eq:Gamma+F}
\Gamma_N(s,t|s) + \bar F_N (s,t|s) &= 
\oneN \sum_{i=2}^N V_{s;s}^i Z_{t;s}^i =
\Theta_N (s,s;t) - \Theta_N(s,0;t) - \int_0^s \Pi_N(s,u;t) du 
 - \widehat{\epsilon}_N(s,t) 
\,,
\end{align}
where $\widehat \epsilon_N(s,t) := \oneN \sum_{i=2}^N 
\int_0^s \delta^i_N(s,u) Z^i_{t;s} du$ is such that  
$\widehat{\epsilon}^{\,a}_N \simeq 0$ (see \eqref{eq:delta-x}). 
Next, setting  
\begin{equation}\label{eq:smallphi-def}
\phi_N(s,v) := \Phi_N(s,v)  -
 \oneN \sum_{i=2}^N \Gamma_{sv;s}^{ii} 
- \oneN \sum_{i=2}^N \int_0^s \Phi_N(s,u) \Gamma_{uv;s}^{ii} du \,, 
\qquad \qquad  v \in [0,s] \,,
\end{equation}
we see that 
$$
\int_0^t \Phi_N(s,v) dv = 
\Gamma_N(s,t|s) 
+ \int_0^s \Phi_N(s,u) \Gamma_N (u,t|s) du + \int_0^t \phi_N(s,v) dv 
\,,
$$
so combining \eqref{def:Theta-Pi} and \eqref{eq:Gamma+F}
results with  
\begin{align}\label{eq:detail}
\int_0^t \Phi_N(s,v) dv + \bar F_N (s,t|s) 
=&  
\bar \chi_N(s,t|s) \Phi_N(s,s) - \bar \chi_N(0,t|s) \Phi_N(s,0)
+ \int_0^t \phi_N(s,v) dv - \widehat{\epsilon}_N(s,t) 
 \nonumber \\
& - \int_0^s \Big\{ 
\bar E_N(u,t|s) \Phi_N(s,u) + \bar \chi_N(u,t|s) \Psi_N(s,u|s)  
\Big\} du \,.
\end{align}
Recalling that $\bar \chi_N^a(0,t)=0$, we thus get \eqref{eq:apxf}
by employing \eqref{eq:self-av-star} 
on the $\E_\star$-expectation of the \abbr{rhs} of \eqref{eq:detail} and 
relying on the following analog of \cite[Lemma 3.4]{BDG2}.
\begin{lem}\label{cancelem}
For $\phi_N(s,v)$ of \eqref{eq:smallphi-def},
$$
\lim_{N \to \infty} \sup_{(s,v) \in \Delta_T} |\phi^a_N(s,v)| = 0 \,.
$$
\end{lem}
\begin{proof}[Proof of Lemma \ref{cancelem}] 
Recall that $\widetilde{\Gamma}=\Gamma$ and $\widetilde{k}_s=k_s$
in our special case of Lemma \ref{condexp}. Thus, setting 
\[
\bar \gamma_N (u,v|s) := 
\frac{1}{N^2} \sum_{i,j=2}^N x^j_s x^i_u \Gamma_{uv;s}^{ji} 
\,,
\quad \qquad
\gamma_N^1(u,v|s) :=
\frac{\qs}{\sqrt{N}} \oneN \sum_{i=2}^N  x_u^i \Gamma^{1i}_{uv;s} \,,
\]
we deduce 
from \req{eq:gaid}, \eqref{eq:kval-exp} and \eqref{eq:smallphi-def} that
for any $v,u \in [0,s]$,  
\begin{align*}
0 = &  \oneN \sum_{i=2}^N \Big[ k_{sv}^{ii} - \Gamma_{sv;s}^{ii} 
- \sum_{j=1}^N \int_0^s k_{su}^{ij} \Gamma_{uv;s}^{ji} du \Big] 
=  \phi_N(s,v) +\oneN \Big[ \nu''(C_N(s,v)) \bar C_N(s,v) - \Phi_N(s,v) \Big] \\
& - \int_0^s \nu''(C_N(s,u)) \bar \gamma_N (u,v|s) du 
-\int_0^s 
\Big[ \frac{q_N(s)}{\qs^2} \nu''(C_N(s,u)) 
- \frac{\nu''(q_N(u)) \nu'(q_N(s))}{\nu'(\qs^2)} \Big]
\gamma_N^1(u,v|s) du  \,.
\end{align*}
Recalling Proposition 
\ref{use-CPs} that the uniform moment bounds 
\eqref{eq:u-mombd} apply for $\Ps$ and any $U_N \in \Ua_N$,
it thus suffices to show that $\E_\star [(\bar \gamma_N)^2] \simeq 0$
and $\E_\star [(\gamma^1_N)^2] \simeq 0$. To this end,
from the definitions of $\bar A_N$, $\Gamma^{ji}_{uv;s}$ 
(see \eqref{eq:hfdef}, \eqref{eq:gammadef}), and the \abbr{lhs} of
\eqref{def:Ahat}, we find that
\begin{align*}
\bar \gamma_N (u,v|s) & = \E_\star \Big[ 
\frac{1}{N^2} \sum_{i,j=2}^N x^j_s x^i_u 
(G^j(\Bx_u)-V^j_{u;s})(G^i(\Bx_v)-V^i_{v;s}) |\Fa_s \Big]  \\
&=
\E_\star\Big[ (\bar A_N(u,s) - \bar A_N (u,s|s)) (\bar A_N(v,u) - 
\bar A_N (v,u|s))
 |\Fa_s \Big] \,. 
\end{align*}
In particular, by Cauchy-Schwarz
\[
\sup_{u,v \in [0,s]} \Big\{ \E_\star \big[ \bar \gamma_N (u,v|s)^2 \big] \Big\}
\le  \sup_{u,v \in [0,s]} \Big\{
\E_\star\Big[ (\bar A_N(v,u) - \bar A_N (v,u|s))^2 \Big] \Big\}
\]
which goes to zero as $N \to \infty$ (apply  
Corollary \ref{cor-self} for $\Psi(z)=(z_1-z_2)^2$ and 
$\BZ_N = (\bar A_N(\cdot),\bar A_N (\cdot|s))$ with 
$\Psi(\E_\star \BZ_N) = 0$).
Similarly, we get from \eqref{eq:hfdef}, \eqref{eq:gammadef} 
and the right-most identity in \eqref{def:Ahat} that 
\begin{align*}
\gamma_N^1(u,v|s) = &
\E_\star \Big[ \frac{\qs}{\sqrt{N}} (G^1(\Bx_u)-V^1_{u;s})
\frac{1}{N} \sum_{i=2}^N x^i_u (G^i(\Bx_v)-V^i_{v;s}) |\Fa_s \Big] \\
= & 
\E_\star \Big[ (V_N(u)-V_N(u|s)) (\bar A_N(v,u) - \bar A_N (v,u|s))
|\Fa_s \Big]
\,.
\end{align*}
Thus, as before, the uniform convergence 
to zero of $\E_\star [(\gamma^1_N(u,v|s))^2]$ 
follows by combining Cauchy-Schwarz and Corollary \ref{cor-self} for $\Ps$
(taking here $\BZ_N = (V_N(u),V_N(u|s))$).
\end{proof}

Proceeding to establish \eqref{eq:apxh}, we compute $\wH_N(s|s)$
by employing Lemma \ref{condexp} for $H=\wH_N$ (with $N'=1$). This  
corresponds to having covariance kernel 
$\hat{k}^{1j}_{su}=\oneN \partial_{x^j_u} \widetilde{k}(\bx_s,\bx_u)$.
In view of our definition of $\varphi_N^j(u,s)$, we
then get from the \abbr{rhs} of \eqref{eq:Vid} at $\tau=s$, upon 
utilizing \eqref{eq:HN-stoch} and \eqref{eq:HN-Q}, that 
\begin{align*}
\wH_N(s|s) & =
\oneN \sum_{j=1}^N \int_0^s \partial_{x^j_u} 
\{ \widetilde{k}(\bx_s,\bx_u) \} dx^j_u - 
\int_0^s \oneN \big[ \sum_{j=1}^N \varphi_N^j(u,s) Q^j_{u:s} \big] du \\
&= \oneN \widetilde{k}(\Bx_s,\Bx_s) - \oneN \widetilde{k} (\Bx_s,\Bx_0)  
-\int_0^s \Big[\bar D_N(s,u|s) \Phi_N(u,s) + \qs^{-2} Q_N(u|s) \Phi_N^1(u,s) \Big] du 
- \epsilon^{\dagger}_N(s)  \,,
\end{align*}
for $\epsilon^{\dagger}_N(s) := \frac{1}{2N} \int_0^s  
\{ \Delta_{\Bx_u} \widetilde{k} (\Bx_s,\Bx_u) \} du$ such that 
$(\epsilon^{\dagger}_N)^a \simeq 0$ (see \eqref{eq:HN-corr}). 
In view of the second identity of \eqref{eq:kval}, considering 
$\E_\star \wH_N(s|s)$ yields \eqref{eq:apxh} (upon applying 
\eqref{eq:self-av-star} for the function $z_1 z_2$), thereby completing the proof of Proposition \ref{comp1}.
\qed
  
\subsection{Proof of Proposition \ref{lem-diff}}\label{sec:diff} We first show 
that $t \mapsto \chi(s,t)=\bar \chi(s,t)$
is continuously differentiable on $s \ge t$. Indeed, per fixed $t$ we have 
from \req{eqE} and the \abbr{rhs} of \eqref{eq:dedef1} 
that $\bar E(s,t)=[k_C \bar E(\cdot,t)](s) + h(s,t)$, with   
$$
h(s,t):=[\Phi(s,s)-f'(K(s))] \chi(s,t)
-\int_0^s  \chi(u,t) \Psi(s,u) du
-\int_0^{t \wedge s} \Phi(s,u) du 
$$
in $\Ca_b([0,T]^2)$, and integral operator $k_C$ on $\Ca ([0,T])$ 
of uniformly bounded kernel $\Phi(s,u)$ on $[0,T]^2$. 
As in the proof of \cite[Lemma 4.1]{BDG2}, Picard iterations 
yield that 
\begin{equation}
\bar E(s,t)
=\sum_{n\ge 0} [k_{C}^n h(\cdot,t)](s)
= h(s,t) + \int_0^s \kappa_C (s,v) h(v,t) dv \,,
\label{blurp}
\end{equation}
with a uniformly bounded kernel $\kappa_{C}$.
Plugging \req{blurp} into the \abbr{rhs} of \req{eq:C1-chi1}, 
we find by Fubini's theorem that
$$
\chi(s,t)= 
s \wedge t + \int_0^s[\int_0^{t \wedge v} \Phi(v,u) du]\kappa_1(s,v) dv
+ \int_0^s \chi(v,t)\kappa_2(s,v) dv \,,
$$
for some uniformly bounded $\kappa_1$ and
$\kappa_2$ (which depend only on $\Phi$, $\Psi$ and $f'(K(\cdot))$).
Applying Picard's iterations 
now with respect to the
integral operator $[\kappa_2 g](s) = \int_0^s \kappa_2(s,v) g(v) dv$,
we deduce that 
$$
\chi(s,t)= s\wedge t + \int_0^s \Big[ (u\wedge t) \kappa_3 (s,u) +
\big[ \int_0^{t \wedge u} \Phi(u,v) dv \big] \kappa_4(s,u) \Big] du \;,
$$
for some uniformly bounded $\kappa_3$ and $\kappa_4$. With 
$s \wedge t=t$ continuously differentiable on $s\ge t$, we
conclude by Fubini's theorem that $\chi(s,t)=\int_0^t R(s,u) du$, 
for the bounded continuous 
$$
R(s,t)=1+\int_t^s [\kappa_3 (s,u) + \Phi(u,t) \kappa_4(s,u)]du \;.
$$
In particular, $R(s,s)=1$ for all $s$. Next, having  
that $\bar E(s,0)=0$ for all $s$ and
$\bar E(s,t)=\bar E(s,s)$ for all $t \geq s$, imply the same 
for $\chi(s,t)$ (see the \abbr{rhs} of \eqref{eq:C1-chi1}),
and in particular $R(s,t) = (\partial_2 \chi) (s,t) = 0$ 
when $t>s$. From the \abbr{lhs} of \eqref{eq:C1-chi1} we 
see that $\partial_2 \bar C(s,t) = R(s,t) + \bar D(s,t)$, 
hence also 
$\partial_1 \bar C (s,t) = \partial_2 \bar C(t,s) = \bar D(t,s) 
+ R(t,s)$ (by the symmetry of $\bar C$).  From the \abbr{rhs} of \eqref{eq:qQdef}
we have $Q(t) = \partial q(t)$, so by the \abbr{lhs} of \eqref{eq:ext}
\begin{equation}\label{eq:part2C}
\partial_2 C(s,t) = \bar D(s,t) + R(s,t) + \frac{q(s) Q(t)}{\qs^2} \,.
\end{equation}
These imply in turn that the symmetric $\Upsilon(\cdot,\cdot)$ of \eqref{eq:Upsil}
is differentiable and by \eqref{eq:Phi}, \eqref{eq:Phi1},
\[
\partial_2 \Upsilon(s,u) = \bar D (s,u) \Phi (u,s) + \frac{Q(u)}{\qs^2} \Phi^1 (u,s)
+ R(s,u) \big[ \nu'(C(s,u))  -  \frac{\nu'(q(s)) \nu'(q(u))}{\nu'(\qs^2)} \big] \,,
\]
with \eqref{eqH} a consequence of \eqref{eqwH}. Similarly, the 
symmetric $\Phi(\cdot,\cdot)$ of \eqref{eq:Phi} is differentiable and by \eqref{eq:Psi}, 
\begin{align}\label{eq:p2-Phi}
\partial_2 \Phi(s,t) &= \Psi(s,t) + \nu''(C(s,t)) R(s,t) \,, \\
\partial_2 [\bar C(s,u) \Phi(t,u)] &= \bar D(s,u) \Phi(t,u)
+ \bar C(s,u) \Psi(t,u) + \bar C(s,u) \nu''(C(t,u)) R(t,u) + R(s,u) \Phi(t,u) \,. 
\nonumber
\end{align}
Combining the latter with \eqref{eqD}, then substituting into the \abbr{lhs}
of \eqref{eq:dedef1} we get that for all $t,s \in [0,T]^2$,
\begin{equation}\label{eq:tempD}
\bar D(s,t)=-f'(K(t)) \bar C(t,s)+
\int_0^{t \vee s} \Phi(t,u) R(s,u) du
+\int_0^{t \vee s} R(t,u) \nu''(C(t,u))
\bar C(s,u)  du \,.
\end{equation}
Similarly, comparing \eqref{eq:Phi1} and \eqref{eq:Psi1} it is easy to check that
\[
\partial_2 \Phi^1(s,u) - \Psi^1(s,u) = R(s,u) \Big[ q(u) \nu''(C(s,u)) 
- \frac{\qs^2 \nu'(q(u)) \nu''(q(s))}{\nu'(\qs^2)} \Big] \,,
\]
which together with \eqref{eqQ1} and \eqref{eq:qQdef} (with $\bv(\cdot)=\bv_\star(\cdot)$), results with \eqref{eqq} (at $\beta=1$). Further, 
combining \eqref{eqq} at $\beta=1$, \eqref{eq:part2C} and \eqref{eq:tempD} 
at $t>s$ leads to 
\begin{align}
\partial_2 C(s,t) = & R(s,t) - f'(K(t)) C(t,s)+ \int_0^t \Phi(t,u) R(s,u) du \nonumber \\
&+ \int_0^t R(t,u) \Big[ \nu''(C(t,u)) C(s,u) 
-  \frac{q(s) \nu'(q(u)) \nu''(q(t))}{\nu'(\qs^2)} \Big] 
du + \b q(s) \bv'_\star (q(t))  \,.
\label{eqCts}
\end{align}
Noting that $R(s,u)=0$ when $u>s$, whereas
$\partial_1 C(s,t) = \partial_2 C(t,s)$, 
interchanging $t$ and $s$ in \eqref{eqCts}
results for $s>t$ with \req{eqC} at $\beta=1$. 

Since $K(s)=C(s,s)$, with $C(s,t)=C(t,s)$ 
and $\partial_2 C = D+R$ for $D := \bar D + q(s) Q(t)/\qs^2$
(see \eqref{eq:part2C}),
it follows that for all $h>0$,
$$
K(s)-K(s-h) = \int_{s-h}^{s}(D(s,u) + R(s,u))du
+ \int_{s-h}^{s}(D(s-h,u) + R(s-h,u))du \,.
$$
Recall that $R(s,u)=0$ for $u > s$, hence, 
dividing by $h$ and taking $h \downarrow 0$, we thus get
by the continuity of $D$ and that of 
$R$ for $s \geq t$ that $K(\cdot)$ is differentiable,
with 
\begin{equation}\label{eq:K-alt}
\partial_s K(s)=2 D(s,s)+R(s,s)=2 D(s,s) + 1 \,,
\end{equation}
resulting by \req{eq:tempD} with \req{eqZ} for $\beta=1$.

From the \abbr{rhs} of \eqref{eq:C1-chi1} we know that
$(\partial_1 \bar \chi) (u,t) = \bar E(u,t)+ {\bf 1}_{\{u<t\}}$, 
which together with \eqref{eq:part2C} results for $s \ge t$, with 
\begin{align*} 
\bar \chi(u,t)\Phi(s,u) \mid^s_0 &=
\int_0^s  \Big[(\partial_1 \chi) (u,t) \Phi(s,u)
+ \bar \chi(u,t) (\partial_2 \Phi) (s,u) \Big] du \\
& = 
\int_0^t \Phi(s,u) du + 
\int_0^s  \bar E (u,t) \Phi(s,u) du 
+ \int_0^s \bar \chi(u,t) \big[ \Psi (s,u) 
+  \nu''(C(s,u)) R(s,u) 
\big] du
\,.
\end{align*}
It thus follows from \req{eqE} and the \abbr{lhs} of \eqref{eq:dedef1}
that for any $s \in [t,T]$,
\begin{equation}\label{eq:tempE}
\bar E(s,t)=-f'(K(s)) \bar \chi (s,t)+
\int_0^s \bar \chi (u,t)\nu''(C(s,u)) R(s,u) du 
\end{equation}
(recall that $\bar \chi(0,t)=0$). Thus, setting as in \cite[(4.4)]{BDG2},
\begin{equation}\label{eq:smE}
g(s,t):=-f'(K(s)) R (s,t)+\int_0^s R(u,t)\nu''(C(s,u)) R(s,u) du 
\end{equation}
for $s,t \in [0,T]^2$, we get \eqref{eqR} (at $\b=1$), by following 
\cite[Page 31]{BDG2} (now with \eqref{eq:tempE} and
the \abbr{rhs} of \eqref{eq:C1-chi1} instead of  
\cite[(4.3)]{BDG2} and \cite[(1.18)]{BDG2}, respectively).

\section{Critical points and the conditional model}\label{sec:eliran}


In this section, using the Kac-Rice formula, we relate the dynamics of Theorem \ref{thm-uncond} corresponding to initial conditions distributed according to $\mu_{\bsigma}^{q_o}$ around a uniformly chosen critical point $\bsigma$
from $\mathscr{C}_{N,\qs}(I_N,\,I'_N)$  to those of Theorem \ref{thm-macro} that correspond to initial conditions distributed according to $\mu_{\bn}^{q_o}$ and the  conditional disorder given  $\cpt(E,G,\bn)$. 

Setting 
\[
\omega_{N}:=\frac{2\pi^{N/2}}{\Gamma\left(N/2\right)}
\]
for the surface area of the $(N-1)$-dimensional unit sphere, we
start with the following consequence of the Kac-Rice formula (of \cite[Theorem 12.1.1]{AT}). 
\begin{prop}
	\label{prop:mean} Let $(\bsigma,\BJ) \mapsto g_\BJ(\bsigma)$
	be a continuous mapping in $\BJ$ such that $\mathbb{E} [ g_\BJ(\bsigma)^{2} ] < \infty$
	and the 
	field
	\begin{equation}
	\nonumber
	\SNqs \ni \bsigma \mapsto (H_{\BJ}(\bsigma),\,\partial_{\perp} H_{\BJ}(\bsigma),\,g_\BJ(\bsigma)) 
	\end{equation} 
	has a.s. continuous sample functions and a law invariant to rotations. We then have  for 
	$\bn=(\sqrt N\qs,0,\ldots,0)$, $\cpt(E,G,\bn)$ of \eqref{eq:cond-J},
	$\mathscr{C}_{N,\qs}(I,I')$ of \eqref{eq:CrtPts} and open intervals
    $I,\,I',\,I_0\subset\mathbb{R}$,  that  
    \begin{equation}
	\label{eq:2706-1}
	\begin{aligned}
	& \mathbb{E}\#\left\{ \bsigma\in\mathscr{C}_{N,\qs}\left(I,\,I'\right):\,g_\BJ(\bsigma)\in I_0\right\} \leq (\sqrt N\qs)^{N-1} \omega_{N} \, \varphi_{\nabla_{\rm  sp}H_{\BJ}\left(\bn\right)}(0)\\
	 & \qquad \qquad \qquad \qquad \qquad \times\int_{I\times I'}d\eta(E,G) \mathbb{E}\left\{ \left|\det\left(\nabla^{2}_{\rm  sp}H_{\BJ}\left(\bn\right)\right)\right|\mathbf{1}\Big\{ g_\BJ(\bn)\in\bar{I}_0\Big\} \, \Big| \, \cpt(E,G,\bn) \right\} , 
	\end{aligned}
	\end{equation}
	where  $\nabla_{\rm  sp}H_{\BJ}\left(\bsigma\right)=\{F_i H_{\BJ}\left(\bsigma\right) \}_{i=1}^{N-1}$ and  
	$\nabla^{2}_{\rm  sp}H_{\BJ}\left(\bsigma\right)=\{F_i F_j H_{\BJ}\left(\bsigma\right) \}_{i,j=1}^{N-1}$ 
	for an arbitrary piecewise smooth orthonormal frame field $\{F_i\}$ on the sphere,  with
$\varphi_{\nabla_{\rm  sp}H_{\BJ}\left(\bn\right)}(0)$ denoting the Gaussian density of $\nabla_{\rm  sp}H_{\BJ}\left(\bn\right)$ at $0$, while $\eta$ denotes the joint law of $(-H_{\BJ}(\bn)/N, -\partial_\perp H_{\BJ}(\bn)/\|\bn\|)$
and $\bar{I}_0$ is the closure of $I_0$.	
\end{prop}
\begin{remark} Under additional regularity conditions about $g_\BJ(\bsigma)$, the variant of the Kac-Rice
formula in \cite[Theorem 12.1.1]{AT} would have implied 
that \eqref{eq:2706-1} holds with equality and with $I_0$ instead of $\bar I_0$ on the \abbr{rhs}. 
\end{remark}
\begin{proof} Recall that in the pure case of $\nu(r)=b_m r^m$ the value of 
$\partial_\perp H_{\BJ}(\bsigma)$ is determined by $H_{\BJ}(\bsigma)$, 
whereas in the mixed case (i.e. any other $\nu(\cdot)$), the joint law of 
$(H_{\BJ}(\bsigma),\partial_\perp H_{\BJ}(\bsigma))$ is non-degenerate
(c.f. the statement of Theorem \ref{thm-macro}). 
We assume hereafter that $\nu(\cdot)$ corresponds to a mixed case, 
leaving to the reader the modifications required for handling such
degeneracy in the pure case.

Specifically, fixing $\epsilon,\,\delta>0$ define $I_{\delta}=\left\{ x+y:\,x\in I_0, |y|<\delta \right\}$
and $g_\BJ^{\epsilon}(\bsigma)=g_\BJ(\bsigma)+\epsilon Z$,
where $Z\sim N(0,1)$ is independent of $\bsigma$ and
all other random variables. Note that $(\BJ,g_\BJ^\ep(\bsigma))$ has a continuous, strictly positive density $(J,x)\mapsto p_\BJ(J)p_Z(\frac{1}{\ep}(x-g_J(\bsigma)))$, where $p_\BJ$ and $p_Z$ are the densities of $\BJ$ and $Z$. By \cite[Section 4.1]{BSZ} the vector $(H_{\BJ}\left(\bsigma\right),\partial_\perp H_{\BJ}\left(\bsigma\right),\nabla_{\rm  sp}H_{\BJ}\left(\bsigma\right),\nabla^{2}_{\rm  sp}H_{\BJ}\left(\bsigma\right))$, 
which is measurable w.r.t $\BJ$, has a non-degenerate\footnote{In the sense that the law of this array, when 
interpreting $\nabla_{\rm sp}^{2}H_{\BJ}(\bsigma)$
as the corresponding upper triangular matrix, is absolutely continuous w.r.t. the Lebesgue measure on
$\R\times\R\times\mathbb{R}^{N-1}\times\mathbb{R}^{N(N-1)/2}$.} Gaussian joint density. 
Therefore, the vector
	\[
	(H_{\BJ}\left(\bsigma\right),\partial_\perp H_{\BJ}\left(\bsigma\right),\nabla_{\rm  sp}H_{\BJ}\left(\bsigma\right),\nabla^{2}_{\rm  sp}H_{\BJ}\left(\bsigma\right),g_\BJ^\ep(\bsigma))
	\]
	has a non-degenerate, strictly positive, continuous density.
	
	Combining this with the assumptions made on $g_\BJ(\bsigma)$, the formula \eqref{potential} for the 
	Hamiltonian and its rotation-invariant law, we conclude that with $f(\bsigma)=\nabla_{\rm  sp} H_{\BJ}(\bsigma)$,
	$\nabla f(\bsigma)=\nabla_{\rm  sp}^{2}H_{\BJ}(\bsigma)$,
	$$h(\bsigma)=(-H_{\BJ}(\bsigma)/N,-\partial_\perp H_{\BJ}\left(\bsigma\right)/(\sqrt N \qs),\,g^{\epsilon}_\BJ(\bsigma))$$
	and $B=I\times I'\times I_{\delta}$ all the conditions of \cite[Theorem 12.1.1]{AT} hold, except maybe the bound in condition (g) on the modulus of continuity of $g^{\epsilon}_\BJ(\bsigma)$. However, 
in the current setting the latter condition is not necessary in order to  conclude only the upper bound
of \cite[Eq. (12.1.4)]{AT}, i.e., an inequality in the direction $\leq$, instead of an equality. 
Indeed, going through the proof of the upper bound of \cite[Theorem 12.1.1]{AT} --- which
	is based on the Euclidean version \cite[Theorem 11.2.1]{AT} --- one
	sees that the bound on the modulus of continuity of $h(\bsigma)$
	is only used when invoking \cite[Lemma 11.2.12]{AT}
	to conclude that a.s. there is
	no point $\bsigma$ such that both $f(\bsigma)=0$
	and $h(\bsigma)\in\partial B$. However, the latter fact follows here 
	directly from the definition of $g_\BJ^{\epsilon}(\bsigma)$
	and the fact the number of points such that $\nabla_{\rm sp} H_{\BJ}(\bsigma)=0$
	is a.s. finite. Thanks to the assumed rotation-invariance, 
	the upper bound of \cite[Eq. (12.1.4)]{AT} that we have just stated simplifies to 
		\begin{equation}
		\label{eq:3006-2}
		\begin{aligned}
		& \mathbb{E}\#\left\{ \bsigma 
		\in\mathscr{C}_{N,\qs}\left(I,\,I'\right):\,g^\epsilon_\BJ(\bsigma)\in I_\delta\right\} \leq (\sqrt N\qs)^{N-1}
		\omega_{N} \,
		\varphi_{\nabla_{\rm  sp}H_{\BJ}\left(\bn\right)}(0)\\
		& \qquad \qquad \qquad 
		\qquad   \mathbb{E}\left\{ \left|\det\left(\nabla^{2}_{\rm  sp}H_{\BJ}\left(\bn\right)\right)\right|\mathbf{1}\big\{h(\bn)\in B \big\} \, \Big| \, \nabla_{\rm  sp}H_{\BJ}\left(\bn\right)=0 \right\} .
		\end{aligned}
		\end{equation}	
		Recalling \cite[Section 4.1]{BSZ} that $(-H_{\BJ}(\bn)/N,-\partial_\perp H_{\BJ}\left(\bn\right)/(\sqrt N \qs))$ and $\nabla_{\rm  sp}H_{\BJ}\left(\bn\right)$ are independent, by further conditioning on the former we obtain from \eqref{eq:3006-2} that
\begin{equation}
\label{eq:3006-1}
\begin{aligned}
& \mathbb{E}\#\left\{ \bsigma 
\in\mathscr{C}_{N,\qs}\left(I,\,I'\right):\,g^\epsilon_\BJ(\bsigma)\in I_\delta\right\} \leq (\sqrt N\qs)^{N-1}
\omega_{N} \,
\varphi_{\nabla_{\rm  sp}H_{\BJ}\left(\bn\right)}(0)\\
& \qquad \qquad \qquad 
\qquad  \times\int_{I\times I'}d\eta(E,G) \mathbb{E}\left\{ \left|\det\left(\nabla^{2}_{\rm  sp}H_{\BJ}\left(\bn\right)\right)\right|\mathbf{1}\Big\{ g^\epsilon_\BJ(\bn)\in {I_\delta}\Big\} \, \Big| \, \cpt(E,G,\bn) \right\} .
\end{aligned}
\end{equation}		
Let $\Xi_{L}(\epsilon,A)$ and $\Xi_{R}(\epsilon,A)$, respectively, denote 
the left- and right-hand side of \eqref{eq:3006-1}, with general
$A\subset\mathbb{R}$ instead of $I_{\delta}$. 
Note that  $\lim_{\epsilon\to0^+}\P\{\ep Z<\delta\}=1$ and
	\[
	\mathbb{E}\#\left\{ \bsigma \in\mathscr{C}_{N,\qs}\left(I,I'\right):\,g_\BJ(\bsigma)\in I_0\right\} \leq \frac{1}{\P\{\ep Z<\delta\}}\Xi_{L}\left(\epsilon,I_{\delta}\right).
	\]
Consequently, denoting by $\bar{I}_{\delta}$ the closure of $I_{\delta}$, it follows from \eqref{eq:3006-1} that 
	\begin{align*}
	\mathbb{E}\#\left\{ \bsigma \in\mathscr{C}_{N,1}\left(I,I'\right):\,g_\BJ(\bsigma)\in I_0\right\}   \leq\lim_{\delta\to0^+}\varlimsup_{\epsilon\to0^+}\Xi_{L}\left(\epsilon,I_{\delta}\right) & \leq  \lim_{\delta\to0^+}\varlimsup_{\epsilon\to0^+}\Xi_{R}\left(\epsilon,\bar{I}_{\delta}\right)\\
	& \leq\lim_{\delta\to0^+}\Xi_{R}\left(0,\bar{I}_{\delta}\right)=\Xi_{R}\left(0,\bar{I}_0\right) 
	\end{align*}
	where the last inequality holds since $g_\BJ^{\epsilon}(\bsigma)\stackrel{a.s.}{\to}g_\BJ\left(\bsigma\right)$,
	as $\epsilon\to0$ and the indicator function of $\bar{I}_{\delta}$
	is upper semi-continuous, while the equality holds due to monotone convergence.
	This completes the proof. 
	\end{proof}
	
For $G$ large enough, the determinant on the \abbr{rhs} of
\eqref{eq:2706-1} is uniformly integrable in $N$ and the expectation of the determinant and the indicator
can be separated, yielding the following lemma.

\begin{lem}
	\label{lem:meansratio} Assume that $g_\BJ(\bsigma )$ satisfies \eqref{eq:2706-1}. Let  
	$\bn =(\sqrt N\qs,0,\ldots,0)$,
	$I_{N},I_{N}^{\prime}\in\mathbb{R}$ be a pair of open intervals as in Theorem \ref{thm-uncond} 
and $I_0\subset\R$ a fixed open interval. If it holds that 
	\begin{equation}
\lim_{N\to\infty}	\sup_{E\in I_N}\sup_{G\in I'_N}\mathbb{P}\left\{  g_\BJ(\bn)\in\bar{I}_0 \, \Big| \, \cpt(E,G,\bn) \right\}
	=0,\label{eq:2706-2}
	\end{equation}
	then in addition 
	\begin{equation}
	\label{eq:3112-01}
	\lim_{N\to\infty}\frac{\mathbb{E}\#\left\{ \bsigma \in\mathscr{C}_{N,\qs}\left(I_N,I_N'\right):\,g_\BJ(\bsigma)\in I_{0}\right\} }{\mathbb{E}\#\mathscr{C}_{N,\qs}\left(I_N,I_N'\right) }=0.
\end{equation}
\end{lem}
\begin{proof} From \eqref{eq:2706-1} we have an upper bound for the numerator of \eqref{eq:3112-01}. By an application of the Kac-Rice formula \cite[Theorem 12.1.1]{AT}, the denominator of \eqref{eq:3112-01} is equal to the \abbr{rhs} of \eqref{eq:2706-1} with the indicator omitted. Thus, to complete the proof it suffices to show that
\[
\lim_{N\to\infty}	\sup_{E\in I_N}\sup_{G\in I'_N}\frac{\mathbb{E}\left\{ \left|\det\left(\nabla^{2}_{\rm  sp}H_{\BJ}\left(\bn\right)\right)\right|\mathbf{1}\Big\{ g_\BJ(\bn)\in\bar{I}_0\Big\} \, \Big| \, \cpt(E,G,\bn) \right\}}{\mathbb{E}\left\{ \left|\det\left(\nabla^{2}_{\rm  sp}H_{\BJ}\left(\bn\right)\right)\right| \, \Big| \, \cpt(E,G,\bn) \right\}}=0.
\]
By \eqref{eq:2706-2} and the Cauchy-Schwarz inequality, it is therefore enough to show that 
\begin{equation}
\label{eq:2712-01}
\limsup_{N\to\infty} \sup_{E\in I_N}\sup_{G\in I'_N} \frac{\mathbb{E}\left\{ \left|\det\left(\nabla^{2}_{\rm  sp}H_{\BJ}\left(\bn\right)\right)\right|^2 \, \Big| \, \cpt(E,G,\bn) \right\}}{\Big(\mathbb{E}\left\{ \left|\det\left(\nabla^{2}_{\rm  sp}H_{\BJ}\left(\bn\right)\right)\right| \, \Big| \, \cpt(E,G,\bn) \right\}\Big)^2}<\infty.
\end{equation}
To this end, recall \cite[Section 4.1]{BSZ}, 
that conditional on $\cpt(E,G,\bn)$,
\[
\nabla^{2}_{\rm  sp}H_{\BJ}\left(\bn\right) \stackrel{d}{=} \sqrt{\frac{N-1}{N}\nu''(\qs^2)}\, \mathbf M+G \, \mathbf I,
\]
where $\mathbf M$ is a normalized $(N-1)$-dimensional \abbr{GOE} matrix, i.e., a real symmetric matrix with independent centered Gaussian entries (up to symmetry), such that 
\[
\E \mathbf M_{ij}^2=
\begin{cases}
2/(N-1), &i=j\\
1/(N-1), &i\neq j.
\end{cases}
\]
We have assumed that 
$\inf I_N' \to  \Gs >2\sqrt{\nu''(\qs^2)}$. Thus, 
the conditional distribution of $\nabla^{2}_{\rm  sp}H_{\BJ}\left(\bn\right)$ is identical to that of a shifted (scaled) \abbr{GOE} matrix whose eigenvalues are bounded away from $0$, uniformly in $G\in I_N'$ (and $E\in I_N$). 
Considering \cite[Corollary 23]{Subag2nd} (at $k=2$),
 this yields \eqref{eq:2712-01}, thereby completing the proof.
\end{proof}

Recall the joint law $\P^{N,q_o}_{\BJ,\bsigma}$ on $\Ca(\R^+,\R^{2N})$, of $\BB_t$ 
and the corresponding strong solution 
$\bx_t$ of \eqref{diffusion} for initial conditions $\bx_0$ distributed per $\mu_{\bsigma}^{q_o}$
(see Proposition \ref{existN}), denoting by $\E^{N,q_o}_{\BJ,\bsigma}$ the corresponding expectation.
\begin{lem}\label{lem:cont_g} 
For $\err_{N,T}(\bsigma)$ of \eqref{eq:err},  the function 
\begin{equation}\label{dfn:barg}
(\bsigma,\BJ) \mapsto \bar g_\BJ(\bsigma) :=  \E^{N,q_o}_{\BJ,\bsigma}   \big[\err_{N,T}(\bsigma) \big]
\end{equation}
satisfies the conditions of Proposition \ref{prop:mean}. Further, \eqref{eq:2706-2} then 
holds  for any open intervals $I_{N}$, $I_{N}^{\prime}$ as in Theorem \ref{thm-uncond}, 
and any fixed open interval $I_0$ such that $0\notin \bar I_0$.
\end{lem}
\begin{proof} Clearly $\bar g_\BJ \in [0,4]$, is uniformly bounded. 
The continuity of  $\bsigma \mapsto (H_{\BJ}(\bsigma),\,\partial_{\perp} H_{\BJ}(\bsigma))$ follows for example
from the representation \eqref{potential}. The invariance of the law of $
(H_{\BJ}(\bsigma),\,\partial_{\perp} H_{\BJ}(\bsigma),\, \bar g_\BJ(\bsigma))$  
under rotations follows by the argument  detailed in Remark \ref{rem:rot-sym}. 
Turning to show that  $(\bsigma,\BJ) \mapsto \bar g_\BJ(\bsigma)$ is a.s. continuous, 
upon fixing $N$ and the driving Brownian motion $\BB$ we have by the triangle inequality and 
Cauchy-Schwarz, that  
\[
| \err_{N,T}(\bsigma\,|\,\bx_0,\BJ) - \err_{N,T}(\widetilde \bsigma\,|\,\widetilde \bx_0, \widetilde \BJ) | \le 
L_1 \|\bsigma-\widetilde \bsigma\|_2 + L_2 \|\BJ-\widetilde \BJ\|
+ L_3 \sqrt{1 \wedge \| e_N \|_\infty} 
\]
where $e_N(s) := N^{-1} \|\bx_s- \widetilde \bx_s\|^2_2$,  
$L_1 := N^{-1/2} \|K_N\|_\infty^{1/2}$, $L_2 := \widetilde c \beta N^{-1/2} (1+ \| K_N\|_\infty^{m/2})$ and 
\[
L_3 := 4 
+  \|K_N\|_\infty^{1/2} + \|\widetilde K_N\|_\infty^{1/2} + \|B_N\|_\infty^{1/2}  
+ c \beta \|\widetilde \BJ\|_\infty^N (1 + \|K_N\|_\infty^r) (1 + \|\widetilde K_N\|_\infty^r) \,,
\]
for $\widetilde c = \sqrt{\nu(1)}$, the finite constants $c$, $r$ from \eqref{eq:lipH}
and with the $L_2$-norm $\|\BJ\|$ which is normalized as in \eqref{eq:norm}. 
Next, fixing $\bsigma\in\SNqs$, to jointly produce $\bar g_{\BJ}(\bsigma)$ and $\bar g_{\BJ}(\widetilde\bsigma)$ for  arbitrary $\widetilde\bsigma\in\SNqs$, let $\widetilde\BO$ be an orthogonal matrix which only rotates the space spanned by $\bsigma$ and $\widetilde\bsigma$ (i.e., $\widetilde \BO \bx=\bx$ if
$\langle\bx,\bsigma\rangle=\langle\bx,\widetilde\bsigma\rangle=0$), such that 
$\widetilde \BO \bsigma= \widetilde \bsigma$. Then,
\[
\sup_{\bx\in\SN} \| \widetilde\BO \bx - \bx\|_2 = \sup_{\bx\in\SN\cap {\rm sp}\{\bsigma,\widetilde\bsigma \}} \| \widetilde\BO\bx-\bx\|_2 =\frac{1}{\qs} \|\bsigma-\widetilde\bsigma\|_2 \,.
\]
Drawing
$\bx_0$ from law $\mu^{q_o}_{\bsigma}$, we set $\widetilde\bx_0 := \widetilde\BO \bx_0$ as the 
initial condition of laws $\mu^{q_o}_{\widetilde\bsigma}$, noting that by design 
$\|\bx_0 - \widetilde\bx_0 \|_2 \le  \| \bsigma - \widetilde \bsigma \|_2/\qs$.
 Utilizing 
this coupling and  Cauchy-Schwarz, yields that
\begin{align*}
|\bar g_{\BJ}(\bsigma) - \bar g_{\widetilde \BJ}(\widetilde \bsigma)| \le 
\int \Big\{  \|\bsigma-\widetilde \bsigma\|_2 \, \E [ L_1 | \BJ, \bx_0 ]
& +  \|\BJ-\widetilde \BJ\| \, \E [ L_2 | \BJ, \bx_0 ]
\nn \\
& + \big\{ \E [ 1 \wedge \|e_N\|_\infty | \BJ, \widetilde \BJ, \bx_0]  \,  \E [ L_3^2 | \BJ, \widetilde \BJ, \bx_0] \big\}^{1/2} 
\Big\} d\mu_{\bsigma}^{q_o} (\Bx_0) \,.
\end{align*}
From \eqref{eq:conc2} we deduce that $\int \E[L_i | \BJ,\bx_0] d\mu_{\bsigma}^{q_o} (\Bx_0)$, $i=1,2$, are a.s. finite. Further,
fixing a sequence $(\widetilde \bsigma,\widetilde \BJ) \to (\bsigma,\BJ)$, necessarily 
also $(\widetilde\bx_0, \widetilde \BJ) \to (\bx_0,\BJ)$. In view of 
\eqref{eq:K-unif-bd}, this implies a uniform, over $(\widetilde \bsigma, \widetilde \BJ)$, 
bound on $\E [ \| \widetilde K_N \|_\infty^k | \widetilde \BJ, \bx_0]$. Thereby, such uniform bound 
applies also for $\int \E [L_3^2 | \BJ, \widetilde \BJ, \bx_0] d\mu_{\bsigma}^{q_o} (\Bx_0)$, with \eqref{eq:diffeo} yielding
the a.s. continuity of $\bar g_{\BJ}(\bsigma)$.

Next, setting 
$\widetilde g_{\BJ}(\bsigma) :=\E^{N,q_o}_{\BJ,\bsigma}  [\err_{N,T}(\bsigma) {\bf 1}_{\CL_{N,M}} ]$, 
we have  in view of \eqref{eq:BDG2-217} and \eqref{dfn:barg}, that  
\[
\lim_{N \to \infty} \sup_{E \in I_N} \sup_{G \in I_N'} 
\E\Big[ | \bar g_{\BJ}(\bn) - \widetilde g_{\BJ}(\bn) | \, \big|  \cpt(E,G,\bn) \Big] =  0 \,.
\]
We thus establish \eqref{eq:2706-2} whenever $0 \notin \bar I_0$, once we show that in such a case
\begin{equation}
\lim_{N\to\infty}	\sup_{E\in I_N}\sup_{G\in I'_N}\mathbb{P}\left\{  
\tilde g_\BJ(\bn)\in\bar{I}_0 \, \Big| \, \cpt(E,G,\bn) \right\} =0.
\label{eq:2706-3}
\end{equation}
To this end, recall from our proof of Proposition 
\ref{existN}, that given $\cpt(E,G,\bn)$ one has
$\BJ = \BJ_o+\bar{\BJ}_{E,G}$ where the law of 
$\BJ_o$ is independent of $(E,G)$ and 
the only non-zero entries of $\bar{\BJ}_{E,G} = \E[\BJ\,|\,\cpt(E,G,\bn)]$ are given by \eqref{eq:bp1}. Hence, 
\begin{align*}
&\lim_{N\to\infty}	\sup_{E\in I_N}\sup_{G\in I'_N} \{ N^{-1} \|(\Bx_0, \BJ_o+\bar{\BJ}_{E,G}, \BB)-(\Bx_0, \BJ_o+\bar{\BJ}_{\Es,\Gs}, \BB)\|^2 \} \\
&=  \lim_{N\to\infty}	\sup_{E\in I_N}\sup_{G\in I'_N} \sum_{p=2}^m \; 
(b_p \qs^{p} \, \langle \bv_p, (E-\Es,G-\Gs) \rangle)^2 = 0.
\end{align*}
The Lipschitz property \eqref{eq:lippr} then implies that 
\[
\lim_{N\to\infty}	\sup_{E\in I_N}\sup_{G\in I'_N}
| \widetilde g_{\BJ_o + \bar \BJ_{E,G}}(\bn) - \widetilde g_{\BJ_o + \bar \BJ_{\Es,\Gs} }(\bn) |  = 0 \,,
\]
whereas from the $L_1$-convergence in Theorem \ref{thm-macro}  we deduce that
\begin{equation*}
	\lim_{N\to\infty}	\mathbb{P}\left\{  \widetilde g_\BJ (\bn)\in\bar{I}_0 \, \Big| \, \cpt(\Es,\Gs,\bn) \right\} =0.
\end{equation*}
Finally, note that combining the preceding two displays results with  \eqref{eq:2706-3}.
\end{proof}

\subsection*{Proof of Theorem \ref{thm-uncond}}
With $\bar g_\BJ \in [0,4]$, by Markov's inequality, for any $\delta,\epsilon >0$, 
\begin{align*}
&\E\Big\{ \sum_{\bsigma\in\mathscr{C}_{N,\qs}(I_N,\,I'_{N})} \P^{N,q_o}_{\BJ,\bsigma} (\err_{N,T}(\bsigma)>\ep)
\Big\} \leq \frac{1}{\epsilon}
\E\Big\{ \sum_{\bsigma\in\mathscr{C}_{N,\qs}(I_N,\,I'_{N})} \,\bar g_\BJ(\bsigma) \Big\} 
\\
&  \leq	\frac{\delta}{\epsilon} \mathbb{E}\#\left\{ \bsigma \in\mathscr{C}_{N,\qs}\left(I_N,I_N'\right) \right\} +
 	\frac{4}{\epsilon}\mathbb{E}\#\left\{ \bsigma \in\mathscr{C}_{N,\qs}\left(I_N,I_N'\right):\,\bar g_\BJ(\bsigma) > 
 	\delta \right\} .
\end{align*}
In addition, for any $\delta>0$ it follows from Lemmas \ref{lem:meansratio} and \ref{lem:cont_g}, that 
\begin{equation*}
	\lim_{N\to\infty}\frac{\mathbb{E}\#\left\{ \bsigma \in\mathscr{C}_{N,\qs}\left(I_N,I_N'\right):\,\bar g_\BJ(\bsigma) >  \delta \right\} }{\mathbb{E}\#\mathscr{C}_{N,\qs}\left(I_N,I_N'\right) }=0.
\end{equation*}
Combining the above and taking $N \to  \infty$ followed by $\delta \to 0$ results with \eqref{eq:uncond2}.

Next, denoting by $Y_a$ the indicator of the event that
\[
\#\mathscr{C}_{N,\qs}\left(I_{N},\,I'_{N}\right)>a\mathbb{E}\left\{ \#\mathscr{C}_{N,\qs}\left(I_{N},\,I'_{N}
\right)\right\},
\]
we have by Markov's inequality, that for any $\delta>0$,
\begin{align*}
&\P\Big\{ \frac{Y_a}{\#\mathscr{C}_{N,\qs}(I_N,\,I'_{N})}\sum_{\bsigma\in\mathscr{C}_{N,\qs}(I_N,\,I'_{N})} \P^{N,q_o}_{\BJ,\bsigma} (\err_{N,T}(\bsigma)>\ep) >\delta \Big\}
\\
&\leq \frac{1}{a\delta\mathbb{E}\left\{ \#\mathscr{C}_{N,\qs}\left( I_{N},\,I'_{N}
	\right)\right\}} \E\Big\{\sum_{\bsigma\in\mathscr{C}_{N,\qs}(I_N,\,I'_{N})} \P^{N,q_o}_{\BJ,\bsigma}(\err_{N,T}(\bsigma)>\ep)\Big\} \stackrel{N\to\infty}{\longrightarrow}0,\end{align*}
from which \eqref{eq:uncond} follows.
\qed


\section{Proof of Proposition \ref{prop-sphere}}\label{sec:hard-sphere}


As $\chi(s,t)=\int_0^t R(s,u) du$ is the limit of 
$\chi_N(s,t)$, it follows from the definition 
\eqref{integrated} of $\chi_N$ that 
\begin{equation}\label{eq:rbd}
|\int_{t_1}^{t_2} R(s,u) du |^2 \le K(s) (t_2-t_1) \,,   
\qquad 0 \leq t_1 \leq t_2  \leq s < \infty \,.
\end{equation}
Likewise, the limit $\bar C(s,t) = C(s,t)-q(s)q(t)/\qs^2$ 
of the empirical correlation functions $\bar C_N (s,t)$
must be a non-negative definite kernel
on $\reals_+ \times \reals_+$. In particular, 
$C(s,t)^2 \leq K(s) K(t)$, whereas by \eqref{eq:conc2} 
we have that $\sup_{t \geq 0} K(t) < \infty$. Unlike the special 
case considered in \cite[Proposition 1.1]{DGM}, here the functions
$(C,R)$ may take negative values. Nevertheless, we next show 
that if $(R^{(L)},C^{(L)},q^{(L)},K^{(L)})$ are solutions of the system (\ref{eqR})--(\ref{eqZ})
with $K^{(L)} (0)=1$ and potential $f_L(\cdot)$ as in (\ref{eq:fdef}) with 
$\bphi=1+2\b q_o \bv'_\star(q_o)>0$, 
then $K^{(L)}(s) \to 1$ as $L \to \infty$, uniformly over $s \ge 0$.
\begin{lem}\label{lem-bdd}
Assuming $K^{(L)}(0)=1$, there exist $B<\infty$, such that for all $L \geq B$,
\begin{equation}\label{eq:kbdd}
\sup_{s \ge 0} |K^{(L)}(s)-1|\le \frac{B}{2L} \,.
\end{equation}
\end{lem}
\begin{proof} First note that  for some $B_0=B_0(\bphi,k)$ finite and any $B \in [B_0,L]$,
\begin{equation}\label{eq:gL-form}
g_L(r): =1- 2 f_L'(r) r =  1 + 4 L r(1-r)- \bphi r^{2k} 
\end{equation}
satisfies $g_L(1-B/(2L)) \ge B/2$ and $g_L(1+B/(2L)) \le -B/2$.
Further, from \eqref{eq:K-alt} and
the \abbr{lhs} of \eqref{eq:ext}--\eqref{eq:dedef1} we see that 
\begin{equation}\label{eq:hL-form}
\partial_s K^{(L)}(s) = 1+ 2 D^{(L)}(s,s) 
 = g_L (K^{(L)}(s)) + 2 \beta {\sf A}^{(L)} (s,s) \,,
\end{equation}
where it is easy to verify that (in terms of $V(\cdot)$ and $\bar A(\cdot,\cdot)$ of \eqref{eqQ1} and \eqref{eqD}), 
\[
{\sf A} (s,t) :=   q(t) \bv'_\star (q(s)) + \beta  \bar A(s,t) +  \beta q (t) V(s) / \qs^2 =
\lim_{N \to \infty} \wE \big[ \oneN \sum_{i=1}^N G^i(\Bx_s) x^i_t \big] \,.
\]
Recall \cite[(2.15)]{BDG2}, that for some universal constant $c<\infty$ any $s$, $\bJ$ and $N$, 
\[
G_N(s) := \oneN \sum_{i=1}^N |G^i(\Bx_s)|^2 \leq c (\| \bJ \|_\infty^N)^2  [1+ K_N(s)^{m-1} ] \,.
\]
Hence, by Cauchy-Schwarz inequality and \eqref{eq:ger2} (at $k=4$), it follows that  for some 
other universal constant $\kappa<\infty$ (which is independent of $L$), 
\[
|{\sf A} (s,s)|^2  \le  \lim_{N \to \infty}  \wE[ G_N(s) K_N(s) ] \le c  
\lim_{N \to \infty} \wE \big[ (\| \bJ \|_\infty^N)^2  (K_N(s) + K_N(s)^m ) ] 
\le \kappa  ( K(s) + K(s)^m ) 
\]
(in the last step we relied also on  Corollary \ref{cor-self}).
We thus have, similarly to \cite[(2.3)]{DGM}, that for all $s$ and $L$, 
\[
|\partial_s K^{(L)}(s) - g_L(K^{(L)}(s))|^2 \le (2 \beta)^2 \kappa [ K^{(L)}(s) + K^{(L)}(s)^m ] \,.
\]
Our claim \eqref{eq:kbdd} then follows as in \cite[proof of Lemma 2.2]{DGM} (employing
the argument used there for $K^{(L)} \ge 1$, to handle now also the case $K^{(L)} \le 1$).
\end{proof}

Adapting the proof of \cite[Lemma 2.3]{DGM}, we next establish the 
equi-continuity and uniform boundedness of $(R^{(L)},C^{(L)},K^{(L)},q^{(L)})$,
which thereby admit limit points $(R,C,K,q)$.
\begin{lem}\label{lem-tight}
Set $\mu^{(L)}(s):=f'_L(K^{(L)}(s))$ and $\hat h^{(L)}(s):=\partial_s K^{(L)}(s)$. Then
$(R^{(L)},C^{(L)},q^{(L)},K^{(L)}, \mu^{(L)}, \hat h^{(L)})$ 
and their derivatives are bounded uniformly in $L \geq B$
(of Lemma \ref{lem-bdd}) and over $\Delta_T$.
\end{lem}
\begin{proof} With $|C^{(L)}(s,t)| \le \sqrt{K^{(L)}(s) K^{(L)}(t)}$ and $|q^{(L)}(s)| \le \sqrt{K^{(L)}(s)}$,
the bound \eqref{eq:kbdd} on $K^{(L)}$ results for $L \ge B$ 
with $C^{(L)},q^{(L)} \in [-2,2]$.  Further, then $|\mu^{(L)}(s)| \le 2B+|\bphi| 2^{k-1}$ (see
\cite[proof of Lemma 2.3]{DGM}). In view of \eqref{eq:hL-form}, 
\begin{equation}\label{hlid}
\hat h^{(L)}(s) = 1 - 2 K^{(L)}(s)  \mu^{(L)}(s) + 2 \b {\sf A}^{(L)} (s,s),
\end{equation}
yielding in turn the uniform boundedness of $\hat h^{(L)}(s)$. 

Since \eqref{eqR} matches  \cite[(1.7)]{DGM}, 
it follows that for the function $H_L (s,t)$ of \cite[(2.2)]{DGM},
\begin{equation}\label{eq:R-H}
R^{(L)}(s,t) = \Lambda_L(s,t) H_L(s,t) \,, \qquad \Lambda_L(s,t)=\exp(-\int_t^s \hat \mu^{(L)}(u) du ) \,, 
\qquad \forall (s,t) \in \Delta_T \,.
\end{equation}
Recall that $\nu''(\cdot)$ is uniformly bounded on the compact $[-2,2]$,  hence
$H_L$ of \cite[(2.2)]{DGM} is uniformly bounded over $\Delta_T$ and $L \ge B$,
and thereby the same applies for $R^{(L)}$.

Upon replacing $f_L'(K^{(L)}(s))$ by $\mu^{(L)}(s)$ in \eqref{eqR}--\eqref{eqZ}, we deduce from our
preceding statements the claimed uniform boundedness for $\partial_s q^{(L)}$, $\partial_s K^{(L)}$,
$\partial_s C^{(L)}(s,t)$ and $\partial_s R^{(L)}(s,t)$, when $s \ge t$. Following  
\cite[proof of Lemma 2.3]{DGM}, the same applies for $\partial_t H_L(s,t)$
and consequently for $\partial_t R^{(L)}(s,t)$. Further, from \eqref{eq:Phi}, 
such uniform boundedness applies to $\bar D^{(L)}(s,t)$ of \eqref{eq:tempD},
hence by \eqref{eq:part2C} also to  
\[
\partial_t C^{(L)}(s,t) = \bar D^{(L)}(s,t) + R^{(L)}(s,t) + (q^{(L)}(s)/\qs^2)  \partial_t q^{(L)}(t) \,.
\]
Next, $\partial_s \hat h^{(L)}(s) = - 4L \hat h^{(L)}(s) + \kappa_L(s)$ for 
\[
\kappa_L(s) := (g'_L(K^{(L)}(s)) + 4L) \hat h^{(L)}(s) + 2\b \partial_s {\sf A}^{(L)}(s,s) \,.
\]
In view of \eqref{eq:gL-form} we have that
$|g_L'(r) + 4L| \le 4 B + k |\bphi| 2^{2k}$ whenever $|r -1| \le B/(2L) \le 1/2$,
while $|\partial_s {\sf A}^{(L)}(s,s)|$ is bounded uniformly in $L \ge B$ 
and $s \le T$ (by \eqref{eqZ} and 
the uniform boundedness of $(R^{(L)},C^{(L)},q^{(L)})$ and  
$\partial_s (R^{(L)},C^{(L)},q^{(L)})$). In particular,  
$\alpha(T) := \sup \{ |\kappa_L(u)| : L \geq B, u \leq T\}$ is finite.
Next, recall \eqref{eq:hL-form} that $K^{(L)}(0) =1$ and $g_L(1)=1-\bphi$ (see \eqref{eq:gL-form}), 
resulting for our choice of $\bphi = 1 +2 \b q_o  \bv'_\star (q_o) = 1 +  2\b {\sf A}^{(L)}(0,0)$ with 
$\hat h^{(L)}(0)=0$. Thus, 
\[
\hat h^{(L)}(s) = \int_0^s e^{-4 L (s-u)} \kappa_L(u) du 
\]
yielding that 
\begin{equation}\label{eq:hbd}
\sup_{s \in [0,T]} |\hat h^{(L)}(s)|\le \frac{\alpha(T)}{4 L} \,, \qquad \forall L \ge B \,,
\end{equation}
from which the uniform boundedness of $|\partial_s \hat h^{(L)}|$ follows.
Finally, by definition, for our choice of $f_L(\cdot)$, 
\[
\partial_s \mu^{(L)}(s) = f''_L(K^{(L)}(s)) \hat h^{(L)}(s) 
= \big[ 2 L  + \frac{(2k-1) \bphi}{2} K^{(L)}(s)^{2k-2} \big]  \hat h^{(L)}(s) \,,
\] 
which by (\ref{eq:hbd}) provides 
the uniform boundedness of $|\partial_s \mu^{(L)}|$.     
\end{proof} 

\medskip
\noindent{\bf Proof of Proposition \ref{prop-sphere}.}
Recall Lemma \ref{lem-tight} that  
$(R^{(L)},C^{(L)},q^{(L)},K^{(L)},\mu^{(L)},\hat h^{(L)})$, $L \geq B$ are equi-continuous 
and uniformly bounded on $\Delta_T$. Hence, by the Arzela-Ascoli theorem, 
this collection has a limit point 
$(C,R,q,K,\mu,\hat h)$ with respect to uniform convergence on $\Delta_T$.

By Lemma \ref{lem-bdd} we know that the limit $K(s) \equiv 1$ on $[0,T]$, whereas 
by \eqref{eq:hbd} we have that $\hat h(s) \equiv 0$ on $[0,T]$. Considering 
$L_n \to \infty$ for which 
$(R^{(L_n)},C^{(L_n)},q^{(L_n)},K^{(L_n)}, \mu^{(L_n)}, \hat h^{(L_n)})$ converges to 
$(R,C,q,K,\mu,\hat h)$ we find that the latter must satisfy \eqref{eqZs}. 
Further, since $R^{(L)}(t,t)=1$, $C^{(L)}(t,t)=K^{(L)}(t)$ and $q^{(L)}(0)=q_o$, 
integrating \eqref{eqR}--\eqref{eqq} we see that 
$R^{(L)}(s,t)= 1 + \int_t^s {\sf A}_R^{(L)}(\theta,t) d\theta$,
$C^{(L)}(s,t)=K^{(L)}(t)+\int_t^s {\sf A}_C^{(L)}(\theta,t) d\theta$
and $q^{(L)}(s) = q_o + \int_0^s {\bf A}_q^{(L)}(\theta) d\theta$, where 
\begin{align*}
{\sf A}_R^{(L)}(\theta,t) & := -  \mu^{(L)}(\theta) R^{(L)}(\theta,t) + \b^2 \int_t^\theta
 R^{(L)}(u,t) R^{(L)}(\theta,u) \nu''(C^{(L)}(\theta,u)) du , \\ 
{\sf A}_C^{(L)}(\theta,t) &: = - \mu^{(L)}(\theta) C^{(L)}(\theta,t) + \b {\sf A}^{(L)}(\theta,t) \,, \\
{\sf A}_q^{(L)} (\theta) & := -\mu^{(L)}(\theta) q^{(L)} (\theta) +  \b^2 \int_0^\theta R^{(L)}(\theta,u) 
\Big[ q^{(L)}(u) \nu''(C^{(L)}(\theta,u)) 
- \frac{\qs^2 \nu'(q^{(L)}(u)) \nu''(q^{(L)}(\theta))}{\nu'(\qs^2)} \Big] du \\
& \qquad + 
\b \qs^2 \bv'_\star (q^{(L)}(\theta)) \,.
\end{align*}
Note that $ \mu^{(L_n)} (s) \to \mu(s)$,  while
${\sf A}_R^{(L_n)}(s,t)$, ${\sf A}_C^{(L_n)}(s,t)$ and ${\sf A}^{(L_n)}_q(s,t)$ 
converge, uniformly on $\Delta_T$,
to the right-hand-sides of \eqref{eqRs}--\eqref{eqqs}, respectively. We thus deduce that
for each limit point $(C,R,q,\mu)$, the functions $C(s,t)$, $R(s,t)$ and $q(s)$
are differentiable in $s$ on $\Delta_T$ and all 
limit points satisfy \eqref{eqRs}--\eqref{eqZs}.
Further, $C^{(L)}(s,t)$ are 
non-negative definite kernels with $C^{(L)}(t,t) \to 1$ as $L \to \infty$.
Consequently, each of their limit points corresponds 
to a $[-1,1]$-valued
non-negative kernel on $[0,T]^2$. Similarly, as
$R^{(L)}(t,t)=1$ and $R^{(L)}(s,t)$ satisfy \eqref{eq:rbd},  
both constraints apply for any limit point $R(s,t)$. We further 
extend $R(\cdot,\cdot)$ to a function on $[0,T]^2$ by setting $R(s,t)=R^{(L)}(s,t)=0$ whenever $s<t$. 

With $\wH(\cdot)$ a continuous functional of $(R,C,q)$, it remains only to verify
that the system of equations \eqref{eqRs}--\eqref{eqZs}
with $q(0)=q_o$, $C(s,t)=C(t,s)$, $C(t,t)=R(t,t)=1$ and $R(s,t)=0$ for $s<t$,
admits at most one bounded solution $(R,C,q)$ on $[0,T]^2$. To this end 
consider the difference between the integrated form of \eqref{eqRs}--\eqref{eqqs} 
for two such solutions $(C,R,q)$ and $(\bar C,\bar R,\bar q)$. 
Since $\nu'',\nu',\bv'_\star$ are locally Lipschitz, we get as in \cite[proof of Prop. 1.1]{DGM}, 
that $\Delta R=|R-\bar R|$,
and $\Delta C =|C-\bar C| + |q(s)-\bar q(s)| + |q(t)-\bar q(t)|$ satisfy on $\Delta_T$
\begin{align*}
\Delta R(s,t)&\le \kappa_1 
\big\{ \int_t^s [ \Delta R(v,t) + \Delta C(v,t) ] dv +\int_t^s h(v) dv \big\} \,,
\\
\Delta C(s,t)&\le \kappa_1 \big[
\int_t^s  \Delta C(v,t)  dv +h(t)+\int_t^s h(v) dv \big] \,,
\end{align*}
where $h(v) :=  \int_0^v [\Delta R(v,u) +\Delta C(v,u)] du$ and
$\kappa_1 < \infty$ depends on $T$, $\b$, $\nu(\cdot),\bv'_\star(\cdot)$ and
the maximum of $|R|$, $|C|$, $|q|$, $|\bar R|$, $|\bar C|$ and $|\bar q|$ on $[0,T]^2$.
Integrating these inequalities over $t \in [0,s]$, 
since $\Delta R(v,u)=0$ for $u \geq v$ 
and $\Delta C(v,u)=\Delta C(u,v)$, we find similarly to \cite[Page 860]{DGM}, that 
$$
0 \leq h(s) \leq 2 \kappa_2 \int_0^s h(v) dv \,, \qquad h(0)=0 \,,
$$
for some finite constant $\kappa_2$ (of the same type of 
dependence as $\kappa_1$). By Gronwall's lemma we deduce that
$h \equiv 0$ on $[0,T]$, hence $\Delta R(s,t) = \Delta C(s,t) = 0$ for a.e. $(s,t) \in \Delta_T$. 
By the continuity and symmetry of these functions,  
the same applies for all $(s,t) \in [0,T]^2$, yielding the stated uniqueness and thereby completing 
the proof.
\hfill \qed

\begin{section}{Proof of Proposition \ref{prop:fdt}}
\label{sec:fdt}

Consider the convex set $\Aa^+$ of \emph{bounded} continuous functions $(R,C,q) \in 
\CC_b(\Delta_\infty) \ts \CC_b (\R_+^2) \ts \CC_b (\R_+)$ such that $C(s,t)=C(t,s)$, 
$R(s,s)=C(s,s)=1$ and $q(0)=q_o$, equipped with the norm
\begin{equation}\label{eq:unorm}
\| (R,C,q)\|= \sup_{(s,t) \in \Delta_\infty} |R(s,t)|+ \sup_{(s,t) \in \Delta_\infty} |C(s,t)| + \sup_{s \ge 0} |q(s)| \,.
\end{equation}
Analogously to \cite[(4.1)-(4.3)]{DGM}, we recall from Proposition \ref{prop-sphere} that 
$(R,C,q)$ of \eqref{eqRs}-\eqref{eqZs} is the unique fixed point of the mapping 
$\Psi: (R,C,q) \mapsto (\widetilde R, \widetilde C,\widetilde q)$ on $\Aa^+$ 
such that for any $(s,t) \in \Delta_\infty$
\begin{align}
\partial_s \widetilde R(s,t) =
& - \mu(s) \widetilde R(s,t) + \b^2 \int_t^s
\widetilde R(u,t) \widetilde R(s,u) \nu''(C(s,u)) du ,\label{eqRs-Psi}\\
\partial_s \widetilde C(s,t)= &  - \mu(s) \widetilde C(s,t) +  \b^2 I_1 (s,t) + \b^2 I_2(s,t),
 \label{eqCs-Psi}\\
\partial_s \widetilde q(s) = & -\mu(s) \widetilde q(s) + \b^2 I_3(s),
\label{eqqs-Psi}
\end{align}
with $\mu(s) = \mu_{(R,C,q)} (s) = \frac{1}{2} + \b^2 I_0(s)$ of \eqref{eqZs} and
\begin{align*}
 I_0(t) &:= \int_{-t}^0 R(t,t+u) \Big[ \psi(C(t,t+u)) -
\frac{\psi(q(t)) \nu'(q(t+u))}{\nu'(\qs^2)} \Big] \, du + \b ^{-1} q(t) \bv'_\star(q(t)) \,,  \\
 I_1(t+v,t)&:= \int_{-t}^v R(t+v,t+u) \Big[ \nu''(C(t+v,t+u)) C(t+u,t) 
 -  \frac{q(t) \nu'(q(t+u)) \nu''(q(t+v))}{\nu'(\qs^2)} \Big] \, du \,, \\
 I_2(t+v,t)&:=\int_{-t}^0 R(t,t+u) \Big[\nu'(C(t+v,t+u)) - \frac{\nu'(q(t+v)) \nu'(q(t+u))}{\nu'(\qs^2)} \Big] \, du 
+ \b^{-1} q(t) \bv'_\star (q(t+v))   \,, \\
I_3(t) & := 
\int_{-t}^0 R(t,t+u) 
\Big[ q(t+u) \nu''(C(t,t+u)) 
- \frac{\qs^2 \nu'(q(t+u)) \nu''(q(t))}{\nu'(\qs^2)} \Big] du + 
\b^{-1} \qs^2 \bv'_\star (q(t)) \,.
\end{align*}
We next characterize the possible limits
$(R_{\fdt},C_{\fdt})$ in \eqref{dfn:cS-space} in case we have for 
$\beta >0$, $|q_o| \le \qs$ that:
\newline
{\bf (H1).} There exists a closed set $\Aa  \subset \{(R,C,q) \in \Aa^+ : \|(R,C,q)\| \le \rho\}$, 
where the functions $\{R(t+\cdot,t), t \ge T_0\}$ are uniformly integrable \abbr{wrt} 
Lebesgue measure on $\R$ and 
\begin{equation}\label{eq:mu-pos}
\liminf_{v \to -\infty}  \inf_{t \ge -v} \Big\{ \frac{1}{|v|} \int_{v}^0 \mu_{(R,C,q)} (t+u) du \Big\} > 0 \,.
\end{equation}
{\bf (H2).} $\Psi$ is a contraction on $(\Aa,\|\cdot\|)$ and 
the subset $\Sa$ of $\Aa$ with property \eqref{dfn:cS-space} for some $|\alpha| \le 1$,
is non-empty.
\begin{prop}\label{FDT1}
Assuming {\bf (H1)-(H2)},
the solution $(R,C,q)$ of \eqref{eqRs}--\eqref{eqZs} 
is the unique fixed point of $\Psi$ in $\Sa$ and 
$(R_{\fdt},C_{\fdt})$ of \eqref{dfn:cS-space} are a solution in 
$\widetilde \Ba:=\{(R,C) \in \CB(\R_+) \ts \CB(\R) : C(0)=R(0)=1$, $C(\tau)=C(-\tau) \}$
of 
\cite[(4.15)-(4.16)]{DGM},  with $\mu$ as in \cite[(4.17)]{DGM}, but now for 
$(\IJ,\alpha)$ satisfying  \eqref{eq:Qa-const} and \eqref{eqIFDT}.
\end{prop}
\begin{proof} We first verify that for the given $\b$ and $q_o$,
any $S=(R,C,q) \in \Sa$ results with $\Psi(S) \in \Sa$.  To this end, proceeding similarly to 
\cite[proof of (4.7)]{DGM}, we have for $(R,C,q) \in \Sa$ that 
as $t \to \infty$ the bounded integrands in the formulas for $I_i(\cdot,\cdot)$, $i=0,1,2,3$,
converge pointwise (per fixed $u=v-\theta$),
to the corresponding expression for $(R_{\fdt},C_{\fdt},\qinf)$.
Further, thanks to the uniform integrability of the collection
$\{ R(t+\cdot,t), t \ge T_0 \}$ (when $(R,C,q) \in \Aa$, see {\bf (H1)}), 
the contributions of the integrals over $[-t,-m]$ decay to zero as $m \to \infty$, uniformly in $t$.
Thus, applying the bounded convergence theorem for the integrals 
over $[-m,v]$, then taking $m \to \infty$, we deduce that
for each fixed $v \geq 0$, in analogy with \cite[(4.11)-(4.12)]{DGM},
\begin{align}
\hh{I}_0 &:= \lim_{t\ra\infty} I_0(t) =
\int_0^\infty R_{\fdt}(\theta) \Big[ \psi(C_{\fdt}(\theta))  -
\frac{\psi(\qinf) \nu'(\qinf)}{\nu'(\qs^2)} \Big] d\theta + \b ^{-1} \qinf \bv'_\star(\qinf) \,,
\label{eqI0l}
\\
\hh{I}_1(v) &:=
\lim_{t\ra\infty} I_1(t+v,t) 
= \int_0^\infty
 R_{\fdt} (\theta)  \Big[ \nu''(C_{\fdt} (\theta)) C_{\fdt} (v-\theta) 
 -  \frac{\qinf \nu''(\qinf) \nu'(\qinf)}{\nu'(\qs^2)} \Big] \, d\theta \,,
\label{eqI1l}
\\
\hh{I}_2(v) &:=
\lim_{t\ra\infty} I_2(t+v,t) 
=\int_v^\infty R_{\fdt} (\theta-v) \Big[ \nu'(C_{\fdt} (\theta)) 
 - \frac{\nu'(\qinf)^2}{\nu'(\qs^2)} \Big] \, d\theta
+ \b^{-1} \qinf \bv'_\star (\qinf)   \,,
\label{eqI2l}
\\
\hh{I}_3 &:= \lim_{t\ra\infty} I_3 (t) =  \int_0^\infty R_{\fdt} (\theta) 
\Big[ \qinf \nu''(C_{\fdt}(\theta)) 
- \frac{\qs^2 \nu'(\qinf) \nu''(\qinf)}{\nu'(\qs^2)} \Big] d\theta + 
\b^{-1} \qs^2 \bv'_\star (\qinf) \,.
\label{eqI3l}
\end{align} 
Using the notation $I_i(\cdot,t):=I_i(t)$ for $i=0,3$, we further know  by the preceding  that 
$\sup_{t,v \ge 0} \{  | I_i(t+v,t) | \} < \infty$ for $0 \le i \le 3$, yielding in particular the finiteness of   
$\sup_{t \ge 0} \sup_{v \in [0,\tau]} \{ \L(t+\tau,t+v) \}$ for $\L(\cdot,\cdot)$ of \cite[(4.8)]{DGM}.
Recall from \eqref{eqRs-Psi} that $\widetilde R(s,t)=\L(s,t) \widetilde H(s,t)$
for $\widetilde H(\cdot,\cdot)$ of \cite[(4.9)]{DGM}, hence by bounded convergence
(as in \cite{DGM}), we have for any $\tau \geq 0$,
\begin{align}
\hh{\L}(\tau-v) &:= \lim_{t \to \infty} 
\L(t+\tau,t+v) = e^{-(\tau-v)\mu} \,,  \qquad \qquad \forall v \in [0,\tau] \,, \label{La-hat}\\
\widetilde C_{\fdt} (\tau) &:= \lim_{t \to \infty} 
\widetilde C(t+\tau,t)= \hh{\L}(\tau) +
\b^2 \int_0^{\tau} \hh{\L}(\tau-v) \hh{I}_1(v) dv
+ \b^2 \int_0^\tau \hh{\L}(\tau-v) \hh{I}_2(v) dv \,, \label{Cfdt-Psi}\\
\widetilde H_{\fdt} (\tau) &:= 
\lim_{t \to \infty} \tilde H(t+\tau,t) =  1+ \sum_{n\ge 1}\beta^{2n} \sum_{\s\in\NC_n}
\int_{0\le \theta_1\le \cdots\le \theta_{2n}\le \tau}
\prod_{i\in \cro(\s)} \nu''(C_{\fdt}(\theta_i-\theta_{\sigma(i)}))
\prod_{j=1}^{2n} d\theta_j \,, \label{Hfdt-Psi}\\
\widetilde R_{\fdt} (\tau) &:= \lim_{t \to \infty} 
\widetilde R(t+\tau,t)= \hh{\L}(\tau) \widetilde H_{\fdt} (\tau) \,.\label{Rfdt-Psi}
\end{align}
Unlike \cite{DGM}, here in principle $I_i(s,t)$ might take negative values. However, 
thanks to \eqref{eq:mu-pos},
\begin{equation}\label{eq:mu-inf}
\mu := \lim_{t \to \infty} \{ \mu(t) \} = \frac{1}{2} + \b^2 \hh{I}_0 >  0 \,.
\end{equation}
With $\L(t,x) = \L(t,0) \L(0,x)$,  also
\begin{equation*}
\widetilde q(t) = \L (t,0) q_0 + \b^2 \int_{-t}^0 \L(t,t+v) I_3(t+v) d v\,,
\end{equation*}
where by \eqref{eq:mu-pos} we have that $\L(t,0) \to 0$ and
the integral over $[-t,-m]$ decays to zero as $m \to \infty$, uniformly in $t$. Applying 
bounded convergence for the integral over $[-m,0]$, then taking $m \to \infty$, we see that 
\begin{equation}\label{q-inf-Psi}
\wqinf := \lim_{t \to \infty} \widetilde q(t) = \b^2 \hh{I}_3 \int_0^\infty \hh{\L}(v) dv 
= \frac{\b^2}{\mu} \hh{I}_3 \,.
\end{equation}
Thus, $\Psi(\Sa) \subset \Sa$, with $\Psi$ inducing on $\Sa$ the mapping  
$\Psi_{\fdt}:(R_{\fdt},C_{\fdt},\alpha) \to (\widetilde R_{\fdt},\widetilde C_{\fdt},\widetilde \alpha)$
given by \eqref{La-hat}--\eqref{q-inf-Psi}, for $\hh{I}_i$, $i=0,1,2,3$ as in the \abbr{rhs} of 
\eqref{eqI0l}-\eqref{eqI3l}. In particular, $\widetilde R_{\fdt}$ and $\widetilde C_{\fdt}$ are differentiable on $\R_+$ 
and satisfy \cite[(4.23)-(4.24)]{DGM} for $\widetilde{R}_{\fdt}(0)=\widetilde{C}_{\fdt}(0)=1$ and 
the preceding values of $\hh{I}_i$, $i=0,1,2$.

Next, recall {\bf (H2)} that 
$\Psi$ is a contraction on $(\Aa,\|\cdot\|)$, hence also on its non-empty subset $\Sa$. 
Thus, starting at any $S^{(0)}=(R^{(0)},C^{(0)},q^{(0)}) \in \Sa$ yields a Cauchy
sequence $S^{(k)}=\Psi(S^{(k-1)}) \in \Sa$, $k=1,\ldots$ for the norm $\|\cdot\|$ 
of \eqref{eq:unorm}, with $S^{(k)} \to S^{(\infty)}$ in the closed subset $\Aa$ of 
$(\Aa^+,\|\cdot\|)$. Further, fixing $\tau \geq 0$,
with $|(x,y,z)|:=|x|+|y|+|z|$, since $S^{(k)} \in \Sa$ we have that 
$$
\lim_{T \to \infty} \sup_{t,t' \geq T} 
|S^{(\infty)}(t+\tau,t)-S^{(\infty)}(t'+\tau,t')| \leq 
2 \|S^{(\infty)}-S^{(k)}\| \,.
$$
Taking $k \to \infty$ we deduce that $\{t \mapsto S^{(\infty)}(t+\tau,t)\}$ 
is a Cauchy mapping from $\R_+$ to $|(x,y,z)| \le \rho$, hence          
$S^{(\infty)}(t+\tau,t)$ converges as $t \to \infty$. This applies for any 
$\tau \geq 0$, hence $S^{(\infty)} \in \Sa$ is the unique fixed point 
of the contraction $\Psi$ on 
$(\Sa,\|\cdot\|)$. In particular, as shown in \eqref{eqI0l} this implies also that $\mu(t) \to \mu$ 
of \eqref{eq:mu-inf}. Recall that any fixed point of $\Psi$ must satisfy \eqref{eqRs}-\eqref{eqZs},
hence the unique solution of the latter equations in $\Aa^+$ 
must coincide with $S^{(\infty)}$ and in particular be in $\Sa$. 
As noted before, this yields the existence of $(R_{\fdt},C_{\fdt}) \in  \widetilde \Ba$ 
which for a suitable choice of $\alpha$ forms a fixed point of $\Psi_{\fdt}$. 
Considering \eqref{eq:mu-inf} and
\cite[(4.24)]{DGM} for $\widehat I_i(\cdot)$, $i=0,1,2$, of \eqref{eqI0l}-\eqref{eqI2l}
we arrive at \cite[(4.15)-(4.17)]{DGM}, now with the possibly non-zero $\IJ$ as given in \eqref{eqIFDT}.
Finally, in view of 
\eqref{q-inf-Psi} and \eqref{eqI3l}, 
our constraint \eqref{eq:Qa-const} on $\alpha$  is merely the 
fixed point condition $\widetilde \alpha =\alpha$.  
\end{proof}

\noindent
\emph{Proof of  Proposition \ref{prop:fdt}:} We start with our second claim, where we allow for
arbitrary $\b>0$, but assume that the unique fixed point  $(R,C,q)$ of $\Psi$ in $\Aa^+$ satisfies 
\eqref{dfn:cS-space} as well as the properties in {\bf (H1)}. While proving 
Proposition \ref{FDT1} we have showed that it results with \eqref{eqI0l}-\eqref{eqI3l}, 
and thereby with $\mu(t) \to \mu$ for $(R_{\fdt},C_{\fdt},\mu)$ a 
solution of \cite[(4.15)-(4.17)]{DGM} on $\widetilde \Ba$ with
$(\IJ,\alpha)$ satisfying \eqref{eq:Qa-const}-\eqref{eqIFDT}. To complete 
our claim, note that \eqref{eq:gamma-cons} amounts to \cite[(1.21)]{DGM}
holding for $\phi(\cdot)$ of \eqref{dfn:phi-FDT} and $b=1/2$, so by 
\cite[Proposition 5.1]{DGM} we have that $
(R_{\fdt},C_{\fdt},\mu)=
(-2D',D,\phi(1))$ 
satisfies \cite[(4.15)-(4.17)]{DGM} for $\IJ$ of \eqref{eq:I-gamma} and
the unique $D(\cdot)$ of \eqref{FDTDb}.

Turning to our first claim, note that $\alpha=0$ 
satisfies \eqref{eq:Qa-const} for any value of $\b$. Further, from 
\cite[(4.17)]{DGM} and \eqref{eqIFDT} we see that $\mu \to \frac{1}{2}$ when 
$\b \downarrow 0$ and since the finite polynomials $\nu'(x)$ and $\bv'_\star(x)$ 
are both zero at $x=0$, it is easy to check that  
$\alpha=0$ is the \emph{only} solution of \eqref{eq:Qa-const} for small $\b>0$.
In case $q_o=0$ it is also shown in \cite[Section 4]{DGM} that for small $\b$
our assumptions {\bf (H1)-(H2)} hold for $\Aa$ consisting of 
$e^{\delta |s-t|} (R,C,q) (s,t) \in [0,\rho (r|s-t|+1)^{-3/2}] \ts [0,c] \ts \{0\}$ and suitably chosen
parameters $\delta,r,\rho,c$. Leaving the details to the reader, such analysis can be 
extended to yield {\bf (H1)-(H2)}  for any  $|q_0| \le q_\star$ and 
$\beta \in [0,\beta_1)$,
again with $\alpha=0$, but  now  for 
\[
\Aa := \Big\{ (R,C,q) \in \Aa^+ : |R(t+\tau,t)| \le \rho (r \tau+1)^{-3/2} e^{-\delta \tau}, 
|C(t+\tau,t)| \le c e^{-\delta \tau}, |q(\tau)| \le \kappa e^{-\eta \tau}  \Big\}
\]
and certain positive $\delta,r,\rho,c,\kappa,\eta$ (that may depend on 
$\b$ and $q_o$). The unique fixed point of $\Psi$ in $\Sa$ one gets from 
Proposition \ref{FDT1} must then have $\IJ=\alpha=0$, with 
$(R_{\fdt},C_{\fdt},\mu)$ 
the unique solution of \cite[(4.15)-(4.17)]{DGM} within a subset of
$\widetilde \Ba$ analogous to $\Ba(\delta,r,\rho,c)$ of \cite[Proposition 4.2]{DGM},
except for allowing here possible negative values of $R_{\fdt}$ or $C_{\fdt}$.  Recall 
that  for all $\b$ up to $\b_c$ of \cite[(1.23)]{DGM} both
\eqref{eq:I-gamma} and \eqref{eq:gamma-cons}  hold for 
$\gamma=1/2$ and $\IJ=D_\infty=0$. Thus, as we have seen before,
for such $\b$ the unique solution of \cite[(4.15)-(4.17)]{DGM} alluded to above
corresponds to $C_{\fdt}(\cdot)=D(\cdot)$ for the $[0,1]$-valued solution of 
\eqref{FDTDb}.
\hfill\qed

\end{section}

\end{document}